\documentclass[twoside,11pt]{article}

\usepackage[colorlinks,bookmarksopen,bookmarksnumbered,citecolor=black,urlcolor=red]{hyperref}
\usepackage{setspace}
\usepackage[margin=1in,right=1in,top=1in,bottom=1in]{geometry} 
\usepackage{amsmath,amsthm,amssymb,scrextend, enumerate, mathrsfs}
\usepackage{fancyhdr}
\usepackage{tikz}
\usepackage{pgfplots}
\usepackage{natbib}
\usepackage{amsmath, threeparttable}
\usepackage[labelfont=bf]{caption}
\usepackage[mathscr]{euscript}
 \let\mathscr\relax
\usepackage[scr]{rsfso}
\usepackage{amsthm}
\usepackage{amssymb}
\usepackage{multicol}
\usepackage{multirow}
\usepackage{booktabs}
\usepackage{lscape}

\newtheorem{theorem}{Theorem}
\newtheorem{corollary}{Corollary}
\newtheorem{proposition}{Proposition}
\newtheorem{lemma}{Lemma}

\newtheorem{assumption}{Assumption}
\newtheorem{example}{Example}
\newtheorem{remark}{Remark}

\newcommand{\E}{\mathbb{E}}
\newcommand{\mc}{\mathcal}

\renewcommand{\subsectionmark}[1]{}

\begin{document}

\title{Source-Condition Analysis of Kernel Adversarial Estimators}

\author{
Antonio Olivas-Martinez\textsuperscript{1} and Andrea Rotnitzky\textsuperscript{1} \\
\textsuperscript{1}Department of Biostatistics, University of Washington
}

\maketitle

\begin{abstract}
In many applications, the target parameter depends on a nuisance function defined by a conditional moment restriction, whose estimation often leads to an ill-posed inverse problem. Classical approaches, such as sieve-based GMM, approximate the restriction using a fixed set of test functions and may fail to capture important aspects of the solution. Adversarial estimators address this limitation by framing estimation as a game between an estimator and an adaptive critic. We study the class of Regularized Adversarial Stabilized (RAS) estimators that employ reproducing kernel Hilbert spaces (RKHSs) for both estimation and testing, with regularization via the RKHS norm. Our first contribution is a novel analysis that establishes finite-sample bounds for both the weak error and the root mean squared error (RMSE) of these estimators under interpretable source conditions, in contrast to existing results. Our second contribution is a detailed comparison of the assumptions underlying this RKHS-norm-regularized approach with those required for (i) RAS estimators using $\mathcal{L}^2$ penalties, and (ii) recently proposed, computationally stable Kernel Maximal Moment estimators.

\end{abstract}

\section{Introduction}\label{sec:intro}

In many scientific and policy-relevant applications, target parameters depend on nuisance functions identified through conditional moment restrictions. Estimating such functions often leads to ill-posed inverse problems, where small perturbations in intermediate quantities can cause large estimation errors. These settings arise in causal inference with unmeasured confounding \citep{santos2011instrumental, severini2012efficiency, escanciano2013identification, freyberger2015identification, miao2018identifying, tchetgen2020introduction, kallus2021causal, ghassami2022minimax, bennett2022inference, bennett2023source, cui2024semiparametric} and in analyses of missing-not-at-random data \citep{miao2016varieties, li2023non, miao2024identification}.

Classical approaches---such as series or sieve-based Generalized Method of Moments (GMM)---replace the infinite set of implied moment equations with a finite system, obtained by testing against a pre-chosen set of functions. While computationally straightforward, these methods can miss key structural components of the solution when they lie in directions poorly represented by the chosen test space.

Adversarial estimators \citep{bennett2019deep, muandet2020dual, muandet2020kernel, dikkala2020minimax, liao2020provably, kallus2021causal, ghassami2022minimax, mastouri2021proximal, bennett2022inference, zhang2023instrumental, bennett2023minimax, bennett2023source, park2024proximal} address this limitation by framing estimation as a two-player zero-sum game: a candidate solution is pitted against an adaptive critic that searches for the most challenging directions of violation. Within this framework, Regularized Adversarial Stabilized (RAS) estimators apply Tikhonov regularization to stabilize estimation. Of particular interest are variants that use reproducing kernel Hilbert spaces (RKHSs) for both estimation and testing. When regularized via the RKHS norm---hereafter referred to as KRAS estimators---they admit closed-form solutions without the numerical instability suffered by their $\mathcal{L}^2(\mathbb{P}_n)$--penalized RAS counterparts. The latter, while also closed-form, require solving nearly ill-conditioned systems and can be computationally fragile.

Theoretical analysis of $\mathcal{L}^2(\mathbb{P}_n)$--penalized RAS estimators, hereafter referred to as LRAS estimators, has yielded sharp finite-sample bounds under interpretable source conditions that link solution smoothness to operator smoothing properties. In contrast, existing theory for KRAS estimators is limited: available bounds rely on abstract ill-posedness measures that are hard to interpret and impose uniqueness of the solution.

This paper has two goals of equal importance. First, we develop a new theoretical framework for KRAS estimators that yields finite-sample bounds for both weak error and RMSE under transparent source conditions, without requiring uniqueness of the inverse problem solution. These conditions describe the smoothness of the minimal-norm solution relative to the smoothing effect of the composite operator formed by the conditional expectation and the RKHS kernel integral operator. Second, we place KRAS estimators in the broader landscape by systematically comparing their assumptions underpinning their error bounds, and the convergence rates that each deliver, with those of LRAS estimators and with those of recently proposed, computationally stable Kernel Maximal Moment (KMMR) estimators. This comparative perspective situates our results within the spectrum of available methods, clarifying the trade-offs each entails and the contexts in which their respective assumptions are most plausible. 

The rest of the paper is organized as follows. Section~\ref{sec:targer_param} introduces the general class of target parameters and formulates their identification through conditional moment equations, which often give rise to ill-posed inverse problems. Section~\ref{sec:estimation} reviews the debiased machine learning (DML) framework for estimating the target parameter, emphasizing the importance of controlling both weak error and RMSE in the estimation of nuisance functions. Section~\ref{sec:estimation_nuisance_fns} outlines the broad class of regularized estimation strategies, with a particular focus on adversarial estimators. Section~\ref{sec:conv_analysis} presents our main theoretical results: finite-sample bounds on weak error and RMSE of KRAS estimator under interpretable source conditions, without requiring uniqueness of the solution to the inverse problem. Section~\ref{sec:rel_work} provides comparative analysis of KRAS, LRAS and KMMR estimators. We conclude in Section~\ref{sec:conclusions} with a discussion of open problems.

\section{Target Parameter}\label{sec:targer_param}

Let $O_1, \dots, O_n$ be independent copies of a random vector $O$ with unknown probability distribution $P_0$  assumed to belong to a class $\mc{P}$---to be specified later---of mutually absolutely continuous laws $P$ on some sample space to be defined later. Let $W:=(W', X)$ and $Z:=(Z',X)$ be two subvectors of $O$, where $X$ contains the variables common to both $W$ and $Z$, if any, and $W'$ and $Z'$ contain the remaining, non-overlapping variables. We assume $W'$ and $Z'$ are both non-empty, so that $W$ and $Z$ are distinct but may share the components collected in $X$.

For any random vector $V$ distributed according to $P_V$, we let $\mc{L}^2(P_V)$ denote the space of real-valued measurable functions on the sample space of $V$ with finite second moment under $P_V$. We use calligraphic upper case letters to denote the, assumed common, support of random variables under any law $P$ in $\mc{P}$, i.e. if $V$ is a subvector of $O$, $\mc{V}$ denotes the common support of $V$ under any $P$ in $\mc{P}$. We denote with $\Vert\cdot\Vert_2$
and $\langle\cdot,\cdot\rangle_2$ the norm and the inner product in the $\mc{L}^2$ space with respect to the relevant probability law. For any real-valued map $f :\mc{V}\mapsto \mathbb{R}$, $\Vert f\Vert_\infty:=\sup_{v\in\mc{V}}|f(v)|$. We let $\mathbb{E}_n$ denote the empirical mean operator.

We are concerned with inference based on $O_1, \dots, O_n$ about the value taken at $P_0$ of functionals $P \mapsto \theta(P)$ that admit the following triple representation:
\begin{equation}
    \theta(P)=\Psi_P(h_{0,P})=\Phi_P(g_{0,P})=\mathbb{E}_P\left\{r(O)h_0(W)g_{0,P}(Z)\right\}\label{eq:target_param}
\end{equation}
where 
\begin{enumerate}
    \item 
    $\Psi_P(h) :=\mathbb{E}_P\left\{\widetilde{m}(O;h)\right\}, \Phi_P(g) :=\mathbb{E}_P\left\{m(O;g)\right\} $
    with $\widetilde{m}:\mc{O}\times\mc{L}^2(P_W)\rightarrow\mathbb{R}$ and  $m:\mc{O}\times\mc{L}^2(P_Z)\rightarrow\mathbb{R}$ are known mappings such that for each $o$, the maps $h \mapsto \tilde {m}(o;h)$ and $g\mapsto m(o;g)$ are linear, and the functionals $h\mapsto\Psi_P(h)$ and $g\mapsto\Phi_P(g)$
are bounded (and linear) on $\mc{L}^2(P_W)$ and $\mc{L}^2(P_Z)$ respectively, with Riesz representers denoted by $\tilde\rho_P(W)$ and $\rho_P(Z)$; and
\item $h_{0,P}$ and $g_{0,P}$ solve respectively the integral equations 
\begin{equation}
    \E_P\left\{r(O)h(W)|Z\right\}=\rho_{P} (Z) \quad a.e. (P_Z),\label{eq:int_eq_h}
\end{equation}
and 
\begin{equation}
    \E_P\left\{r(O)g(Z)|W\right\}=\tilde\rho_{P} (W) \quad a.e. (P_W)\label{eq:int_eq_g},
\end{equation} 
where $r$ is a known measurable function of $O$ satisfying $\Vert r\Vert_{\infty} < \infty $. 
\end{enumerate}

We make no additional assumptions on the data generating process other than the existence of solutions to equations~\eqref{eq:int_eq_h} and~\eqref{eq:int_eq_g} at $P=P_0$. Therefore, the aforementioned model $\mc{P}$ is the set of all laws $P$ on the sample space of $O$ such that equations~\eqref{eq:int_eq_h} and~\eqref{eq:int_eq_g} have a solution.
Importantly, we assume that solutions to these equations exist but we do not require these solutions to be unique. Nevertheless, when both equations admit solutions, the parameter $\theta(P)$ is well-defined, in the sense that its value remains the same regardless of the specific solutions $h_{0,P}$ and $g_{0,P}$ that are used in~\eqref{eq:target_param}. This well known fact (see, e.g. \citet{severini2012efficiency}) can be seen as follows. Let $\mc{T}:\mc{L}^2(P_W)\rightarrow\mc{L}^2(P_Z)$ and $\mc{T}^*:\mc{L}^2(P_Z)\rightarrow\mc{L}^2(P_W)$ denote the operators
\begin{equation*}
    \mc{T}h(Z):=\E_P\left\{r(O)h(W)|Z\right\} \quad \text{and} \quad \mc{T}^*g(W):=\E_P\left\{r(O)g(Z)|W\right\}.
\end{equation*}

The operator $\mc{T}_P^*$ is the adjoint operator of $\mc{T}_P$. 
If $h_{0,P}$ and $h^\diamond_{0,P}$ are two solutions of equation~\eqref{eq:int_eq_h}, then $h_{0,P}-h^\diamond_{0,P}$ is in the null space of $\mc{T}_P$ and consequently is orthogonal to the range of the adjoint operator $\mc{T}_P^*$. On the other hand, the existence of a solution to equation~\eqref{eq:int_eq_g} implies that $\tilde{\rho}_P(W)$ is in the range of the operator $\mc{T}_P^*$, and consequently is orthogonal to  $h_{0,P}-h^\diamond_{0,P}$. Therefore, we have 
$\E_P\{\widetilde{m}(O;h_{0,P})\}-\E_P\{\widetilde{m}(O;h^\diamond_{0,P})\}=
\E_P\{(h_{0,P}-h^\diamond_{0,P}) \tilde{\rho}_P(W)\}=0.$  
The same reasoning can be used to argue the uniqueness of $\E_P\{ m(O;g_{0,P})\}$ regardless of the solution $g_{0,P}$ of~\eqref{eq:int_eq_g}.

For some estimands $\theta(P_0)$, see for instance Example~\ref{example_mtp}
below, it just happens that it is a-priori known that for some non-empty subset $\mc{X}_0\subsetneq\mc{X}$ it holds that $\rho_P(Z)=0$ a.e. $(P_Z)$ if and only if $Z= (Z', X)$ a.e. $(P_Z)$ and $X\notin\mc{X}_0$. In such a case, to enforce this restriction we will restrict attention to solutions of equation~\eqref{eq:int_eq_h} of the form  $h_{0,P} (w) = I_0(x) h_{0,P}^\dag (w) $ where  $h_{0,P}^\dag$ is any function in $\mc{L}^2(P_W)$ that satisfies $\mc{T}_P(h)(z)=\rho_P(z)$ whenever $z=(z',x)$ and $x \in \mc{X}_0$,  where here and throughout
\begin{equation*}
    I_0(x):=1\quad\text{if }x\in\mc{X}_0\quad\text{and}\quad I_0(x) := 0\quad\text{otherwise.}
\end{equation*}
We follow the same rule for defining the solutions $g_{0,P}$ whenever the Riesz representer $\tilde{\rho}_P(W)$ has structural zeros at some subset of the support of $X$.

Estimands of the form $\theta(P_0)$ are ubiquitous in causal inference, including settings with unobserved confounding---where identification is achieved by means of instrumental variables \citep{santos2011instrumental, severini2012efficiency, escanciano2013identification, freyberger2015identification} or by means of proxy variables \citep{miao2018identifying, tchetgen2020introduction, kallus2021causal, cui2024semiparametric}---as well as in missing-not-at-random data scenarios, where shadow variables enable identification \citep{miao2016varieties, li2023non, miao2024identification}. We review three examples next. Inference for general estimads $\theta(P_0)$, under weak assumptions on $h_{0,{P_0}}$ and $g_{0,{P_0}}$, has been the focus of extensive recent work \citep{ghassami2022minimax, bennett2022inference,chernozhukov2023simple, bennett2023source}. The general approach adopted in these papers, often referred to as debiased machine learning (DML), is reviewed in the next section.

\subsection{Examples}\label{sec:examples}

In this section, we provide several examples of parameters satisfying the representation~\eqref{eq:target_param}.

\begin{example}[Proximal Inference for Generalized Average Causal Effects]\label{example_gace}

Suppose that $O=(L,A,Z',W',Y)$, where $A$ is a discrete or continuous treatment, $Y$ an outcome, and $L$ a vector of observed baseline covariates, and $Z'$ and $W'$ are correlates of an unmeasured confounder $U$, referred to as negative control treatments and negative control outcomes, respectively, satisfying $Z' \perp Y \mid A,L,U$ and  $(A,Z') \perp W' \mid L,U$. For each $a\in\mc{A}$, let $Y(a)$ denote the potential outcome under treatment level $a$ and $Y ( \cdot):= \left\{Y(a):a \in \mc{A} \right\}$. Let $O_{\text{full}}=(L,A,Z',W',Y, Y(\cdot))$. Let $P_0$ and $P_{\text{full},0}$ denote the unknown distributions of the factual random vector $O$ and the factual/counterfactual vector $O_{\text{full}}$. Suppose the goal is to estimate a parameter $\gamma(P_{\text{full},0}) := \mathbb{E}_{P_{\text{full},0}}\left\{ Y(A)\pi(A,L)\right\}$ where $\pi:\mc{A}\times\mc{L}\mapsto\mathbb{R}$ is a given specified weighting function. In particular, when $A$ is binary and $\pi(a,l) = 2a - 1$, $\gamma(P_{\text{full},0})$ is the average treatment effect (ATE), and when $\pi(a|l)=a$, $\gamma(P_{\text{full},0})=\E_{P_{full,0}}\{Y(1)\}$. Under the consistency assumption $Y(A)=Y$ and additional regularity conditions which essentially ensure that $Z'$ and $W'$ are strong proxies of the unmeasured confounder $U$, it can be shown (see \cite{miao2018identifying, tchetgen2020introduction, kallus2021causal, ghassami2022minimax, cui2024semiparametric}) that $\gamma(P_{\text{full},0})$ is identified and equal to $\theta(P_0)$ of the form~\eqref{eq:target_param}, with $W=(L,A,W')$, $Z=(L,A,Z')$, $
r(o) = 1$, $\tilde{m}(o;h) =\pi(a,l) \cdot h(w)$, $m(o;g) = y \cdot g(z)$
and corresponding Riesz representers $\tilde\rho_{P_0}(w) = \pi(a,l)/p_{0,A|L,W'}(a|l,w')$ and $\rho_{P_0}(z) = \mathbb{E}_{P_0}\left(Y \mid Z=z\right)$,
where $p_{0,A|L,W'}$ denotes the conditional density or probability mass function of $A$ given $(L,W')$. In this example $X:=(A,L)$. Note that while $\rho_{P_0}$ does not have structural zeros, $\tilde\rho_{P_0}$ will have exactly the same structural zeros as the weight function $\pi(a,l)$. For instance, if $A$ is binary and $\pi(a,l)=a$, then $\tilde\rho_{P_0}(w)=0$ for $w=(a=0,l,w')$. 
\end{example}

\begin{example}[Proximal Inference for Policy Shift Interventions]\label{example_mtp} Suppose that under the same assumptions as in Example~\ref{example_gace}, with $A$ being continuous, we are interested in estimating the parameter $\gamma(P_{\text{full},0}) := \mathbb{E}_{P_{\text{full},0}}\left[ Y\{q(A)\}\right]$ where $q:\mc{A} \rightarrow \mc{A}$ is a known function of the actually received treatment \citep{haneuse2013estimation, diaz2023nonparametric}. For instance, $q$ may take the form $q(a)=\left(a+\delta\right) I\{a\in[c,d-\delta-\varepsilon)\} +  \left\{a+\delta(d-a)/(\delta+\varepsilon)\right\} I\{a\in [d-\delta-\varepsilon, d]\}$ where  $[c,d]$ is the support of $A$ and $\delta$ and $\varepsilon$ are fixed pre-specified constants such that $d-\delta-\varepsilon >c$. \cite{olivas2025chap3} showed that under certain regularity conditions which include the requirement that $q$ be invertible and differentiable almost everywhere, $\theta(P_{full,0})$ is identified and admits the representation~\eqref{eq:target_param}, with
$ 
r(o) = 1$, $\tilde{m}(o;h) =h\{q(a),l,w'\}$, $m(o;g) = y \cdot g(z)$
and corresponding Riesz representers $\rho_{P_0}$ as in Example~\ref{example_gace}---since the map $m(o;g)$ is the same as in that example---and $\tilde\rho_{P_0}(w) = I\{a\in[c+\delta,d]\}\cdot\frac{dq^{-1}(a)}{da}\cdot\frac{p_{0,A|L,W}\{q^{-1}(a)\mid l,w'\}}{p_{0,A|L,W}(a|l,w')}$. Recalling that in this problem $X:=(L,A)$, we see that $\tilde\rho_{P_0}$ is zero outside the set $\{(l,a):a\in[c+\delta,d]\}$, so this last set plays the role of $\mc{X}_0$ for estimation of the solution of equation~\eqref{eq:int_eq_g}.
\end{example}

\begin{example}[Missing Not At Random Data with Shadow Variables]\label{example_mnar}
Suppose that $O=(\Delta,\Delta Y,X,Z')$, where $\Delta$ is binary, $Y$ is an outcome observed if and only if $\Delta = 1$, $X$ is a vector of covariates, and $Z'$ is a variable informative of $Y$, referred to as shadow variable, satisfying $Z'\perp \Delta\mid X,Y$ and $Z'\not\perp Y\mid X,\Delta=1$. Let $O_{full}=(\Delta, Y,X,Z)$ denote the full data vector. Let $P_0$ and $P_{\text{full},0}$ denote the unknown distributions of the observed random vector $O$ and the full data vector $O_{\text{full}}$. Suppose the missingness mechanism is not at random and that the goal is to estimate a parameter $\gamma(P_{\text{full},0}) := \mathbb{E}_{P_{\text{full},0}}\left\{ (1-\Delta)\pi(X, Y) \right\}$ where $\pi:\mc{X}\times\mc{Y}\mapsto\mathbb{R}$ is a given function. In particular, when $\pi(X, Y) = Y$, $\gamma(P_{\text{full},0})/P_0(\Delta=0)$ is the mean outcome among individuals with missing outcomes. Under certain structural assumptions on $\pi$,---namely, that $\pi(X,Y)$ depends only on features of $Y$ that are indirectly captured through $Z'$ given $X$---it can be shown \citep{miao2016varieties, li2023non,miao2024identification} that $\gamma(P_{\text{full},0})$ is identified and equals $\theta(P_0)$ of the form~\eqref{eq:target_param}, with $W=(X,\Delta Y)$, $Z=(X,Z')$,
$ 
r(o)=\delta$, $ \widetilde{m}(o;h)=\delta\pi(x,\delta y)h(x,\delta y)$, $  m(o;g)=(1-\delta)g(x,z)$ and corresponding Riesz representers $\tilde\rho_{P_0}(x,\delta y)=\pi(x,\delta y)\E(\Delta\mid X=x,\Delta Y=\delta y)$ and $\rho_{P_0}(x,z)=\E(1-\Delta\mid X=x,Z=z)$.
\end{example}

\section{Debiased Estimation of $\theta(P_0)$}\label{sec:estimation}

In this section we review the general debiased-machine learning estimation strategy for estimating parameters of the form $\theta(P_0)$ as defined in~\eqref{eq:target_param}. 
It is well known that if an estimator $\hat{h}$ is obtained using flexible statistical methods, the plug-in estimator $\hat{\Psi}(\hat{h}) := n^{-1} \sum_{i=1}^n m(O_i, \hat{h})$ does not generally achieve $\sqrt{n}$--rate convergence to $\theta(P_0)$ even when $\hat{h}$ converges in $\mc{L}^2(P_{0,W})$ to a solution $h_{0,P_0}$. The main reason is that the dependence of
\[
\Psi_{P_0}(\hat{h}) - \Psi_{P_0}(h_{0,P_0}) = \int (\hat{h}-h_{0,P_0})(w)\:\tilde{\rho}_{P_0}(w)\: dP_0(o)
\]
on $\hat{h} - h_{0,P_0}$ causes the plug-in estimator to inherit the bias of $\hat{h}$. Writing
\begin{align*}
\int (\hat{h}-h_{0,P_0})(w)\:\tilde{\rho}_{P_0}(w)\: dP_0(o) &= \int \hat{h}(w)\tilde{\rho}_{P_0}(w) dP_0(o)-\int m(o,g_{0,P_0}) dP_0(o)
\\
&=\int r(o)\hat{h}(w)g_{0,P_0}(z) dP(o)-\int m(o,g_{0,P_0}) dP_0(o)
\end{align*}
suggests estimating the bias of the plug-in estimator with
$
n^{-1}\sum_{i=1}^n\{r(O_i)\hat{h}(W_i)\hat{g}(Z_i)-m(O_i,\hat{g})\}$,
where $\hat{g}$ is an estimator converging to a solution  $g_{0,P}$ constructed using flexible statistical methods. Subtracting this estimated bias from the plug-in estimator, yields the bias-corrected estimator of $\theta(P_0)$ given by $n^{-1}\sum_{i=1}^n \phi(O_i;\hat{h},\hat{g})$, where for any $h$ and $g$,
\begin{equation*}
    \phi(O;h,g) := \tilde{m}(O;h) + m(O;g) - r(O) h(W) g(Z).
\end{equation*}
The map $o \mapsto \phi (o;h_{P_0},g_{P_0}) - \theta(P_0)$ corresponds to an influence function at $P_0$ of the functional $P \mapsto\theta(P)$ in model $\mathcal{P}$ (see, e.g., \cite{kennedy2024semiparametric} and references therein). 

The preceding debiased estimator may still fail to be asymptotically linear when flexible machine learning methods are employed to compute $\hat{h}$ and $\hat{g}$ due to overfitting. A popular approach to addressing this issue is to integrate debiasing with a technique called cross-fitting. Specifically, the sample is partitioned into $K$ roughly equal sized subsamples, the data in all but one sample, say the $k^{th}$ sample, is used to compute the estimators of the nuisance functions, denoted as $\hat{h}^{(-k)}$ and $\hat{g}^{(-k)}$, and the estimator of $\theta(P_0)$ is computed as
$
\hat{\theta}=n^{-1}\sum_{k=1}^K\:\sum_{i\: \text{in sample}\:k}\phi(O_i;\hat{h}^{(-k)},\hat{g}^{(-k)}).$
This estimator belongs to the class of DML estimators \citep{chernozhukov2018double}, and it is also commonly described as an influence-function-based one-step estimator with cross-fitting \citep{robins2008higher}. 

The following theorem establishes the asymptotic behavior of $\hat\theta$, a result that is now standard in the literature on semiparametric and DML estimation, and whose proof we consequently omit (see \cite{robins2008higher, chernozhukov2018double, chernozhukov2023simple} and \cite{kennedy2024semiparametric}). 

\begin{theorem}
   Suppose that $h_{0,P_0}$ and $g_{0,P_0}$ solve equations~\eqref{eq:int_eq_h} and~\eqref{eq:int_eq_g}, respectively, at $P=P_0$. Suppose that $\int\{\phi(o;\hat {h}^{(-k)},\hat{g}^{(-k)}) - \phi(o; h_{0,P_0},g_{0,P_0})\}^2 dP_0(o) = o_P(1)$ holds for each $k \in \{1,...,K\}$. Then, $\sqrt{n}\{ \hat\theta-\theta(P_0)\} =\sqrt{n}\mathbb{E}_{n}%
\left[ \phi (O;h_{0,P_0},g_{0,P_0})-\theta(P_0)
\right] +\sqrt{n}\:R_{n}+o_{P}(1)$, where
    \[
      R_n:=\sum_{k=1}^K\frac{n_k}{n} \:\int r(o) \left\{\hat h^{(-k)}(w)-h_{0,P_0}(w)\right\}\left\{g_{0,P_0}(z)-\hat g^{(-k)}(z)\right\}dP_0(o),
    \]
with $n_k$ denoting the size of the $k^{th}$ subsample. In particular, if $\E_{P_0}\{\phi(O;h_{0,P_0},g_{0,P_0})^2\}<\infty$ and $R_n=o_p(n^{-1/2})$, then $ 
    \sqrt{n}\{ \hat\theta-\theta(P_0)\}\stackrel{\text{d}}{\longrightarrow}\mc{N}(0,\tau_0^2),
    $
where $\tau_0 ^{2}=\E_{P_0}[ \{ \phi (O;h_{0,P_0},g_{0,P_0})-\theta(P_0)\}
^{2}] $.
\label{theo:asymp_distr}
\end{theorem}

Theorem \ref{theo:asymp_distr} highlights the importance of constructing estimators $\hat h^{(-k)}$ and $\hat g^{(-k)}$ that converge sufficiently quickly to solutions $h_{0,P_0}$ and $g_{0,P_0}$ of the integral equations~\eqref{eq:int_eq_h} and~\eqref{eq:int_eq_g}, respectively, in order to ensure that $R_n=o_P(n^{-1/2})$. As noted in \cite{bennett2022inference} and \cite{chernozhukov2023simple}, a sufficient condition on the convergence rates of the nuisance estimators to ensure that $R_n=o_P(n^{-1/2})$ can be derived from the following derivation: 
    
    \begin{align*}
  \vert R_n\vert &\le \sum_{k=1}^K\frac{n_k}{n} \: 
   \Bigg| \int \Bigg[ \int r(o)\left\{\hat h^{(-k)}(w)-h_{0,P_0}(w)\right\}dP_0(o|z)\Bigg]\left\{g_{0,P_0}(z)-\hat g^{(-k)}(z)\right\}dP_0(z) \Bigg|
  \\
  & \le \sum_{k=1}^K\frac{n_k}{n} \left\Vert \mc{T}_{P_0} (\hat h^{(-k)} -h_{0,P_0}) \right\Vert_2 \left\Vert \hat g^{(-k)} -g_{0,P_0} \right\Vert_2
\end{align*}
where the inequality follows from the Cauchy–Schwarz inequality. A symmetric bound can be obtained by integrating first with respect to $dP_0(o|w)$. Combining both bounds, yields:
\[
\vert R_n \vert  \le \sum_{k=1}^n \frac{n_k}{n}\min\{\Vert \mc{T}_{P_0} (\hat h^{(-k)} -h_{0,P_0}) \Vert_2\Vert \hat g^{(-k)} -g_{0,P_0} \Vert_2, \Vert \mc{T}^*_{P_0} (\hat g^{(-k)} -g_{0,P_0})\Vert_2\Vert \hat h^{(-k)} -h_{0,P_0}\Vert_2\}
\] 
Thus, we upper bound $R_n$ with the minimum of two products: the weak error of one estimator times the RMSE of the other--by weak error, say of estimator $\hat{h}^{(k)}$, we mean $\Vert \mc{T}_{P_0}(\hat{h}^{(k)}-h_{0,P_0})\Vert_2$. We can also upper bound $R_n$ using the looser inequality:
\[
\vert R_n \vert \le \Vert r \Vert_\infty\sum_{k=1}^n \frac{n_k}{n} \left\Vert \hat g^{(-k)} -g_{0,P_0}\right \Vert_2 \left\Vert \hat h^{(-k)} -h_{0,P_0}\right\Vert_2
\]

This later bound is less sharp than the former, since weak errors tend to be substantially smaller than RMSEs. This observation highlights a limitation of the existing analyses of KMMR estimators because they provide only bounds for RMSE: relying exclusively on RMSE bounds can lead to overly stringent smoothness conditions when establishing $\sqrt{n}$--consistency of the DML estimator of $\theta(P_0)$.

\section{Adversarial Estimation of the Nuisance Functions}\label{sec:estimation_nuisance_fns}

For the remainder of the paper, we focus on estimating the nuisance function that solves equation~\eqref{eq:int_eq_h}. The same analysis applies to equation~\eqref{eq:int_eq_g} after an appropriate redefinition of the relevant objects. Furthermore, to simplify notation, throughout we will omit the subscript $P$ from the operator $\mc{T}_{P}$, the solution $h_{0,P}$, the Riesz representer $\rho_P$, and all other operators and quantities, e.g. expectations, defined throughout that depend on $P$.

Estimating solutions to integral equations is inherently challenging due to their ill-posed nature: small perturbations in the input function $\rho$ can lead to large deviations in the solution. A naive way to circumvent this instability is to restrict attention to a sufficiently small subspace $\mathcal{H}$ of $\mathcal{L}^2(P_W)$ in which the inverse problem becomes well-posed. When $\mathcal{H}$ is finite-dimensional, the resulting estimator is computationally tractable, but the approach is theoretically unsatisfying as it relies on the strong assumption that a solution lies in the small subspace $\mc{H}$. 

A broadly applicable alternative strategy is to regularize the estimation problem. This can be done by either allowing the finite-dimensional space $\mathcal{H}$ to grow with the sample size, as in sieve-based methods, or by penalizing the complexity of functions in an infinite-dimensional space $\mathcal{H}$, such as an RKHS. Both strategies yield estimators that approximate the solution of the original ill-posed problem as the level of regularization is gradually reduced \citep{carrasco2007linear, cavalier2011inverse}. From a unified perspective, several regularization approaches can be viewed as targeting solutions to surrogate regularized optimization problems of the form

\begin{equation}
    \arg\min_{h \in \mathcal{H}} \left\{ \| \rho - \mc{T} I_0h \|^2 + \lambda_{\mc{H}} \| h-h^\blacklozenge \|^2 \right\},\label{eq:regularized_problem}
\end{equation}
where $h^\blacklozenge$ is a fixed reference function, typically set to zero. These estimators differ in their choice of function class $\mathcal{H}$ (typically a subset of a normed space), in whether they include a penalization term (i.e., whether $\lambda_{\mathcal{H}}$ is positive), and in the choice of norms used for both the loss term $\| \rho - \mc{T} I_0h \|^2$ and the regularization term $\| h -h^\blacklozenge \|^2$.

When $\lambda_{\mathcal{H}} = 0$, a common choice is to take $\mathcal{H}$ as a finite-dimensional subspace spanned by the first $p$ elements of an orthonormal basis of $\mathcal{L}^2(P_W)$ and to use the $\mathcal{L}^2(P_W)$--norm for the loss. In this setting, regularization is relaxed by allowing the dimension $p$ to grow with the sample size. Estimators constructed in this way are referred to as sieve-based and include methods proposed in \citet{newey2003instrumental} and \citet{ai2003efficient, ai2007estimation}. In contrast, when $\lambda_{\mathcal{H}} \neq 0$, the approach is referred to as Tikhonov regularization. An influential earlier example of this approach is the work of \citet{chen2012estimation}, who consider nonparametric conditional moment models and apply Tikhonov regularization by penalizing the $\mathcal{L}^2$--norm of the solution in a sieve space.
 

In more recent proposals, $\mathcal{H}$ is typically taken to be an infinite-dimensional function class, often restricted via a norm constraint, or a very high-dimensional normed space, like  shape constrained functions, neural networks (NN), RKHSs, and random forests. The penalty parameter $\lambda_{\mathcal{H}}$ is strictly positive and regularization is relaxed by letting it decrease to $0$. For instance, \citet{dikkala2020minimax} allow $\mathcal{H}$ to be any of the aforementioned spaces; \citet{ghassami2022minimax} focus on RKHSs; and \citet{bennett2022inference} consider both RKHSs and NNs. In all of these papers the loss is measured with the $\mathcal{L}^2$--norm, while the penalty is defined via the norm of the function space $\mc{H}$ with $h^\blacklozenge =0$. By contrast, other recent approaches use the $\mathcal{L}^2$--norm for both the loss and the penalty terms. This includes \citet{liao2020provably}, who choose $\mc{H}$ as NNs; and \citet{bennett2023source}, who consider norm-constrained high-dimensional linear models, RKHSs, and NNs. 
The latter also considers an iterated Tikhonov approach, starting with an initial estimator targeting the solution $h_\lambda^{(0)}$ to the regularized problem with $h^\blacklozenge = 0$. At each subsequent step, the estimator targets the solution $h_\lambda^{(k)}$ to the same problem but with $h^\blacklozenge$ set to the previous iterate $h_\lambda^{(k-1)}$. This process is repeated for a fixed number of iterations or an adaptively chosen stopping point.

In the Maximal Moment Restriction (MMR) proposals of \cite{muandet2020kernel, mastouri2021proximal,kallus2021causal, zhang2023instrumental}; and \cite{park2024proximal}, the surrogate problem takes the form:
\begin{equation*}
\arg\min_{h\in\mc{H}}\max_{g\in\mc{G}}\left[\E\{m(O;g)-h(W)g(Z)\}\right]^2+\lambda_\mc{H}\Vert h\Vert^2.
\end{equation*}
We discuss the MMR proposals in Section~\ref{sec:compare_mmr}.

In the next subsection, we review the class of RAS estimators. To avoid notational burden, we will drop the superscript $(-k)$ from $\hat{h}^{(-k)}$ and the subscript $k$ from $n_k$. Furthermore, we will use $\mathbb{E}_n$ to denote the empirical mean operator calculated with data in the entire sample but the $k^{th}$ sample.

 \subsection{Motivation and Rationale for Regularized Adversarial Stabilized Estimation Strategies}\label{sec:adversarial_estimators}

In addition to the statistical advantages discussed in the introduction—most notably, their ability to adaptively focus on directions where the conditional moment restriction is most difficult to satisfy—adversarial estimation strategies offer two key computational benefits. First, they bypass the need to estimate the conditional mean operator $\mathcal{T}$ and their convergence does not depend on smoothness assumptions on the conditional distribution of $W$ given $Z$. Second, they do not require prior knowledge of the functional form of the Riesz representer $\rho$, which facilitates greater automation of the estimation procedure. Methods that do not require knowledge of the functional form of the Riesz representer are referred to in the literature as auto-debiased machine learning (auto-DML). They were first proposed in the case where $X = W = Z$ by \cite{chernozhukov2022riesznet, chernozhukov2022automatic, chernozhukov2022debiased}, and later extended to the setting where $W$ and $Z$ are distinct subvectors of $O$ by \cite{bennett2022inference,bennett2023source}.

The RAS estimation strategies we introduce next, unlike the alternative MMR class of adversarial estimators studied in \cite{muandet2020kernel,zhang2023instrumental, mastouri2021proximal}; and \cite{park2024proximal}, have the advantage that existing analytical results provide bounds for the weak error. As discussed in Section~\ref{sec:estimation}, this makes it possible to establish root--$n$ consistency of the resulting estimator of $\theta(P_0)$ under weaker conditions on the solutions to equations~\eqref{eq:int_eq_h} and~\eqref{eq:int_eq_g}.

Quite generally, RAS estimators are of the form $\hat h(W) =I_0(X)\tilde h(W)$ where
\begin{align}
\tilde h:=\arg\min_{h\in\mc{H}}\max_{g\in\mc{G}}[&\E_n\{m(O;g)-I_0(X)r(O)h(W)g(Z)-c^2 I_0(X)g(Z)^2\}-\lambda_{\mc{G}}\Vert g\Vert^2+\lambda_{\mc{H}}\Vert h\Vert^2]
 \label{def:adversarial}
\end{align}
for some constant $c\ne0$ and, $\mc{G}$ and $\mc{H}$ specified function classes \citep{muandet2020dual,dikkala2020minimax, liao2020provably, kallus2021causal, ghassami2022minimax, bennett2022inference, bennett2023source}. As noted earlier, \cite{bennett2023source} also consider iterated Tikhonov estimators with penalties centered at the estimator from the previous iteration.  
The aforementioned proposals differ on the choice of norms $\Vert\cdot\Vert$ for $g$ and $h$. Additionally, some proposals set $\lambda_{\mc{H}}=0$ and/or $\lambda_{\mc{G}}=0$. 

To motivate the RAS approach we first observe that for any $g \in L^2(P_Z)$, it holds that
$\mathbb{E}[m(O;g)]=\mathbb{E}[\rho(Z)g(Z)]
=\mathbb{E}[r(O) h(W)g(Z)]$,
where the first equality is because $\rho$ is the Riesz representer, and the second is by $h_0$ satisfying equation~\eqref{eq:int_eq_h}. In particular, this implies that for any $g \: \in \: \mc{L}^2(P_Z)$, the residual
\begin{equation}
     m(O;g)-I_0(X) r(O) g(Z)h(W) \label{eq:residuals}
\end{equation}
has mean equal to 0 at $h=h_0^\dag$ where $h_0=I_0h_0^\dag$ is a solution of equation~\eqref{eq:int_eq_h}. Consequently, 
\begin{equation*}
Q_{\mathbb{J}}(h):=\sum_{j=1}^{\mathbb{J}} \left\{\mathbb{E}[m(O;g_j)-I_0(X) r(O) g_j(Z)h(W)]\right\}^2
\end{equation*}
for given functions $g_j, j = 1, \dots , J $ where $\mathbb{J} \subseteq \mathbb{N}$, attains its minimum at $h_0^\dag$. One can then use the empirical version of $Q_{\mathbb{J}}(h)$ as the basis for deriving estimators of $h_{0}^{\dag}$--a strategy that underlies the classical GMM approach. The functions $g_j$ are known in the machine learning literature as $\textit{test functions}$. The question then is: how many and which test functions do we choose? The following result sheds light into the answer. In what follows, given a subspace $\mc{G}$ of $\mc{L}^2(P_Z)$, we let $I_0\:\mc{G}:=\left\{I_0\:g\::\:g\in\mc{G}\right\}$
and we let $\overline{I_0\mc{G}}$ denote its closure in $\mc{L}^2(P_Z)$. 
\begin{proposition}    
Suppose $\{g_j=I_0g_j^\dag:g_j^\dag \in \mc{G},j\in\mathbb{J}\}$ is an $\mc{L}^2(P_Z)$--orthonormal basis of $\overline{I_0\mc{G}}$. Let $\mathcal{P}_{\overline{I_0\mc{G}}}:\mc{L}^2(P_Z) \rightarrow \overline{I_0\mc{G}}$  denote the projection operator. Then 
$ 
Q_{\mathbb{J}}(h)=\Vert\mathcal{P}_{\overline{I_0\mc{G}}} \mc{T} I_0(h_0^\dag-h)\Vert_2^2. $\label{prop:sum_squares}
\end{proposition}
To understand the relevance of this result avoiding distracting heavy notation, suppose for the purpose of this discussion that $I_0 = 1$ and that the integral equation~\eqref{eq:int_eq_h} admits a unique solution $h_0$ lying in a known function class $\mathcal{H}$. Proposition~\ref{prop:sum_squares} implies that estimators minimizing $Q_{\mathbb{J}}(h)$ over $\mathcal{H}$ implicitly target solutions to the projected equation
\begin{equation}
\mathcal{P}_{\overline{\mathcal{G}}} \circ \mathcal{T} \: h = \mathcal{P}_{\overline{\mathcal{G}}} \: \rho,\label{eq:projected_int_eq}
\end{equation}
where $\overline{\mathcal{G}}$ denotes the closure of the linear span of the test functions $\{g_j\}_{j \in \mathbb{J}}$. If $\overline{\mathcal{G}}$ is a strict subspace of $\mathcal{L}^2(P_Z)$, then any function of the form $h_0 + h^\diamond$, with $\mathcal{T} h^\diamond \in \text{Null}(\mathcal{P}_{\overline{\mathcal{G}}})$, also solves~\eqref{eq:projected_int_eq}. Consequently, minimizing $Q_{\mathbb{J}}$ does not distinguish $h_0$ from these alternatives.

This ambiguity cannot be resolved by selecting the minimum norm solution of~\eqref{eq:projected_int_eq}, as this solution will generally differ from $h_0$, the unique solution of the original problem~\eqref{eq:int_eq_h}. To see this, suppose that $W$ and $Z$ have no overlap (i.e., $X$ is null), in which case $\mathcal{T}:\mc{L}^2(P_W) \rightarrow \mc{L}^2(P_Z)$ is compact. Let $(\mu_i, \varphi_i, \psi_i)$ for $i = 1, 2, \dots$ denote the singular value decomposition of $\mathcal{T}$. Since, by assumption, equation~\eqref{eq:int_eq_h} has a solution, then the Riesz representer $\rho$ lies in the range of $\mathcal{T}$, and can therefore be written as
$ 
\rho = \sum_{i:\mu_i \neq 0} \langle \rho, \psi_i \rangle_2 \: \psi_i$.
The minimum-norm solution to the projected equation~\eqref{eq:projected_int_eq} is then $ 
h_0^{\text{proj}} = \sum_{i:\: \mu_i \neq 0 \:,\: \psi_i \in \overline{\mathcal{G}}} \: \mu_i^{-1} \:{\langle \rho, \psi_i \rangle_2} \: \varphi_i$, 
while the solution of the equation of interest~\eqref{eq:int_eq_h} is $ 
h_0 = \sum_{i: \mu_i \neq 0}\: \mu_i^{-1} \:{\langle \rho, \psi_i \rangle_2} \:\varphi_i $.
If $\rho \notin \overline{\mathcal{G}}$, then there exists some $i$ for which $\langle \rho, \psi_i \rangle_2 \neq 0$ and $\psi_i \notin \overline{\mathcal{G}}$, meaning that $h_0^{\text{proj}}$ omits the corresponding direction $\varphi_i$. In the machine learning literature, this is referred to as \textit{misspecification of the test class} $\mathcal{G}$. In particular, if $\rho$ has nonzero projection onto infinitely many directions $\psi_i$, then any finite-dimensional test class $\mathcal{G}$ will necessarily be misspecified. This is the case in sieve-based GMM estimation, where the dimension of $\mathcal{G}$ increases with sample size, and the test functions $\{g_j\}$ are chosen to minimize approximation error. However, for any fixed sample size, the sieve space cannot capture all directions—especially the high-frequency components of $\rho$—and as a result, the corresponding estimator will fail to recover those components of the true solution $h_0$.

To avoid misspecification, the test class $\mathcal{G}$ must be dense in $\mathcal{L}^2(P_Z)$. However, when $Z$ has continuous components, this requires indexing the class by the entire natural numbers, i.e., $\mathbb{J} = \mathbb{N}$, implying that the objective function $Q_{\mathbb{J}}$ involves an infinite sum. As a result, the empirical version of this criterion becomes computationally infeasible, since one cannot evaluate or minimize an infinite sum in practice. Remarkably, when $\mc{T}(h_0-h)$ lies in $\mc{G}$ there exists an equivalent reformulation that avoids direct evaluation of the infinite sum. Specifically, let $c$ be any non-zero constant and define
\[
M_P(h) := 4c^2 \max_{g\in \mc{G}} \mathbb{E}_P\left\{m(O;g)-I_0(X)r(O)h(W)g(Z)-c^2I_0(X)g(Z)^2\right\}.
\]

\begin{proposition}
   For any constant $c \ne0$, any $h_0^\dag$ such that  $h_0=I_0h_0^\dag$ is a solution of equation~\eqref{eq:int_eq_h}, and any $h\in \mc{L}^2(P_{W})$, it holds that:
    \begin{enumerate}
    \item 
    If $\mc{G}\subset\mc{L}^2(P_Z)$ such that $(2c^2)^{-1}\mc{T}(h_0^\dag-h)\in\mc{G}$ then $ 
         M_P(h) =\Vert I_0\mc{T}(h_0^\dag-h)\Vert_2^2.\label{eq:mph} $ 
    \item 
    Suppose $\mc{G}$ is a subspace of $\mc{L}^2(P_{Z})$. Let $I_0\mc{G}:=\left\{I_0g\::\:g\in\mc{G}\right\}$ and let $\overline{I_0\mc{G}}$ denote its closure in $\mc{L}^2(P_Z)$. Let $\{g_j=I_0g_j^\dag:g_j^\dag \in \mc{G},\:j\in\mathbb{J}\}$ denote an $\mc{L}^2(P_Z)$--orthonormal basis of $\overline{I_0\mc{G}}$ where $\mathbb{J}\subseteq\mathbb{N}$. If  $\mc{P}_{\overline{I_0\mc{G}}}\circ\mc{T}_{P}I_0(h_0^\dag-h)\in\ I_0\:\mc{G}$, then $ 
    Q_{\mathbb{J}}(h)=M_P(h)$. 
    \end{enumerate}
    \label{prop:innermax}
\end{proposition}
Proposition~\ref{prop:innermax} is ultimately a consequence of the Fenchel conjugate formulation of the square function \citep{borwein2006convex}. Part 1 of this proposition has been shown by \cite{dikkala2020minimax} and \cite{kallus2021causal}. Part 2 of the proposition has been established in the finite-dimensional setting by \cite{bennett2019deep} and \cite{muandet2020dual}, for the specific case $m(O; g) = Y g(Z)$, with $Y$ and $Z$ denoting distinct components of $O$. To the best of our knowledge, the extension of this result to infinite-dimensional function classes has not yet been reported in the literature. Proposition~\ref{prop:innermax} has the following key implication.

\begin{corollary}
 Let $\mc{G}\subset\mc{L}^2(P_{Z})$ and let $c \ne 0$ be a constant. If there exists $h_0^\dag$ such that  $h_0=I_0h_0^\dag$ solves equation~\eqref{eq:int_eq_h}, and for any $h\in \mc{H}$ it holds that $(2c^2)^{-1}\mc{T}_{P}(h_0^\dag-h)\in\mc{G}$, then 
\begin{align*}
     &\arg\min_{h \in \mc{H}}\left\{\Vert \rho-\mc{T}I_0h\Vert_2^2+\lambda_\mc{H}\Vert h\Vert^2\right\}\\
     &=\arg\min_{h \in \mc{H}}\big[\max_{g\in \mc{G}} \mathbb{E}_P\left\{m(O;g)-I_0(X)r(O)h(W)g(Z)-c^2I_0(X)g(Z)^2\right\}+\lambda_\mc{H}\Vert h\Vert^2\big].\nonumber 
\end{align*}
\label{corollary:closed_arg_max}
\end{corollary}
It is now clear from Corollary~\ref{corollary:closed_arg_max} that when $\mc{G}$ is a linear space, under the \textit{closeness} condition 
\begin{equation}
\mc{T}(h_0^\dag-h)\in\mc{G} \quad \text{holds for every} \: h \in \mc{H},
\label{condition:closeness}
\end{equation}
the RAS estimators~\eqref{def:adversarial} with $\lambda_\mc{G}=0$ are estimating solutions of surrogate problems of the form~(\ref{eq:regularized_problem}) that employ the $\mc{L}^2(P_Z)$--norm for the loss term. 
If $\mc{H}$ is a normed space and it is known than $\Vert h\Vert_\mc{H}<B$ for some constant $B$, then the outer minimization can be restricted to functions in $\mc{H}$ with norm at most $B$. In this case, the closeness condition~\eqref{condition:closeness} only needs to hold over that restricted subset.

On the other hand, when the closeness condition~\eqref{condition:closeness} does not hold, part 2 of Proposition~\ref{prop:innermax} tells us that the RAS estimators can still be interpreted as regularized GMM estimators over a class of test functions that are dense in the $\mathcal{L}^2(P_Z)$–closure of the class $I_0 \: \mathcal{G}$. The requirement that the projection of $\mathcal{T}_{P} I_0(h_0^\dag - h)$ onto $\overline{I_0\mc{G}}$ lies in $I_0 \mathcal{G}$ is equivalent to requiring that the maximum over $\mathcal{G}$ of the expected stabilized residuals exists. Existing finite-sample error bounds for RAS estimators have been established only under the stronger closeness condition~\eqref{condition:closeness}, but the preceding discussion on test-class misspecification suggests a potential relaxation: if one is willing to assume that $\rho \in \mathcal{G}$, then it may be possible to derive analogous bounds under the weaker condition that the projection of $\mathcal{T} I_0(h_0^\dag - h)$ onto $\overline{I_0 \mathcal{G}}$ belongs to $I_0 \mathcal{G}$.

KRAS estimators $\tilde h$ of the form~\eqref{def:adversarial} with $\lambda_\mathcal{H} > 0$, $\lambda_\mathcal{G} > 0$, $\mathcal{H}$ and $\mathcal{G}$ taken as RKHSs, and penalization based on the corresponding RKHS norms were first studied by \citet{dikkala2020minimax} in the special case $m(o;g) = y \cdot g(z)$, and by \citet{ghassami2022minimax} in the more general setting $m(o;g) = s(o) g(z)$, where $s(\cdot)$ is a known function satisfying $\|s\|_\infty < \infty$. In both works, the set $\mathcal{X}_0$ was taken to be the entire domain $\mathcal{X}$, equivalently assuming $I_0 \equiv 1$. Choosing $\mathcal{H}$ and $\mathcal{G}$ as RKHSs and penalizing via their respective RKHS norms is attractive because it yields computationally stable estimators that admit a closed-form expression. In contrast, omitting the $\lambda_\mathcal{G} \|g\|_\mathcal{G}^2$ penalty or replacing RKHS norms with empirical $L^2(\mathbb{P}_n)$ norms for the penalization of $h$ may still result in closed-form estimators, but their computation typically involves inversion of nearly ill-posed $n \times n$ matrices. Including the $\lambda_\mathcal{G} \|g\|_\mathcal{G}^2$ penalization and using RKHS norms for penalizing $h$ leads instead to inversion of well-posed matrices, thus stabilizing the computation. See \citet{dikkala2020minimax} and \citet{ghassami2022minimax} for the closed form expression of KRAS estimators.

\cite{dikkala2020minimax} and \cite{ghassami2022minimax} provided a finite-sample analysis of KRAS estimators. Their analyses begin by showing that, under suitable regularity conditions and with appropriately chosen tuning parameters $\lambda_{\mathcal{G}}$ and $\lambda_{\mathcal{H}}$, the weak error $\|\mathcal{T}(\hat{h} - h_0)\|_2$ is of order $O(\delta_n)$, where $\delta_n$ is an upper bound on the critical radius of the classes $\mathcal{H}$ and $\mathcal{G}$, and $h_0$ is any solution to equation~\eqref{eq:int_eq_h}. They then assume that the solution $h_0$ to equation~\eqref{eq:int_eq_h} is unique and derive a bound on the RMSE of the form
\begin{equation}
    \|\hat{h} - h_0\|_2 \le \tau(\delta_n),\label{eq:mip}
\end{equation}
where, for any $\delta > 0$,
$ \tau(\delta) := \sup \left\{ \|h\|_2 : h \in \mathcal{H},\ \|h\|_{\mathcal{H}} \le 2B,\ \|\mathcal{T} h\|_2 \le \delta \right\}$ 
with $B$ being a universal constant such that $\|\hat{h}\|_{\mathcal{H}} \le B$ and $\|h_0\|_{\mathcal{H}} \le B$. Note that $\tau(\delta)$ is a measure of the ill-posedness of equation~\eqref{eq:int_eq_h} that depends both on the data generating process and on $\mc{H}$. 

Unfortunately, if there exists a nonzero function $h^{(0)} \in \mathcal{N}(\mathcal{T}) \cap \mathcal{H}$, where $\mathcal{N}(\mathcal{T})$ denotes the null set of the conditional mean operator $\mathcal{T}$, then the bound~\eqref{eq:mip} becomes vacuous; that is, $\tau(\delta_n) > C > 0$ for some universal constant $C$. Indeed, 
$h^{(0)^*} := B \: {h^{(0)}}/{\|h^{(0)}\|_{\mathcal{H}}}$
satisfies $\mathcal{T} h^{(0)^*} = 0$ and $\|h^{(0)^*}\|_{\mathcal{H}} = B < 2B$, implying that $\tau(\delta) \ge \|h^{(0)^*}\|_2 =: C$ for all $\delta > 0$. Similar comments were made by \cite{kallus2021causal} and \cite{bennett2023source}. 

Even under the condition $\mathcal{N}(\mathcal{T}) \cap \mathcal{H} = {0}$, there is no known general methodology for expressing or bounding $\tau(\delta)$ as a function of $\delta$. \citet{dikkala2020minimax} obtained such a bound in a particular setting, relying on a strong structural assumption on the data-generating process that, in effect, requires $Z$ to be a strong predictor of $W$. While this result illustrates that bounding $\tau(\delta)$ is possible in certain special cases, to our knowledge no general approach is currently available.

In contrast, more satisfactory finite-sample analyses exist for the computationally unstable RAS estimators that employ $\mc{H}$ and $\mc{G}$ as RKHSs with $\mathcal{L}^2(\mathbb{P}_n)$ penalties \citep{liao2020provably,bennett2023source}. These analyses adopt a different notion of ill-posedness that does not require uniqueness of the solution to the integral equation, focusing instead on the minimum $\mathcal{L}^2(P_W)$--norm solution satisfying the so-called \textit{$\beta$--source condition}.

To review the definition of the $\beta$--source condition and its implication for controlling the regularization bias in complete generality, consider an arbitrary bounded linear operator 
$\mc{A} : \mc{E}_1 \rightarrow \mc{E}_2$,
where $\mc{E}_1$ and $\mc{E}_2$ are Hilbert spaces, and suppose we aim to solve the equation $\mc{A} \:h= \rho$
for some $\rho \in \mc{E}_2$. Assume a solution to the equation exists, and let $h_0$ denote the solution with minimal $\mc{E}_1$--norm. Let $\mc{A}^*:\mc{E}_2 \rightarrow \mc{E}_1$ denote the adjoint of the operator $\mc{A}$. The $\beta$--source condition posits that $h_0$ lies in the range of the fractional power operator $(\mc{A}^* \mc{A})^{\beta /2}$; that is, 
\begin{equation}
h_0= (\mc{A}^* \mc{A})^{\beta /2} h^{*} \quad \text{for some} \: h^{*} \in \mc{E}_1.
\label{def:betasource-generic}
\end{equation}
When $\mc{A}$ is a compact operator, this condition amounts to assuming that $h_{0}$ concentrates its energy in the well-identified directions of the inverse problem—i.e., the directions along which the operator $\mathcal{A}^* \mathcal{A}$ is most stable (see Chapter 3 of \cite{carrasco2007linear} for the interpretation of the source condition when $\mc{A}$ is compact and Chapter 4 of \cite{hohage2002lecture} for the definition of the power of a bounded operator).
By excluding the possibility that $h_0$ has substantial components in directions with near-zero singular values of $\mathcal{A}$, the assumption facilitates control over $\Vert h_\lambda - h_0 \Vert_{\mc{E}_1}$ where 
\begin{equation}
    h_\lambda := \arg \min_{h \in \mc{E}_1} \Vert \mc{A} h - \rho \Vert_{\mc{E}_2}^2 + \lambda \Vert h \Vert_{\mc{E}_1}^2
\label{def:hlambda-generic}
\end{equation} 
is the solution to the surrogate Tikhonov-regularized problem. We emphasize that in~\eqref{def:hlambda-generic}, $\mc{E}_1$ and $\mc{E}_2$ are Hilbert spaces. This is a crucial assumption, as it guarantees that the argmin exists. By contrast, consider the modified problem
\begin{equation*}
\arg \min_{h \in \mc{H}_1} \big\{ \Vert \mc{A} h - \rho \Vert_{\mc{E}_2}^2 + \lambda \Vert h \Vert_{\mc{E}_1}^2 \big\}
\end{equation*}
where $\mc{H}_1 \subsetneq \mc{E}_1$. In this case, a minimizer need not exist, even if $\mc{H}_1$ is dense in $\mc{E}_1$--for example, when $\mc{H}_1$ is a universal RKHS embedded in $\mc{E}_1 = \mc{L}^2(P_W)$. Of course, existence can be recovered if one assumes in addition that the minimizer $h_\lambda$ from \eqref{def:hlambda-generic} lies in $\mc{H}_1$, although one cannot generally expect this inclusion to hold without additional structure. Importantly, assuming that $h_0$ is in $\mc{H}_1$ does not suffice to ensure that $h_\lambda$ is in $\mc{H}_1$. For instance, suppose that $\mc{H}$ is an RKHS included in $\mc{E}_1 = \mc{L}^2(P_W)$ and $\mc{H}_1=\left\{h\in\mc{H}: \Vert h\Vert_2 < B \right\}$. Suppose  that $h_0$ is in $\mc{H}_1$. In such case, one can show that $\Vert h_\lambda \Vert_2 <B$. However, $h_\lambda$ need not be in $\mc{H}_1$ because it need not fall inside $\mc{H}$. These issues will be central when we examine the assumptions used in the analysis of LRAS estimators in Section~\ref{sec:rel_work}.

The aforementioned error bounds are given by the following classical result from the theory of inverse problems (see Theorem 5.2 of \cite{hohage2002lecture}). 
\begin{theorem}
    Let $\mc{A} : \mc{E}_1 \rightarrow \mc{E}_2$ be a bounded linear operator between the Hilbert spaces $\mc{E}_1$ and $\mc{E}_2$ and let $\rho \in \mc{E}_2 $ be in the range of $\mc{A}$. Suppose the solution $h_0$ to the equation $\mc{A}h=\rho$ with minimal $\mc{E}_1$--norm satisfies the $\beta$--source condition (\ref{def:betasource-generic}). Let $h_\lambda$ be defined as in~(\ref{def:hlambda-generic}). Then, there exist constants $c_1$ and $c_2$ such that $\Vert h_0 - h_{\lambda} \Vert_{\mc{E}_1}^2 < c_1 \:\lambda^{\min\left\{\beta ,2\right\}}$ and  $ \Vert \mc{A} ( h_0 - h_{\lambda}) \Vert_{\mc{E}_2}^2 < c_2 \:\lambda^{\min\left\{\beta +1 ,2\right\}}$.
    \label{thm:source-general}
\end{theorem}
In the next section, we provide an alternative finite-sample analysis for estimators of the form $\hat{h} = I_0 \tilde{h}$, where $\tilde{h}$ is a KRAS estimator. We derive informative finite-sample bounds on both the weak error and the root mean squared error, under two variants of the $\beta$--source condition imposed on the solution with minimal $\mathcal{H}$--norm. Crucially, we do not assume $\mathcal{N}(\mathcal{T}) \cap \mathcal{H} = {0}$, nor do we rely 

In Section~\ref{sec:rel_work} we argue that while adversarial stabilized estimators with $\mathcal{L}^2(\mathbb{P}_n)$ penalties can achieve faster convergence rates under comparable smoothness assumptions, these rates rely on a strong condition that KRAS does not require—and may fail entirely when that condition is violated. We conjecture that, under these assumptions, the slower rates of KRAS are not a limitation of our analysis—i.e., they do not stem from overly loose bounds—but rather reflect an intrinsic feature of the KRAS methodology itself: by penalizing with the RKHS norm instead of the $\mathcal{L}^2(\mathbb{P}_n)$ norm, the KRAS estimator implicitly tackles a more severely ill-posed inverse problem. Thus, KRAS trades statistical speed for computational stability and less stringent assumptions. 

\section{Finite-sample Bounds on the Errors of the KRAS Estimators}\label{sec:conv_analysis}

In this section we will focus on the KRAS estimator $\hat {h}_{KRAS}(W) :=I_0(X)\tilde {h}_{KRAS}(W)$, where
\begin{align}
\tilde {h}_{KRAS}:=\arg\min_{h\in\mc{H}_B}\max_{g\in\mc{G}}&\big[\E_n\{m(O;g)-I_0(X)r(O)h(W)g(Z)-c^2\: I_0(X)g(Z)^2\} \label{def:h_tilde}\\
&-\lambda_{\mc{G}}\Vert g\Vert^2_\mc{G}+\lambda_{\mc{H}}\Vert h\Vert^2_\mc{H}\big]\nonumber,
\end{align}
$\mc{H}$ and $\mc{G}$ are RKHSs with kernels $K_{\mc{H}}$ and $K_{\mc{G}}$, and $\mc{H}_B:=\{h\in\mc{H}\mid \Vert h\Vert_\mc{H}\le B\}$ with $B$ a constant defined later after Assumption~\ref{ass:K_H}. 

We shall introduce two distinct $\beta$--source conditions and derive finite-sample bounds for both the RMSE and the weak error under each condition and additional regularity conditions. The source conditions differ in the operators that define them. The first arises directly from an application of Theorem~\ref{thm:source-general}. The second is obtained after a transformation that clarifies the implicit inverse problem actually being targeted, providing insight into its structure and enabling a qualitative comparison with other kernel-based adversarial estimators, as discussed in Section~\ref{sec:rel_work}.

In our derivation of the finite-sample bounds, we make a number of assumptions, which we now introduce and discuss. Our first assumption is a realizability condition:

\begin{assumption}
There exists a solution to equation~\eqref{eq:int_eq_h} of the form $I_0 h$ where $h \in \mathcal{H}$; equivalently, the set
$\mathcal{H}_0 := \{h \in \mathcal{H} : I_0 h \ \text{solves equation~\eqref{eq:int_eq_h}} \}$ 
is nonempty.
\label{ass:relizability}
\end{assumption}

When $I_0 \neq 1$, Assumption~\ref{ass:relizability} does not require a solution to equation~\eqref{eq:int_eq_h} to lie in $\mathcal{H}$ itself. This is important in cases where $\rho$ has a structural zero on $\mathcal{X} \setminus \mathcal{X}_0$: then no solution lies in $\mathcal{H}$ unless the RKHS contains functions that vanish exactly on $\mathcal{X} \setminus \mathcal{X}_0$. Since most RKHSs used in practice do not have this property, existing KRAS estimators from the literature---which implicitly assume $\mathcal{H}$ contains the exact solution---are generally unsuitable in such cases. This motivates our modified KRAS estimators, which explicitly accommodate instances with structural zeros.

When $\mathcal{X}_0 = \mathcal{X}$, so that $I_0 \equiv 1$, the realizability condition above reduces to the standard assumption used in \citet{dikkala2020minimax} and \citet{ghassami2022minimax}. Note that the condition $\mc{H}_0\cap\mc{N}(\mc{T})=\{0\}$ required for the bound on the RMSE computed by \cite{dikkala2020minimax} and \cite{ghassami2022minimax} to be non-vacuous rules out that $\mc{H}_0$ has more than one element. In contrast, we do not assume that $\mathcal{H}_0$ is a singleton. Instead, we allow equation (\ref{eq:int_eq_h}) to admit multiple solutions and derive error bounds on $\Vert \tilde{h}_{KRAS} - h_{0,KRAS}^\dagger \Vert_2$, where

\begin{equation}
h_{0,KRAS}^\dagger := \arg \min_{h \in \mathcal{H}_0} \Vert h \Vert_{\mathcal{H}}^2.
\label{def:h_0_dag}
\end{equation}

To ensure that $h_{0,KRAS}^\dagger$ in (\ref{def:h_0_dag}) is well-defined---meaning the argmin exists and is unique---we will make the following assumption.
\begin{assumption} 
$\mc{W}$ is compact and the kernel function $K_{\mc{H}}:\mc{W}\times\mc{{W}}\rightarrow\mathbb{R}$ is continuous.\label{ass:K_H}
\end{assumption}
In Lemma~\ref{lemma:i0t_bounded} of Appendix~\ref{sec:proof_lemmas} we show that under this assumption, the operator $I_0^{op}\circ\mc{T}\vert_\mc{H}: \mc{H} \rightarrow \mc{L}^2_2{(P_Z)} $ is bounded and linear, where $I_0^{op}:\mc{L}^2(P_{Z})\rightarrow\mc{L}^2(P_{Z})$ is the operator $\left(I_0^{op}g\right)(z):=I_0(x)\cdot g(z)$ where $ z=(x,z')$, $\mc{T}\vert_\mc{H}:\mc{H} \rightarrow \mc{L}^2(P_{Z})$ is the restriction of the conditional mean operator to $\mc{H}$, and the boundedness of $I_0^{op}\circ\mc{T}\vert_\mc{H}$ is with respect to the topology defined by the norm of $\mc{H}$. The boundedness and linearity of $I_0^{op}\circ\mc{T}\vert_\mc{H}$ implies that its null space is closed and linear, so the projection of any element of $\mc{H}_0$ onto it exists. The residual from such projection is the unique element of $\mc{H}_0$ with minimum $\mc{H}$--norm.

The constant $B$ in the class $\mc{H}_B$ in the definition of $\tilde{h}_{KRAS}$ is any constant satisfying $\Vert h_{0,KRAS}^\dag \Vert_\mc{H} < B$.

It is important to note that under Assumption~\ref{ass:K_H}, the integral operator $T_{\mc{H}}:\mathcal{L}^2(P_{W})\rightarrow\mathcal{L}^2(P_{W})$ defined for any $h \in \mathcal{L}^2(P_{W})$, as 
\begin{align}
    T_{\mc{H}}h(w)&:=\int K_{\mc{H}}\left(w,\tilde w\right)h(\tilde w)dP_{W}(\tilde w) \label{def:int_op}
\end{align}
is a self-adjoint, Hilbert-Schmidt operator (see Theorem 2.34 of \cite{carrasco2007linear}), and therefore compact, with respect to the topology defined by the $\mc{L}^2(P_{W})$--norm in the domain and the image. On the other hand, it is well known that the range of $T_{\mc{H}}$ is included in $\mc{H}$. However, the operator $T_{\mc{H}}:\mathcal{L}^2(P_{W})\rightarrow \mc{H}$ is not compact with respect to the topology defined by the $\mc{H}$--norm in the image.


We introduce our two source conditions by first noting that the surrogate regularized problem that the KRAS estimator $\tilde h_{KRAS}$ targets,  where for simplicity we take $c=1$, is 
    \begin{align}
h_{\lambda_{\mc{H}},KRAS}^\dag
: =&\arg\min_{h\in\mc{H}_B}\left\{\frac{1}{4}\left\Vert\rho- I_0^{op} \circ \mc{T}\vert_\mc{H} \:h\right\Vert_2^2+\lambda_{\mc{H}}\left\Vert h\right\Vert_{\mc{H}}^2\right\}\nonumber\\
 =&\arg\min_{h\in\mc{H}}\left\{\frac{1}{4}\left\Vert\rho- I_0^{op} \circ \mc{T}\vert_\mc{H} \:h\right\Vert_2^2+\lambda_{\mc{H}}\left\Vert h\right\Vert_{\mc{H}}^2\right\}.
\label{argmin_h}
    \end{align}    
We remark that because $\mc{H}_B$ is a strict subset of the Hilbert space $\mc{H}$, in principle, the argmin over $\mc{H}_B$ need not exist. However, in Lemma~\ref{lemma:h_lambda} of Appendix~\ref{sec:proof_lemmas} we show that the argmin over $\mc{H}$ falls in $\mc{H}_B$ and consequently the argmin over $\mc{H}_B$ exists because it agrees with the argmin over $\mc{H}$, thus justifying the second equality in~\eqref{argmin_h}. We emphasize that the argmin over $\mc{H}$ exists because $\mc{H}$ is a Hilbert space (see the discussion preceding Theorem~\ref{thm:source-general}).  

Now, in Lemma~\ref{lemma:i0t_bounded} of Appendix~\ref{sec:proof_lemmas} we show that the adjoint $(I_0^{op}\circ \mc{T}\vert_\mc{H})^*:\mc{L}^2(P_Z) \rightarrow \mc{H}$ of $I_0^{op}\circ \mc{T}\vert_\mc{H}$ is given by $T_{\mc{H}} \circ\mc{\mc{T}}^* \circ I_0^{op}$ where, recall, $\mc{T}^*:\mc{L}^2(P_Z) \rightarrow \mc{L}^2(P_W)$ is the adjoint of the conditional mean operator $\mc{T}:\mc{L}^2(P_W) \rightarrow \mc{L}^2(P_Z)$.

Identifying in Theorem \ref{thm:source-general} the operator $\mc{A}$ with $I_0^{op} \circ \mc{T} \vert_\mc{H}$, the target function $h_0$ with $h_{0,KRAS}^\dag$, and the regularized solution $h_\lambda$ with $h_{\lambda_{\mc{H}},KRAS}^\dag$, and noting that for any $h \in \mc{H}$, $\mc{T}^*\circ I_0^{op} \circ I_0^{op} \circ \mc{T}\vert_\mc{H} \:h=  \mc{T}^*\circ I_0^{op} \circ \mc{T} \: h $ naturally leads to the following source condition
\begin{assumption}
    There exists $\beta >0$ such that $  h_{0,KRAS}^\dag = (T_{\mc{H}}\circ\mc{T}^* \circ I_0^{op}\circ  \mc{T})^{\beta/2} \: h^{*}$ for some $h^{*} \in \mc{H}$. Equivalently, 
    \begin{equation}
    h_{0,KRAS}^\dag = (T_{\mc{H}}\circ\mc{T}^* \circ I_0^{op}\circ \mc{T})^{\beta/2} \: T_\mc{H}^{1/2} \: h^{**} \quad \text{for some } \: h^{**} \in \mc{L}^2(P_W).
    \label{eq:source-kras-basic}
    \end{equation}
    \label{ass:source-basic}
\end{assumption}
  To motivate our second source condition, consider the operator $T_\mc{H}^{1/2} :\mc{L}^2(P_W) \rightarrow \mc{H}$. This is a surjective operator that is bounded when $\mc{H}$ is endowed with the topology induced by the RKHS inner product. Furthermore, when restricted to the orthogonal complement to the null space of $T_\mc{H}$, the operator $T_\mc{H}^{1/2}$ is an isometry. That is, $\Vert T_\mc{H}^{1/2} h \Vert_{\mc{H}} = \Vert h \Vert_2$ for all $h\in \mc{N}(T_\mc{H})^\perp$ (see Lemma~\ref{lemma:null_space} (d) in Appendix~\ref{sec:proof_lemmas}). In addition, the linear operator $\widetilde{\mc{T}}:=I_0^{op} \circ \mc{T}\circ T_\mc{H}^{1/2}:\mc{L}^2(P_W)\rightarrow\mc{L}^2(P_Z)$ is bounded (see Lemma~\ref{lemma:i0t_bounded} in Appendix~\ref{sec:proof_lemmas}). Then, letting $h_{\lambda_{\mc{H}},KRAS}':=\arg\min_{h\in\mathcal{L}^2(P_{W})}\{\frac{1}{4}\Vert\rho-\widetilde{\mc{T}}h\Vert_2^2+\lambda_{\mc{H}}\Vert h\Vert_2^2\}$, we have that
 \begin{align*}
 h_{\lambda_{\mc{H}},KRAS}^\dag
&:=\arg\min_{h\in\mc{H}}\left\{\frac{1}{4}\left\Vert\rho- I_0^{op} \circ \mc{T}|_\mc{H}h\right\Vert_2^2+\lambda_{\mc{H}}\left\Vert h\right\Vert_{\mc{H}}^2\right\}\\
&=T_\mc{H}^{1/2}\left[\arg\min_{h'\in\mc{N}(T_\mc{H})^\perp}\left\{\frac{1}{4}\left\Vert\rho- I_0^{op} \circ \mc{T}\circ T_\mc{H}^{1/2}h'\right\Vert_2^2+\lambda_{\mc{H}}\left\Vert T_\mc{H}^{1/2}h'\right\Vert_{\mc{H}}^2\right\}\right]\\
     &=T_\mc{H}^{1/2}\left[\arg\min_{h'\in\mc{N}(T_\mc{H})^\perp}\left\{\frac{1}{4}\left\Vert\rho- \widetilde{\mc{T}}h'\right\Vert_2^2+\lambda_{\mc{H}}\left\Vert h'\right\Vert_2^2\right\}\right]\\
      &=T_\mc{H}^{1/2}\left[\arg\min_{h'\in\mc{L}^2(P_W)}\left\{\frac{1}{4}\left\Vert\rho- \widetilde{\mc{T}}h'\right\Vert_2^2+\lambda_{\mc{H}}\left\Vert h'\right\Vert_2^2\right\}\right]:=T_\mc{H}^{1/2}h_{\lambda_{\mc{H}},KRAS}',
\end{align*}
where the second to last equality follows because the minimum over $h' \in \mc{L}^2(P_W)$ of the expression in curly brackets is attained at an element of $\mc{N}(T_\mc{H})^\perp$. Note that the first line in the display indicates that the KRAS regularized problem combines a $\mathcal{L}^2$--loss with an RKHS-norm penalization. The second to last equality establishes that, by applying a suitable transformation to the underlying operator $\mc{T}$, we can rewrite the KRAS problem so that the penalization is also in $\mathcal{L}^2$. This reframing makes the KRAS estimators and the LRAS estimators directly comparable as solutions to regularization problems with the same norm structure but different operators. Specifically, $\tilde{h}_{KRAS}$ is the image---under the isometry $T_\mathcal{H}^{1/2}$---of the Tikhonov-regularized solution $h_{\lambda_{\mc{H}},KRAS}'$ (with an $\mathcal{L}^2$ penalty) to the operator equation
\begin{equation} \widetilde{\mathcal{T}} h = \rho. \label{eq:rkhs_int_eq_h}
\end{equation}
This observation has two important implications. First, since $T_\mathcal{H}^{1/2}$ is a compact operator (see Lemma~\ref{lemma:null_space} (f) in Appendix~\ref{sec:proof_lemmas}), regularization in the RKHS norm---as used in the computation of the KRAS estimators---amounts to solving a more severely ill-posed problem than the original integral equation~\eqref{eq:int_eq_h}, because the modified operator $\widetilde{\mathcal{T}}$ smooths more aggressively than $\mathcal{T}$. Second, by assuming a $\beta$--source condition on the minimal $\mathcal{L}^2(P_{W})$--norm solution $h_{0,KRAS}'$ to equation~\eqref{eq:rkhs_int_eq_h}, we can derive bounds on the regularization bias $\| h_{\lambda_{\mc{H}},KRAS}' - h_{0,KRAS}' \|_2$.  Lemma~\ref{lemma:min_rkhs_norm_sol_h} in Appendix~\ref{sec:proof_lemmas} establishes that $h_{0,KRAS}^\dag=T_\mathcal{H}^{1/2}h_{0,KRAS}'$. This then implies that with such source condition we can also control the regularization bias $\| h_{\lambda_{\mc{H}},KRAS}^\dag - h_{0,KRAS}^\dag \|_2$ of the regularized solution to the original problem $\mc{T}I_0h=\rho(Z)$. 

The first point is key to understanding the distinction between RMSE bounds for LRAS, i.e. RAS estimators over RKHS classes $\mc{H}$ and $\mc{G}$ that use $\mathcal{L}^2(\mathbb{P}_n)$ penalization, versus KRAS estimators, i.e. those that employ RKHS penalization—a topic we return to in Section~\ref{sec:compare_ras}. The second point, in turn, motivates our announced second source condition. To state it, we first note that the adjoint $\widetilde{\mc{T}}^*:\mc{L}^2(P_Z)\rightarrow\mc{L}^2(P_W)$ of $\widetilde{\mc{T}}$ satisfies $\widetilde{\mc{T}}^* = T_\mc{H}^{1/2}  \circ \mc{T}^*\circ I_0^{op}$. 
    \begin{assumption}
    There exists $\beta>0$ and $h^{**}\in\mathcal{L}^2(P_{W})$ such that 
    \begin{align}
h_{0,KRAS}^\dag =T^{1/2}_{\mc{H}}\circ\left(\widetilde{\mc{T}}^*\circ\widetilde{\mc{T}}\right)^{\beta/2} h^{**}=T^{1/2}_{\mc{H}}\circ \left( T_\mc{H}^{1/2} \circ \mc{T}^*\circ  I_0^{op} \circ  \mc{T} \circ  T_\mc{H}^{1/2}\right)^{\beta/2} \: h^{**}.
\label{eq:source_kras}
    \end{align}
     \label{ass:source}
\end{assumption} 
Assumption~\ref{ass:source} is tantamount to requiring the standard $\beta$--source condition on $h_{0,KRAS}'$ with respect to the operator $\widetilde{\mc{T}}$. When $\widetilde{\mc{T}}$ is compact---a property that is satisfied when $W$ and $Z$ do not share common components---this requires that $h_{0,KRAS}'$ decays rapidly in the directions corresponding to small singular values of $\widetilde{\mc{T}}$, with larger values of $\beta$ implying greater smoothness and better alignment with the stable directions of $\widetilde{\mc{T}}$. The final application of  $T^{1/2}_{\mc{H}}$ in Assumption \ref{ass:source} further smooths $h_{0,KRAS}'$ with respect to the singular value decomposition of the integral operator $T_{\mc{H}}$. Thus, Assumption \ref{ass:source} encodes smoothness of $h_{0,KRAS}^\dag$ relative to both the operator $\widetilde{\mc{T}}$ and the geometry of the RKHS $\mc{H}$. 

As we will see in the next theorem, Assumptions \ref{ass:source-basic} and \ref{ass:source} lead to identical bounds on the regularization bias. The assumptions coincide when $\beta = 2$ but impose different smoothness requirements otherwise. We find Assumption \ref{ass:source} particularly useful for comparisons with RAS estimators that penalize using the $\mc{L}^2(\mathbb{P}_n)$--norm, as it posits a $\beta$--source condition related to the operator $\widetilde{\mc{T}}$ acting between $\mathcal{L}^2$ spaces, rather than for an operator whose domain is the RKHS $\mc{H}$.

\begin{theorem}[Regularization Bias]
     Suppose that Assumptions \ref{ass:relizability} and \ref{ass:K_H} hold. If, in addition, Assumption \ref{ass:source-basic} or Assumption \ref{ass:source} hold, then there exist constants $c_1^*$ and $c_2^*$ such that 
    \begin{align*}
        \Vert  h_{\lambda_{\mc{H}},KRAS}^\dag - h_{0,KRAS}^\dag \Vert_{\mc{H}}^2 \le c_1^* \: \lambda_{\mc{H}}^{\min\left\{\beta,2\right\}} \:
        \text{and}\:
        \Vert \mc{T}I_0(h_{\lambda_{\mc{H}},KRAS}^\dag - h_{0,KRAS}^\dag )\Vert_2^2\le c_2^* \:\lambda_{\mc{H}}^{\min\left\{\beta+1,2\right\}}.
    \end{align*}
    \label{theo:biascontrol}
\end{theorem}

Under Assumption~\ref{ass:source-basic}, Theorem~\ref{theo:biascontrol} follows by a direct application of Theorem~\ref{thm:source-general}. In Appendix~\ref{sec:proofs_theorems}, we provide the proof of Theorem~\ref{theo:biascontrol} under Assumption~\ref{ass:source}. Having established conditions under which the regularization bias can be controlled, we now turn to the additional assumptions required to bound the estimation error. Our derivation follows the framework of \citet{bennett2023source}, and therefore relies on conditions closely aligned with theirs. The resulting bounds will be expressed in terms of the critical radius of appropriately defined subsets derived from $\mathcal{G}$. In the course of our derivations, we identified a few instances---particularly in the application of a Lemma in \citet{foster2023orthogonal} (see Appendix~\ref{sec:proof_lemmas})---where the assumptions made by \citet{bennett2023source} on these subsets can be slightly relaxed. The first such assumption, stated below, slightly weakens a similar condition made by all existing analyses of RAS estimators as we discuss subsequently:

 \begin{assumption}[Local closeness and boundedness of $\mc{T}$] There exists $B>\big\Vert h_{0,KRAS}^\dag\big\Vert_{\mc{H}}$ and $B'>0$ such that if $\Vert h\Vert_{\mc{H}}\le 2B$ then $\mc{T}h\in\mc{G}$ and $\left\Vert\mc{T}h\right\Vert_{\mc{G}}\le B'$.
\label{ass:closeness}
\end{assumption}

The local closeness condition, i.e., that $\mc{T}h\in\mc{G}$ for $h$ in some sufficiently large RKHS ball is used in the derivation of the error bounds to connect $h_{\lambda_{\mc{H}},KRAS}^\dag$ with the population version of the regularized optimization problem targeted by the KRAS estimation procedure. Specifically, by part 1 of Proposition \ref{prop:innermax}, it follows that under this assumption,
\begin{equation}      
h_{\lambda_{\mc{H}},KRAS}^\dag=\arg \min_{h\in\mc{H}}\max_{g\in\mc{G}}\E\left\{ m(O;g) -I_0(X)r(O)h(W)g(Z)-c^2\cdot I_0(X)g(Z)^2
\right\}+\lambda_{\mc{H}}\Vert h\Vert_{\mc{H}}.\label{eq:h_star}
        \end{equation}
   
        Note that the right-hand side of~\eqref{eq:h_star} is a population optimization problem that parallels the empirical optimization problem solved by the proposed KRAS estimator $\tilde{h}_{KRAS}$, except that the estimation also includes the penalty term $\lambda_{\mc{G}}\Vert g\Vert_{\mc{G}}^2$ in the inner maximization part.
        
        The boundedness condition in Assumption~\ref{ass:closeness}, i.e. that if $\Vert h\Vert_{\mc{H}}\le 2B$ then $\mc{T}h\in\mc{G}$ and $\left\Vert\mc{T}h\right\Vert_{\mc{G}}\le B'$, is invoked to ensure that $\Vert \mc{T} (\tilde {h}_{KRAS} - h_{0,KRAS}^\dag) \Vert_\mc{G} = O(1)$ which, in turn, is used to control the influence of the penalization $\lambda_\mc{G}\: \Vert g \Vert_\mc{G}^2$ of the inner maximization in the computation of~\eqref{def:h_tilde}. In their analysis of the KRAS estimator, \citet{dikkala2020minimax} and \citet{ghassami2022minimax} imposed the related global closeness that $\mathcal{T}(h-h_{0,KRAS}^\dag) \in \mathcal{G}$ for all $h \in \mathcal{H}$---together with the related Lipchitz/boundedness condition that $\Vert \mathcal{T}(h-h_{0,KRAS}^\dag) \Vert_\mc{G} \le C \Vert h-h_{0,KRAS}^\dag \Vert_\mc{H}$ for some constant $C$. \citet{bennett2023source} also adopted a similar condition when deriving error bounds for the LRAS estimator. When, as we assume throughout, $h_{0,KRAS}^\dag \in \mc{H}$ this requirement is stronger than Assumption~\ref{ass:closeness}. \citet{dikkala2020minimax} discuss conditions under which their assumption, and consequently our Assumption~\ref{ass:closeness}, holds. Nonetheless, Assumption~\ref{ass:closeness} remains, in our view, the most conceptually opaque and difficult to justify among those used in the analysis of RAS estimators.

We will also make the following assumption:
\begin{assumption}
$\mc{Z}$ is compact and the kernel function $K_{\mc{G}}:\mc{Z}\times\mc{{Z}}\rightarrow\mathbb{R}$ is continuous 
\label{ass:K_G}
\end{assumption}

Assumptions~\ref{ass:K_H} implies that $k_1:=\{\sup_{w\in\mc{W}}K_{\mc{H}}(w,w)\}^{1/2} < \infty$ and Assumption~\ref{ass:K_G} that $k_2:=\{\sup_{z\in\mc{Z}}K_{\mc{G}}(z,z)\}^{1/2} < \infty$. An important implication that is used repeatedly in the derivation of the error bounds, is that, if we let 
$\mc{F}_B := \left\{f \in \mc{F} \: : \: \Vert f \Vert_{\mc{F}} \le B \right\}$ where $\mc{F}$ is either $\mc{H}$ or $\mc{G}$ and $\Vert f \Vert_{\mc{F}}$ is the norm in the corresponding RKHS, then since $\Vert f\Vert_\infty\le\Vert f\Vert_\mc{F}\:\{\sup_{v\in\mc{V}}K_\mc{F}(v,v)\}^{1/2}$, the class $\mc{F}_B$ is \textit{uniformly bounded} in the sense that $\Vert f\Vert_\infty\le C$ for some universal constant $C$ and for all $f\in\mc{F}_B$.

To state the results that provide the announced finite-sample bounds, we introduce additional notation and review important concepts from the theory of concentration inequalities. Let $V$ denote a random vector and let $\mc{F}$ denote a class of real-valued measurable functions $f:\mc{V}\rightarrow\mathbb{R}$ included on a normed space. The class $\mc{F}$ is called \textit{star-shaped} if for any $f\in\mc{F}$ and $\alpha\in[0,1]$, the scaled function $\alpha f$ also belongs to $\mc{F}$. We define $star\left(\mc{F}\right):=\left\{\alpha f\: : \: f\in\mc{F},\alpha\in[0,1]\right\}$, which is the smallest star-shaped class that contains all the functions in $\mc{F}$. For any function $f^*\in\mc{F}$, we define  $\mc{F}-f^*:=\left\{f-f^*\: : \: f\in\mc{F}\right\}$.

For a given $\delta>0$, the \textit{localized Rademacher complexity of radius} $\delta$ of a function class $\mc{F}$ is defined as
\[
\mc{R}_n\left(\mc{F},\delta\right):=\mathbb{E}_{\varepsilon,V_1,\dots,V_n}\left[\sup_{f\in\mc{F}:\Vert f\Vert_2\le\delta}\left\vert\frac{1}{n}\sum_{i=1}^n\varepsilon_if(V_i)\right\vert\right],
\]
where $V_1,\dots,V_n$ are independent and identically distributed copies of a random variable $V$ and $\varepsilon_1,\dots,\varepsilon_n$ are independent Rademacher random variables \citep{bartlett2005local}, independent of $V_1,\dots,V_n$. The localized Rademacher complexity of the function class $\mc{F} - f^*$ is a measure of the complexity of a class of functions within a neighborhood of a given target function $f^*$. This measure focuses on functions that are \textit{close} to the target, thereby avoiding unnecessary increases in complexity by ignoring functions that are far from the target. Intuitively, this measure captures the capacity of the functions in $\mc{F}$ in a neighborhood of $f^*$ to fit noise sequences of size $n$.

The \textit{critical radius} of a function class $\mc{F}$, uniformly bounded by $b$, is defined as the smallest positive radius $\delta$ for which the following inequality holds $\mc{R}_n\left(\mc{F},\delta\right)\le\frac{\delta^2}{b}$. We are now ready to state our first main result.

\begin{theorem}
Suppose  $m(o;g) = I_0(x)s(o) g(z)$ for some fixed real-valued function $s$ satisfying $\Vert s \Vert_{\infty} < \infty$. Suppose Assumptions \ref{ass:relizability}, \ref{ass:K_H}, \ref{ass:closeness}, \ref{ass:K_G} hold and either Assumption \ref{ass:source-basic} or \ref{ass:source} hold when $P=P_0$. Suppose that the constant $B$ in Assumption \ref{ass:closeness} is taken large enough so that $\big\Vert h_{0,KRAS}^\dag\big\Vert_{\mc{H}}\le B$. Let $B_1=B'/c^2$. Define the function class $ 
I_0\cdot\mc{G}_{B_1}:=\{I_0\cdot g\: :\: g\in\mc{G}_{B_1}\}$ 
    and let $\delta_n$ be an upper bound on the critical radius of $I_0\cdot\mc{G}_{B_1}$. If $\lambda_{\mc{H}}\le 1$, and $\lambda_\mc{G}\ge157c^2\delta_n^2/B_1^2$, then with probability at least $\kappa(\delta_n) :=1-[36+18\cdot I\{\delta_n < 5b/6 \}\cdot\lfloor\{\log(6/5)\}^{-1}\log(b/\delta_n)\rfloor]\exp(-c_1n\delta_n^2)$ 
    where $c_1=\{8b^2(1+17e)\}^{-1}$ and  $b=k_2B_1$,  
    it holds that $\hat{h}_{KRAS} := I_0\:\tilde{h}_{KRAS}$ and $h_{0,KRAS}:=I_0\:h_{0,KRAS}^\dag$ satisfy
  \begin{equation}
      \left\Vert \hat {h}_{KRAS}-h_{0,KRAS}\right\Vert_2^2\le c_2\left(\frac{\delta_n^2}{\lambda_{\mc{H}}}+\frac{\lambda_{\mc{G}}}{\lambda_{\mc{H}}}+\lambda_{\mc{H}}^{\min\left\{\beta,1\right\}}\right),\label{eq:strong_error}
  \end{equation}
    and
    \begin{equation}
         \left\Vert\mc{T} \left(\hat {h}_{KRAS}-h_{0,KRAS}\right)\right\Vert_2^2\le c_3\left(\delta_n^2+\lambda_{\mc{G}}+\lambda_{\mc{H}}^{\min\left\{\beta+1,2\right\}}\right),\label{eq:weak_error}
    \end{equation}
    where $c_2$ and $c_3$ are universal constants that depend neither on $n$ nor on $\delta_n$. \label{theo:conv}
\end{theorem}

Notice that the bounds in~\eqref{eq:strong_error} and~\eqref{eq:weak_error} depend on the critical radius of only one function class, namely the class $I_0\cdot\mc{G}_{B_1}$. When $I_0=1$, this reduces to the critical radius of $\mc{G}_{B_1}$ itself, which is well understood for many function classes, including various RKHSs; see, for example, Chapters 13 and 14 of \cite{wainwright2019high}. When $I_0$ is not identically 1, and $\mc{G}$ is an RKHS generated by a kernel $K_{\mc{G}}$ satisfying Assumption~\ref{ass:K_G}, the following result shows that the critical radius of $I_0 \cdot \mc{G}_{B_1}$ remains of the same order as that of $\mc{G}_{B_1}$.

\begin{lemma}
  Let $\mc{F}$ denote a class of real-valued measurable functions $f:\mc{O}\mapsto\mathbb{R}$ such that $\Vert f\Vert_\infty\le b$ for all $f\in\mc{F}$. Let $\delta'_n$ and $\delta_n$ denote the critical radii of $I_0\cdot \mc{F}:=\{I_0\cdot f \::\: f\in\mc{F}\}$ and $\mc{F}$, respectively, where $I_0$ is the indicator that $X\in\mc{X}_0$ and $X$ is a subvector of $O$. Assume that $\delta_n^2=o(\exp\{-dn\})$ for any fixed constant $d>0$ and $\delta_n=O(\alpha_n)$ where $\alpha_n$ decreases with $n$. Then, $\delta'_n\le O(\alpha_{\lfloor n\mu_0/2\rfloor})$ where $\mu_0:=P(X\in\mc{X}_0)$.
    \label{lemma:crit_radius_I0G}
\end{lemma}

Theorem~\ref{theo:conv} provides bounds on the weak error $\Vert \mc{T}(\hat {h}_{KRAS}-h_{0,KRAS})\Vert_2$ as well as the RMSE, $\Vert \hat {h}_{KRAS}-h_{0,KRAS}\Vert_2$, for the special case in which $m(o;g)=I_0(x)s(o)g(z)$. Our next result gives bounds on the weak error and the RMSE when the map $g\mapsto m(o;g)$ is an arbitrary linear map for each fixed $o\in\mc{O}$ satisfying the following additional condition:

\begin{assumption}
\begin{enumerate}
\item[]
    \item 
    There exists $B_2>0$ such that, for any $g\in\mc{L}^2(P_{Z})$, it holds that $
\left\Vert m(\:\cdot\:;g)\right\Vert_2\le B_2 \: \left\Vert I_0 \:g\right\Vert_2.
$
\item The function class $
m\circ \mc{G}:=\left\{o \mapsto m(o;g)\::\:g\in\mc{G}\right\}$
can be equipped with a metric $d$ under which it is separable. Moreover, $d$ is controlled by the RKHS norm in the sense that $d(m(\:\cdot\:;g),m(\:\cdot\:,g'))\le \Vert g-g'\Vert_\mc{G}$ for all $g,g'\in\mc{G}$
and it dominates the metric induced by the sup-norm on $\mc{O}$, meaning there exists a constant $B_3>0$ such that
$
\sup_{o\in\mc{O}}|f(o)-f'(o)|\le B_3\cdot d(f,f')\quad\text{for all } f,f'\in m\circ\mc{G}.
$

\end{enumerate}
\label{ass:bounded}
\end{assumption}

In Appendix~\ref{sec:proofs_theorems}, we verify that Assumption~\ref{ass:bounded} holds in Example~\ref{example_mtp} under suitable conditions of the map $q$ and the data-generating process.

\begin{theorem} Suppose Assumptions~\ref{ass:relizability}, \ref{ass:K_H}, \ref{ass:closeness}, \ref{ass:K_G}, \ref{ass:bounded} hold and either Assumption~\ref{ass:source-basic} or~\ref{ass:source} hold for $P=P_0$. Suppose that the constant $B$  in Assumption~\ref{ass:closeness} is taken large enough so that $\big\Vert h_{0,KRAS}^\dag\big\Vert_{\mc{H}}\le B$. Let $B_1=B'/c^2$. Define
$
m\circ \mc{G}_{B_1}:=\left\{o \mapsto m(o;g)\::\:g\in\mc{G}_{B_1}\right\},
$
and let the function class $I_0\cdot\mc{G}_{B_1}$ be defined as in Theorem~\ref{theo:conv}. Let $\delta_n$ be an upper bound on the critical radius of $I_0\cdot\mc{G}_{B_1}$ and of $star(m\circ \mc{G}_{B_1})$. If $\lambda_{\mc{H}}\le 1$, and $\lambda_\mc{G}\ge157c^2\delta_n^2/B_1^2$, then inequalities~\eqref{eq:strong_error} and~\eqref{eq:weak_error} hold for some universal constants $c_2$ and $c_3$, with probability at least as $\kappa(\delta_n)$ as given in Theorem \ref{theo:conv} with $c_1$ as defined in that Theorem, but with $b:=B_1\max\{k_2,B_3\}$, where $B_3$ is the constant specified in Assumption~\ref{ass:bounded}.
    \label{theo:main}
\end{theorem}

The distinction between Theorems~\ref{theo:conv} and~\ref{theo:main} is that in Theorem~\ref{theo:conv}, $\delta_n$ appearing in the bounds~\eqref{eq:strong_error} and~\eqref{eq:weak_error} is an upper bound on the critical radius of the class $I_0\cdot\mc{G}_{B_1}$, whereas in Theorem~\ref{theo:main}, $\delta_n$ stands for an upper bound of the critical radius of  $I_0\cdot\mc{G}_{B_1}$ and  $star(m\circ\mc{G}_{B_1})$. Lemma~\ref{lemma:crit_radius_I0G} establishes that under fairly weak conditions, the critical radius of $I_0\cdot\mc{G}_{B_1}$ is dominated by the critical radius of $\mc{G}_{B_1}$. On the other hand, the critical radius of the function class $star(m\circ\mc{G}_{B_1})$ must be analyzed on a case-by-case basis. In the special case  $m(o;g)=g\{q(z)\}$ for a measurable map $q:\mc{Z}\rightarrow\mc{Z}$, the following result shows that the critical radius of $\mathrm{star}(m\circ\cdot \mc{G}_{B_1})$ is dominated by that of $\mc{G}_{B_1}$.

\begin{lemma}
Suppose $m(o;g)=g\{q(z)\}$ for some measurable map $q:\mc{Z}\rightarrow\mc{Z}$, where $z$ is a subvector of $o$. Suppose that $\mc{G}_{B_1}$ is star-shaped. Then $star(m\circ\mc{G}_{B_1})$ is also star-shaped. Furthermore, letting $\delta_n$ and $\delta'_n$ denote the critical radii of $(m\circ\mc{G}_{B_1})$ and $\mc{G}_{B_1}$, respectively, there exists a universal constant $c_1''$ such that $ 
    \delta_n\le c_1''\delta'_n$.\label{lemma:crit_radius_mG}
\end{lemma}

\begin{remark}
When $\delta_n\ge 5b/6$ for $b=k_2B_1$ ($=B_1\max\{k_2,B_3\}$), the bounds on $\Vert \hat {h}_{KRAS}-h_{0,KRAS}\Vert_2^2$ and $\Vert \mc{T}(\hat {h}_{KRAS}-h_{0,KRAS})\Vert_2^2$ provided in Theorem~\ref{theo:conv} (Theorem~\ref{theo:main}) hold with probability at least $1-36\exp\left(-c_1n\delta_n^2\right)$. On the other hand, if $ c_0\log(\log(n))/n\le \delta_n^2<\left(5b/6\right)^2$ with $c_0=\max\{b^2/\log(\log(3)),2/c_1\}$, and $n\ge 3$, the bounds in Theorem~\ref{theo:conv} and Theorem~\ref{theo:main} remain valid with probability at least $1-18\cdot \{\log(6/5)\}^{-1}\exp\left(-\frac{c_1}{2}n\delta_n^2\right)$, (see Appendix~\ref{sec:proofs_theorems}), which tends to 1 as $n \rightarrow \infty$ because $n \:\delta_n^2 \ge c_0 \:\log(\log(n))\rightarrow\infty$. The condition $c_0\log(\log(n))/n\le \delta_n^2$ holds for a sufficiently large $n$, in particular, when $\mc{G}$ is an RKHS with Gaussian kernel because in such case $\delta_n^2 = O(\log(n)/n)$.
\end{remark}
Theorems~\ref{theo:conv} and~\ref{theo:main} have the following important corollary. 

\begin{corollary}
   Suppose that the conditions of Theorems~\ref{theo:conv} or \ref{theo:main} hold. If the penalization parameters are chosen as $\lambda_{\mc{H}}=O(\delta_n^{2/(1+\min\{\beta,1\})})$ and $\lambda_{\mc{G}}=O\left(\delta_n^2\right)$, then
    \[
    \left\Vert \hat {h}_{KRAS}-h_{0,KRAS}\right\Vert_2=O_p\left(\delta_n^{\frac{\min\left\{\beta,1\right\}}{1+\min\left\{\beta,1\right\}}}\right)\quad\text{and}\quad \left\Vert \mc{T}(\hat {h}_{KRAS}-h_{0,KRAS})\right\Vert_2=O_p\left(\delta_n\right).
    \] \label{coro:optimal_lambda}
\end{corollary}
Recall from Theorem~\ref{theo:asymp_distr} that in order for the debiased machine learning estimator $\hat\theta$ of $\theta(P_0)$ to be asymptotically linear, the remainder term $R_n$ must be of order $o_p(n^{-1/2})$. In view of Corollary~\ref{coro:optimal_lambda}, it follows that if: (i) $\tilde h$ and $\tilde g$ are the KRAS estimators $\tilde {h}_{KRAS}$ and $\tilde {g}_{KRAS}$, (ii) the conditions of Theorems~\ref{theo:conv} or~\ref{theo:main} hold for the problem of estimating the solutions $h_{0,KRAS}:=I_0 \: h_{0,KRAS}^\dag$ and $g_{0,KRAS}:=I_0\:g_{0,KRAS}^\dag$ of equations~\eqref{eq:int_eq_h} and~\eqref{eq:int_eq_g}, (iii) $h_{0,KRAS}^\dag$ and $g_{0,KRAS}^\dag$ satisfy Assumption \ref{ass:source-basic} or Assumption \ref{ass:source} (for their corresponding operators) with parameters $\beta_h^{KRAS}$ and $\beta_g^{KRAS}$, (iv) the relevant critical radius in Theorems~\ref{theo:conv} or~\ref{theo:main} for the classes used in the estimation of $h_{0,KRAS}$ and $g_{0,KRAS}$ are the same (up to multiplicative constants) and equal to universal constants times $\delta_n$, and (v) the tuning parameters are selected as indicated in Corollary~\ref{coro:optimal_lambda}, then 
\begin{align*}
R_n = O(\min \{& \Vert \hat {h}_{KRAS}-h_{0,KRAS}\Vert_2 \left\Vert\mc{T}^*( \hat {g}_{KRAS}-g_{0,KRAS})\right\Vert_2,\\
&\Vert \hat {g}_{KRAS}-g_{0,KRAS}\Vert_2 \Vert\mc{T}( \hat {h}_{KRAS}-h_{0,KRAS})\Vert_2 \})= O_p\left(\delta_n^{1+\frac{\beta^{KRAS}}{1+\beta^{KRAS}}}\right),
\end{align*}
where $\beta^{KRAS}=\min\{\max\{\beta_h^{KRAS},\beta_g^{KRAS}\}, 1\}$. In particular, if the RKHSs used in the internal maximization in both estimation problems are Gaussian, then $\delta_n = O(\sqrt{\log(n)/n})$, which yields $R_n = o_p(n^{-1/2})$. In contrast, if they are $d$-dimensional RKHSs with Sobolev kernels of order $\nu=(l+d)/2$ where $l$ is a positive integer, then $\delta_n = O(n^{-\nu/(2\nu+d)})$, which then yields $R_n = o_p(n^{-1/2})$ whenever $d<l$ and $\beta\in(d/l,1]$. 
Thus, we observe double robustness in the smoothness of the minimum RKHS--norm targets $h_{0,KRAS}^\dagger$ and $g_{0,KRAS}^\dagger$. Specifically, even if one of $h_{0,KRAS}^\dag$ or $g_{0,KRAS}^\dag$ has low smoothness (i.e., the corresponding parameter of the source condition in Assumption \ref{ass:source-basic} or Assumption \ref{ass:source} is small), the remainder term $R_n$ can still be controlled as long as the other nuisance is sufficiently smooth. 

As in the work of \cite{bennett2023source}, we could in principle improve upon the KRAS estimators $\tilde {h}_{KRAS}$ and $\tilde {g}_{KRAS}$ by using an iterated Tikhonov strategy. The details of such procedure will be reported elsewhere. We conclude with a final comment on the computational implementation of our proposed estimator to ensure it satisfies the conditions required by Theorems~\ref{theo:conv} and~\ref{theo:main}.
\begin{remark}
    The errors bounds established in Theorems~\ref{theo:conv} and~\ref{theo:main} require that the estimator $\tilde {h}_{KRAS}$ has RKHS norm at most $B$, where $B$ is the constant specified in Assumption~\ref{ass:closeness}. While this constant is not typically known in practice, one can select a sufficiently large $B$ and solve the optimization problem in~\eqref{def:h_tilde} over the constrained class $\mc{H}_{B}$. The computation of the solution in the restricted class $\mc{H}_{B}$ reduces to the problem of optimizing a quadratic objective function in $\mathbb{R}_n$ subject to a quadratic constraint. As such, enforcement of the constraint $\Vert \tilde{h}_{KRAS} \Vert_{\mc{H}} < B$ can be conducted as follows. First a solution to the unconstrained optimization problem is found. If the solution satisfies the constraint, then this is the solution to the constrained problem. If not, the unconstrained optimization problem is altered as indicated in Exercise 4.22 of \cite{boyd2004convex}, and the unconstrained solution to this new problem agrees with the solution to the original constrained problem. 
\end{remark}

\section{Contrasting Kernel based Adversarial Estimation Strategies}\label{sec:rel_work}

In this section we investigate the advantages and limitations of the three prominent recently proposed kernel-based adversarial approaches:
(i) LRAS estimators, i.e. kernel adversarial stabilized estimators with $\mathcal{L}^2(\mathbb{P}_n)$ penalization;
(ii) KRAS estimators, i.e. kernel adversarial stabilized estimators with RKHS penalization; and
(iii) the kernel maximal moment restriction (KMMR) estimators studied in \cite{muandet2020kernel, mastouri2021proximal, kallus2021causal}; and \cite{zhang2023instrumental}.
We compare these methods in terms of the assumptions they impose, the scope and sharpness of their available finite-sample error bounds, and---most importantly---the nature of the inverse problems they implicitly solve. This perspective brings into focus the trade-offs between statistical convergence rates, computational stability, and the degree of ill-posedness each method must confront.
To keep the notation light, we assume throughout this section that $I_0 = 1$.

\subsection{$\mc{L}^2(\mathbb{P}_n)$-penalized Kernel Regularized Adversarial Stabilized Estimators}\label{sec:compare_ras}

An LRAS estimator using RKHS classes $\mc{H}$ and $\mc{G}$ is defined as:
\begin{align*}
\tilde h_{LRAS}:=\arg\min_{h\in\mc{H}_{(D)}}\max_{g\in\mc{G}}\Big[&\E_n\left\{m(O;g)-r(O)h(W)g(Z)-r(X)g(Z)^2\right\}\\
&-\lambda_{\mc{G}}\E_n\left\{ g(Z)^2\right\}+\lambda_{\mc{H}}\E_n\left\{ h(W)^2\right\}\Big]\nonumber,
\end{align*}
where $\mc{H}_{(D)}:=\{h\in\mc{H}:\Vert h\Vert\le D\}$ for some constant $D$ and some norm $\Vert \cdot \Vert$, and where we set $c=1$ for simplicity.

In their description of LRAS estimators for arbitrary classes $\mc{H}$ and $\mc{G}$, \cite{bennett2023source} state that the argmin is taken over a “norm-constrained” class, but do not specify which norm defines the constraint. At the same time, their main theorems establishing error bounds for LRAS estimators require the function class to be uniformly bounded, i.e., that there exists $B'$ such that $\sup_{h \in \mc{H}_{(D)}} \Vert h \Vert\infty < B'$.

When, as we shall assume here, $\mc{H}$ is an RKHS, this uniform boundedness is guaranteed if the norm used in the definition of $\mc{H}_{(D)}$ is either the supremum norm $\Vert \cdot \Vert_\infty$ or the RKHS norm $\Vert \cdot \Vert_\mc{H}$. By contrast, uniform boundedness is not ensured if $\mc{H}_{(D)}$ is defined using the $\mc{L}^2(P_W)$ or $\mc{L}^2(\mathbb{P}_n)$ norm. For this reason, in what follows we assume that the norm in $\mc{H}_{(D)}$ is either $\Vert \cdot \Vert_\infty$ or $\Vert \cdot \Vert_\mc{H}$.

When $\lambda_{\mc{G}}$ is negligible in relation to $\lambda_{\mc{H}}$, and the closeness condition~\eqref{condition:closeness} holds, the LRAS estimator aims at the solution to the surrogate regularized problem:
\begin{align*}
h_{\lambda_{\mc{H}},LRAS}:=\arg\min_{h\in\mc{H_{(D)}}}\Vert \mc{T}h - \rho \Vert_2^2 +\lambda_{\mc{H}}\Vert h \Vert_2^2.
\end{align*}
However, as indicated in our discussion in Section~\ref{sec:adversarial_estimators}, the argmin over $\mc{H}_{(D)}$ need not exist. Note that if the norm defining $\mc{H}_{(D)}$ is the $RKHS$ norm, we cannot ensure that $h_{\lambda_{\mc{H}},LRAS}$ exists even if a solution $h_0$ with minimal $\mc{L}^2(P_W)$--norm or a solution with minimal RKHS--norm falls in $\mc{H}_{(D)}$. When the norm in the definition of $\mc{H}_{(D)}$ is the sup norm, then it is even harder to justify the existence of $h_{\lambda_{\mc{H}},LRAS}$ since $\mc{H}$ is not a Hilbert space when endowed with the sup norm.

 \citet{bennett2023source} derived bounds for $\Vert \tilde {h}_{LRAS} - h_{0,LRAS} \Vert_2^2$ and $\Vert \mc{T}(\tilde {h}_{LRAS} - h_{0,LRAS}) \Vert_2^2$  for estimators with $\lambda_{\mc{G}}=0$ where $h_{0,LRAS}:=\arg\min_{h\in \mc{L}^2(P_W): \: h \:\text {solves~\eqref{eq:int_eq_h}}}\Vert h\Vert_2$ under the critical assumption: 
\begin{equation}
\arg\min_{h\in\mc{L}^2(P_W)}\Vert \mc{T}h - \rho \Vert_2^2 +\lambda_{\mc{H}}\Vert h \Vert_2^2 \: \in \: \mc{H}_{(D)} \quad \text{for all} \: \lambda_{\mc{H}} \:\: \text{ in a neighborhood of 0}.
\label{ass: relizability-Bennett}
\end{equation}
Under this assumption, $h_{\lambda_{\mc{H}},LRAS}$ exists and is equal to $\arg\min_{h\in\mc{L}^2(P_W)}\Vert \mc{T}h - \rho \Vert_2^2 +\lambda_{\mc{H}}\Vert h \Vert_2^2$. To invoke Theorem~\ref{thm:source-general} to derive bounds for the regularization errors $\Vert h_{\lambda_{\mc{H}},LRAS} - h_{0,LRAS}\Vert_2$ and $\Vert \mc{T}(h_{\lambda_{\mc{H}},LRAS} - h_{0,LRAS})\Vert_2$ these authors imposed the $\beta$--source condition that posits that there exists $\beta>0$ such that
\begin{equation}
    h_{0,LRAS}=\left(\mc{T}^*\mc{T}\right)^{\beta/2}h^{**} \quad \text{for some} \:h^{**} \in \mc{L}^2(P_W).
    \label{eq:source_bennett}
\end{equation}

Unlike for KRAS estimators, establishing bounds for LRAS does not require the realizability condition in Assumption~\ref{ass:relizability} that posits the existence of a solution to equation~\eqref{eq:int_eq_h} within $\mc{H}$. However, bounds for LRAS estimators rely critically on the validity of condition~\eqref{ass: relizability-Bennett}---a condition that the analysis of KRAS estimators does not impose. Thus, the analysis of KRAS estimators requires a more natural assumption that directly places the smoothness requirement---i.e., membership in the RKHS $\mc{H}$---only on the true solution of the integral equation itself. In contrast, the analysis of LRAS estimators demands this smoothness property to hold for the solutions of infinitely many surrogate regularized problems, making it a less natural and generally more stringent condition. 

The $\beta$--source conditions~\eqref{eq:source-kras-basic} or~\eqref{eq:source_kras} assumed for KRAS estimators also differ from the source condition~\eqref{eq:source_bennett} invoked for analyzing LRAS estimators in two key ways. Consider, for instance, the source condition~\eqref{eq:source_kras} made in analyzing KRAS estimators. First, the smoothness assumptions in the conditions~\eqref{eq:source_kras} and~\eqref{eq:source_bennett} are placed on potentially different solutions to equation~\eqref{eq:int_eq_h}: in the KRAS case, it is the minimum-norm solution in $\mc{H}$, whereas in the LRAS case, it is the minimum-norm solution in $\mc{L}^2(P_W)$. Second, the fractional power is taken with respect to different operators: $\widetilde{\mc{T}}$ in the KRAS case, and $\mc{T}$ in the LRAS case. For the same value of $\beta$, condition~\eqref{eq:source_kras} imposes a stronger smoothness requirement than condition~\eqref{eq:source_bennett}, due to the presence of the additional operator $T_{\mathcal{H}}^{1/2}$ in the definition of $\widetilde{\mathcal{T}}$. This is not surprising, since---as discussed in Section~\ref{sec:conv_analysis}---regularization in the RKHS norm is ultimately equivalent to solving the more severely ill-posed problem~\eqref{eq:rkhs_int_eq_h} rather than the original problem~\eqref{eq:int_eq_h}.

We note that the finite-sample bounds for estimators that use optimal tuning parameters in \citet{bennett2023source} have the same functional form as those in Corollary~\ref{coro:optimal_lambda}. However, this similarity should not be taken to imply that the bounds for the KRAS and LRAS estimators achieve the same convergence rate under comparable smoothness levels of their respective minimum-norm solutions. This is so because, qualitatively, a given value of $\beta$ in Assumption~\eqref{eq:source_bennett} corresponds to a smaller $\beta$ in Assumption~\eqref{eq:source_kras} when representing an equivalent level of smoothness in the solution.

A few additional remarks are in place.
\begin{itemize}
\item All existing RAS estimators in the literature are for the case $I_0=1$. Our paper appears to be the first that incorporates structural zeros into RAS estimation.
\item Existing analyses of LRAS estimators, and more generally, of any RAS estimator regardless of the specific function classes $\mc{H}$ and $\mc{G}$ used, rely on a closeness and boundedness condition of the type in Assumption~\ref{ass:closeness}. The assumption is invoked to justify the equality $M_P(h) =\Vert I_0\mc{T}(h_0^\dag-h)\Vert_2^2$. An interesting open question raised by Proposition~\ref{prop:innermax} is whether part 2 of this proposition can be used to relax this assumption.  
\item \citet{bennett2023source} also analyzed the iterative Tikhonov version of LRAS estimators. For LRAS with $t$ Tikhonov iterations, they established weak and strong error bounds analogous to those in Theorem~\ref{theo:conv}, but with the exponents $\min\{\beta, 2\}$ and $\min\{\beta + 1, 2\}$ replaced by $\min\{\beta, 2t\}$ and $\min\{\beta + 1, 2t\}$, respectively. This is as expected since it is well known that iterative Tikhonov regularization with $t$ iterations offers a tangible benefit in regularization bias when the smoothness parameter $\beta$ is between $2$ and $2t$. We conjecture that a parallel theoretical analysis of $t$--iterated Tikhonov KRAS estimators should be doable under either Assumption \ref{ass:source-basic} or \ref{ass:source}. As in the non-iterated case, we anticipate that such an analysis will produce finite-sample bounds with the same structural form as those obtained by \citet{bennett2023source} for their $t$--iterated procedure. 

\item \citet{bennett2023source} also analyzed an adaptive version of the iterated Tikhonov LRAS estimator, deriving weak and strong error bounds with exponents $\beta+1$ and $\beta$, respectively---thereby fully eliminating the saturation phenomenon and enabling convergence rates to improve without bound as smoothness increases. We conjecture that an analogous result holds for adaptive iterated KRAS estimators. We note that while adaptive procedures represent a notable theoretical advance, their practical implementation introduces additional computational challenges that have yet to be fully examined.

\end{itemize}

\subsection{Kernel Adversarial Maximal Moment Restriction Estimators}\label{sec:compare_mmr}

Like regularized adversarial stabilized estimation strategies, maximal moment restriction methods also leverage the fact that the residual~\eqref{eq:residuals} has mean zero for any $g$ in a given function class $g\in\mc{G}\subset\mc{L}^2(P_Z)$. 
A kernel maximal moment estimator (KMMR) estimator (\citet{muandet2020kernel, mastouri2021proximal}) is defined as 
\[
\hat h_{KMMR}:=\arg\min_{h\in\mc{H}}\max_{g\in\mc{G}: \Vert g \Vert_{\mc{G}}\le1}\left[\E_n\{m(O;g)-h(W)g(Z)\}\right]^2+\lambda_\mc{H}\Vert h\Vert_{\mc{H}}^2.
\]
where $\mc{H}$ and $\mc{G}$ are RKHSs with associated norms $\Vert h\Vert_{\mc{H}}$  and  $\Vert g \Vert_{\mc{G}}$ and $\lambda_\mc{H}>0$. \cite{park2024proximal} adopts a similar strategy to construct solutions to integral equations needed for estimation of target parameters that fall outside the class defined in (\ref{eq:target_param}). Like KRAS estimators, KMMR estimators admit a closed-form expression and avoid inversion of ill-posed matrices.

To facilitate the comparison with KRAS estimators we will re-express the surrogate problem that the KMMR estimator targets in terms of the weak error relative to some operator. To do so, we start by noting that an immediate generalization of the results in \cite{zhang2023instrumental} yields:
\begin{align*}
\max_{g\in\mc{G},\Vert g\Vert_\mc{G}\le1}\left[\E\{m(O;g)-h(W)g(Z)\}\right]^2 = \Vert T_\mc{G} \left(\rho-\mc{T} h\right)\Vert^2_\mc{G}
\end{align*}
where $T_{\mc{G}}:\mc{L}^2(P_Z) \rightarrow \mc{G}$ is the integral operator associated with the kernel $K_{\mc{G}}$ that defines the RKHS $\mc{G}$, analogously to the operator defined in~\eqref{def:int_op}. Writing $T_{\mc{G}}=T_{\mc{G}}^{1/2} \circ T_{\mc{G}}^{1/2}$ and noticing that, under Assumption \ref{ass:K_G}, $T_{\mc{G}}^{1/2}:\mc{L}^2(P_Z) \rightarrow \mc{G}$ is an isometry when restricted to the orthogonal complement to the null space of $T_\mc{G}$, and therefore when restricted to the range of $T_\mc{G}^{1/2}$, we have that $\Vert T_\mc{G} \left(\rho-\mc{T} \: h\right)\Vert^2_\mc{G}=\Vert T_\mc{G}^{1/2} \left(\rho-\mc{T} \:h\right)\Vert^2_2 $. Then, recalling that $T_\mc{H}^{1/2}:\mc{L}^2(P_W) \rightarrow \mc{H}$ is surjective, we conclude that the regularized solution targeted by the KMMR estimator $\hat{h}_{KMMR}$ is
\begin{align}
h_{\lambda_{\mc{H},KMMR}}^\dag&:=\arg\min_{h\in\mc{H}}\max_{g\in\mc{G}: \Vert g \Vert_{\mc{G}}\le1}\left[\mathbb{E}\{m(O;g)-h(W)g(Z)\}\right]^2+\lambda_\mc{H}\Vert h\Vert_{\mc{H}}^2 \nonumber
\\
&=\arg\min_{h\in\mc{H}}\Vert T_\mc{G} \left(\rho-\mc{T} h\right)\Vert^2_\mc{G}+\lambda_\mc{H}\Vert h\Vert_{\mc{H}}^2 \label{eqn:KMMR-surrogate-standard}
\\
&= T_\mc{H}^{1/2} \Big[ \arg\min_{h'\in\mc{N}(T_\mc{H})^\perp}\Vert T_\mc{G}\: \rho-T_\mc{G} \circ \mc{T} \circ  T_\mc{H}^{1/2} h'\Vert_\mc{G}^2+\lambda_\mc{H}\Vert h'\Vert_2^2 \Big] \nonumber 
\\
&= T_\mc{H}^{1/2} \Big[ \arg\min_{h'\in\mc{L}^2(P_W)}\Vert T_\mc{G}^{1/2}\: \rho-T_\mc{G}^{1/2} \circ \mc{T} \circ  T_\mc{H}^{1/2} h'\Vert_2^2+\lambda_\mc{H}\Vert h'\Vert_2^2 \Big] \label{eqn:KMMR-surrogate-transformed}
\\
&:= T_\mc{H}^{1/2} \: h_{\lambda_{\mc{H},KMMR}}^{'}. \nonumber
\end{align}

Just as it was the case for the analysis of KRAS estimators, the preceding display implies two ways of analyzing the integral equation that the KMMR estimators are implicitly targeting. Specifically, identity~\eqref{eqn:KMMR-surrogate-standard} implies that the objective is the solution to the equation in $h \in \mc{H}$,
\begin{equation}
T_\mc{G} \circ \mc{T} \:h=T_\mc{G} \: \rho.
\label{eqn:KMMR-in-H}
\end{equation}
On the other hand, identity~\eqref{eqn:KMMR-surrogate-transformed} implies that the objective is the image under $T_\mc{H}^{1/2} $ of the solution to the equation in $h' \in \mc{L}^2(P_W)$
\begin{equation}
T_\mc{G}^{1/2} \circ \mc{T} \circ T_\mc{H}^{1/2} \:h'=T_\mc{G}^{1/2} \: \rho.
\label{eqn:KMMR-in-L2}
\end{equation}
If a solution to equation~\eqref{eqn:KMMR-in-H} exists in $\mc{H}$, then under Assumption \ref{ass:K_G}, a solution $h_{0,KMMR}$ with minimal $\mc{H}$--norm also exists and is unique. This is because, under such assumption, the map $T_\mc{G} \circ \mc{T} : \mc{H} \rightarrow \mc{G}$ is bounded and linear. In Appendix~\ref{sec:proofs_theorems} we show that its adjoint $(T_\mc{G} \circ \mc{T})^* : \mc{G} \rightarrow \mc{H}$ is given by $T_\mc{H} \circ \mc{T}^*$ where $\mc{T}^*:\mc{L}^2(P_Z) \rightarrow \mc{L}^2(P_W)$ is the adjoint of the conditional mean operator $\mc{T}:\mc{L}^2(P_W) \rightarrow \mc{L}^2(P_Z)$. By the identity~\eqref{eqn:KMMR-surrogate-standard} and Theorem \ref{thm:source-general} we then have that if $h_{0,KMMR}$ satisfies the source condition that posits the existence of $\beta>0$ such that 
\begin{equation}
    h_{0,KMMR} = (T_\mc{H} \circ \mc{T}^*\circ  T_\mc{G} \circ \mc{T})^{\beta/2}\:h^{*} \quad \text{for some} \: h^{*} \in \mc{H}
    \label{cond:source-Mastouri-H}
\end{equation}
equivalently, 
\begin{equation}
    h_{0,KMMR} = (T_\mc{H} \circ \mc{T}^*\circ  T_\mc{G} \circ \mc{T})^{\beta/2} \circ T_{\mc{H}}^{1/2}\:h^{**} \quad \text{for some} \: h^{**} \in \mc{L}^2(P_W)
    \label{cond:source-Mastouri}
\end{equation}
then, there exist constants $c_1$ and $c_2$ such that
\begin{equation}
\Vert h_{0,KMMR} - h_{\lambda_{\mc{H},KMMR}}^\dag \Vert_{\mc{H}}^2 \le c_1 \:\lambda_\mc{H}^{\min\left\{\beta,2 \right\}}  \text{ and }  \Vert T_\mc{G} \circ \mc{T} (h_{0,KMMR} - h_{\lambda_{\mc{H},KMMR}}^\dag )\Vert_{\mc{G}}^2 \le c_2 \: \lambda_\mc{H}^{\min\left\{\beta+1,2 \right\}} 
\label{eqn:bound-strong-KMMR}
\end{equation}

On the other hand, the identity~\eqref{eqn:KMMR-surrogate-transformed} points out to a second possible source condition for arriving at the bounds~\eqref{eqn:bound-strong-KMMR}. Specifically, consider the equation~\eqref{eqn:KMMR-in-L2} in $h' \in \mc{L}^2(P_W)$. 
It is easy to show that the adjoint of the operator $T_\mc{G}^{1/2} \circ \mc{T} \circ T_\mc{H}^{1/2}:\mc{L}^2(P_W) \rightarrow \mc{L}^2(P_Z) $ 
is given by $T_\mc{H}^{1/2} \circ \mc{T}^* \circ T_\mc{G}^{1/2}: \mc{L}^2(P_Z) \rightarrow \mc{L}^2(P_W)$. Thus, arguing as in Section~\ref{sec:conv_analysis}, it can be shown that under the $\beta$--source condition that posits the existence of $\beta>0$ such that
\begin{equation}
h_{0,KMMR}=T_{\mc{H}}^{1/2}(T_\mc{H}^{1/2}\circ\mc{T}^*\circ T_\mc{G} \circ \mc{T} \circ T_\mc{H}^{1/2})^{\beta/2}\:h^{**} \quad \text{for some} \: h^{**} \in \mc{L}^2(P_W)
\label{cond:source-KMMR-ours}
\end{equation}
the bounds~\eqref{eqn:bound-strong-KMMR} hold for some constants $c_1$ and $c_2$. 

As in the KRAS case, Assumption~\eqref{cond:source-KMMR-ours} is especially convenient because it rephrases the $\beta$–source condition in terms of the operator $T_\mc{H}^{1/2}\circ\mc{T}^*\circ T_\mc{G} \circ \mc{T} \circ T_\mc{H}^{1/2}$ acting between $\mathcal{L}^2$ spaces. This transformation shifts the problem from its original formulation---where the relevant operator has domain in the RKHS $\mathcal{H}$---to one where both the loss and the penalization are expressed in $\mathcal{L}^2$ norms. In doing so, it places the KMMR estimator on the same footing as the LRAS and the KRAS estimators---the latter regarded as estimators of the image under $T_\mc{H}^{1/2}$ of the solution of the transformed equation~\eqref{eq:rkhs_int_eq_h}---enabling a direct comparison of the operators that implicitly define their associated integral equations.

Let us now compare the advantages and limitations of the KRAS and KMMR estimators.
\begin{itemize}
\item KRAS estimators implicitly require equation~\eqref{eq:rkhs_int_eq_h} to admit a solution in $\mathcal{L}^2(P_W)$, whereas KMMR estimators require a solution to equation~\eqref{eqn:KMMR-in-L2}. The former condition implies the latter, but not conversely---unless the integral operator $T_\mc{G}$ is injective. Injectivity holds when the kernel $K_\mc{G}$ is Mat\'ern, but not necessarily otherwise. 
\item Unlike KMMR estimators, deriving bounds for KRAS estimators requires Assumption~\ref{ass:closeness} on closeness and boundedness, which, as noted earlier, can be hard to justify.

\item Even when $T_\mc{G}: \mc{L}^2(P_Z) \rightarrow \mc{L}^2(P_Z)$ is injective, so that equations~\eqref{eq:rkhs_int_eq_h} and~\eqref{eqn:KMMR-in-L2} share the same solutions, the formulation~\eqref{eqn:KMMR-in-L2} poses a more ill-posed problem. This is because $T_\mc{G}^{1/2}:\mc{L}^2(P_Z)\rightarrow\mc{G}$ is compact when $\mc{G}$ is endowed with the $\mc{L}^2(P_Z)$--norm. Thus, $T_\mc{G}^{1/2}$ in equation~\eqref{eqn:KMMR-in-L2} applies additional smoothing beyond $\mc{T} \circ T_\mc{H}^{1/2}$. Intuitively, solving a more ill-posed problem to a given level of accuracy demands stronger smoothness assumptions on the true solution. This added difficulty is reflected in the source condition imposed on the minimal $\mc{H}$--norm solution: for the same value of $\beta$, the source condition condition~\eqref{cond:source-KMMR-ours} (condition~\eqref{cond:source-Mastouri-H}) entails stronger smoothness than the source condition in Assumption~\ref{ass:source} (Assumption~\ref{ass:source-basic}), because the former is formulated in terms of the fractional power of an operator that smooths more strongly than the operator appearing in the latter.
\item To the best of our knowledge, the sharpest available bounds for the RMSE of $\hat{h}_{KMMR}$ under Assumption~\eqref{cond:source-Mastouri-H} were obtained by \cite{park2024proximal}, refining the earlier analysis of \cite{mastouri2021proximal}. Specifically, \cite{park2024proximal} showed that, with an appropriately chosen $\lambda_{\mc{H}}$, 
\[
\left\Vert \hat h_{KMMR}-h_{0,KMMR}\right\Vert_2=O_p\left(n^{-\frac{\min\{\beta, 2\}}{2(1+\min\{\beta, 2\})}}\right).
\]
At present, no results provide sharper bounds for the weak error $\Vert\mc{T}(\hat h_{KMMR}-h_{0,KMMR})\Vert_2$. This may be due to the fact that applying Theorem~\ref{thm:source-general} yields bounds on the weak regularization error with respect to the transformed operator $T_{\mc{G}} \circ \mc{T}$ (as in the second inequality in~\eqref{eqn:bound-strong-KMMR} rather than $\mc{T}$ itself). The lack of a sharper bound on the weak error has the following consequence. Suppose we estimate the solutions to equations~\eqref{eq:int_eq_h} and~\eqref{eq:int_eq_g} using KMMR estimators, and that the minimal RKHS-norm solutions satisfy the $\beta$--source condition~\eqref{cond:source-Mastouri-H} or~\eqref{cond:source-Mastouri} with parameters $\beta_h^{\text{KMMR}}$ and $\beta_g^{\text{KMMR}}$ respectively. Let $\hat{\theta}_{KMMR}$ be the resulting DML estimator of $\theta_0$. Using the existing finite-sample bounds we can establish that $R_n =o_p(n^{-1/2})$ if $\min\{\beta_h^{KMMR},2\}/(1+\min\{\beta_h^{KMMR},2\} + \min\{\beta_g^{KMMR},2\}/(1+\min\{\beta_g^{KMMR},2\}) > 1$, Thus, we cannot ensure that  $R_n =o_p(n^{-1/2})$ and thus, that $\hat{\theta}_{KMMR}$ is $\sqrt{n}$--consistent when, for instance, $\beta^{KMMR} =1$ where $
\beta^{KMMR}:= \max\left\{\beta_h^{KMMR},\beta_g^{KMMR}\right\}
$.

In contrast, suppose that to estimate the solutions of equations~\eqref{eq:int_eq_h} and~\eqref{eq:int_eq_g} we use KRAS estimators with the RKHSs in the inner maximization for the two estimation problems having the same critical radius $\delta_n$. Suppose that the minimum RKHS norm solutions satisfy the $\beta$--source condition of Assumption~\ref{ass:source-basic} or Assumption~\ref{ass:source} with parameters $\beta_h^{\text{KRAS}}$ and $\beta_g^{\text{KRAS}}$. Let $\hat{\theta}_{KRAS}$ be the DML estimator of $\theta_0$ using the optimal choice of tuning parameters as given in Corollary~\ref{coro:optimal_lambda}. Then, using the bounds in that corollary we can establish that $R_n =o_p(n^{-1/2})$ if 
\begin{equation}
\delta_n^{1+\frac{\min\left\{\beta^{KRAS},1\right\}}{1+\min\left\{\beta^{KRAS},1\right\}}}= o_p (n^{-1/2})
\label{condition_KRAS_root_n}
\end{equation} 
where $\beta^{KRAS}:=\max\{\beta_h^{KRAS},\beta_g^{KRAS}\}$. As noted in Section~\ref{sec:conv_analysis}, when the inner maximization RKHSs are Gaussian, condition \eqref{condition_KRAS_root_n} is always satisfied because $\delta_n = O(\sqrt{\left\{log(n)/n \right\}})$. When these are $d$-dimensional RKHSs with Sobolev kernels of order $\nu=(l+d)/2$ where $l>d$ is a positive integer, then $\delta_n = O(n^{-\nu/(2\nu+d)})$, which yields $R_n = o_p(n^{-1/2})$ whenever $\beta^{KRAS}\in(d/l,1]$.  

A rigorous comparison of the smoothness requirements under which the remainder $R_n$ is $o_P(n^{-1/2})$ for $\hat\theta_{\mathrm{KMMR}}$ and $\hat\theta_{\mathrm{KRAS}}$ is not possible for the following reasons: 
\begin{enumerate}
    \item Qualitatively, we expect that a given $\beta^{\mathrm{KRAS}}$ corresponds to a smaller $\beta^{\mathrm{KMMR}}$ for roughly the same smoothness of the minimal $\mathcal{H}$--norm solution, though the exact correspondence between the two parameters remains unknown.

    \item The absence of sharp bounds on the weak error of the KMMR nuisance parameter estimators means that current error analysis for $\hat{\theta}_{\mathrm{KMMR}}$ relies on bounding the product of their RMSEs. In contrast, the analysis for $\hat{\theta}_{\mathrm{KRAS}}$ relies on bounds on the minimum of two terms, each being the product of the weak error of one nuisance estimator and the RMSE of the other. Since weak errors are typically much smaller than RMSEs, yet-unavailable sharper bounds on the weak error of the KMMR nuisance estimators could potentially lead to a substantially different error characterization for $\hat{\theta}_{\mathrm{KMMR}}$. 
    \item Beyond the lack of weak error bounds, it also remains an open question whether existing RMSE bounds for KMMR estimators can be improved. This is because, the current RMSE bounds, derived by \citet{mastouri2021proximal} and \citet{park2024proximal}, were obtained via an application of the Bennett–Bernstein inequality in Hilbert spaces (Theorem 6.14 of \cite{steinwart2008support}), which provides a global bound that does not account for variation in the complexity of the RKHS induced by different kernels. In contrast, following the approach of \citet{bennett2023source}, we instead leveraged the critical radius---a localized measure of complexity---which enabled bounds that adapt to the effective size of the RKHS. This raises the question of whether a similarly localized analysis could yield sharper convergence rates for the RMSE of KMMR estimators.
    
    \item All that said, we conjecture that no refinement of the current analysis will show that $\hat{\theta}_{\mathrm{KMMR}}$ estimators outperform $\hat{\theta}_{\mathrm{KRAS}}$ estimators in terms of rates of convergence of the remainder $R_n$, as the fundamental issue remains: KMMR estimators target the solution of a more severely ill-posed problem, inherently limiting their precision relative to KRAS estimators. 
\end{enumerate}

\end{itemize}

\section{Discussion}\label{sec:conclusions}

We have introduced a new theoretical framework for analyzing KRAS estimators arising from an auto-DML scheme. Our analysis provides finite-sample bounds on both their weak error and RMSE, under source conditions. A key strength of this framework is that it avoids reliance on hard-to-quantify measures of ill-posedness and does not require the solution to be unique. 
Several important directions for future research remain open. 

First, our analysis has focused on single-step Tikhonov regularization. As indicated in Section~\ref{sec:compare_ras}, we conjecture that under our source conditions it will be possible to derive results for (adaptive and non-adaptive) iterated Tikhonov KRAS estimators, along the lines of \citet{bennett2023source}.

Second, our current analysis of KRAS estimators---and of all RAS estimators for that matter---relies on the closeness and boundedness condition in Assumption \ref{ass:closeness}. It remains an open question whether this Assumption can be relaxed, perhaps by exploiting part 2 of Proposition \ref{prop:innermax}.

Third, to our knowledge, the minimax optimality of LRAS estimators under source conditions has not yet been established. Likewise, it would be important to investigate if KRAS estimators are minimax optimal under our proposed source conditions.

Finally, while RAS estimators have gained popularity in recent years, the literature offers little guidance on how to tune their hyperparameters in practice. When hyperparameter selection strategies are discussed, they are often heuristic and lack theoretical guarantees. Developing principled, data-driven methods for selecting regularization parameters---and establishing their statistical properties---is a critical step toward making these estimators truly practical.

\section*{Acknowledgments}
Andrea Rotnitzky’s work has been supported by the National Heart, Lung, and Blood Institute grant R01-HL137808, and by the National Institute of Allergy and Infectious Diseases grants UM1- AI068635 and R37-AI029168.

\bibliographystyle{apalike}
\bibliography{references}

\begin{thebibliography}{}

\bibitem[Ai and Chen, 2003]{ai2003efficient}
Ai, C. and Chen, X. (2003).
\newblock Efficient estimation of models with conditional moment restrictions containing unknown functions.
\newblock {\em Econometrica}, 71(6):1795--1843.

\bibitem[Ai and Chen, 2007]{ai2007estimation}
Ai, C. and Chen, X. (2007).
\newblock Estimation of possibly misspecified semiparametric conditional moment restriction models with different conditioning variables.
\newblock {\em Journal of Econometrics}, 141(1):5--43.

\bibitem[Bartlett et~al., 2005]{bartlett2005local}
Bartlett, P.~L., Bousquet, O., and Mendelson, S. (2005).
\newblock Local rademacher complexities.
\newblock {\em The Annals of Statistics}.

\bibitem[Bennett et~al., 2022]{bennett2022inference}
Bennett, A., Kallus, N., Mao, X., Newey, W., Syrgkanis, V., and Uehara, M. (2022).
\newblock Inference on strongly identified functionals of weakly identified functions.
\newblock {\em arXiv preprint arXiv:2208.08291}.

\bibitem[Bennett et~al., 2023a]{bennett2023minimax}
Bennett, A., Kallus, N., Mao, X., Newey, W., Syrgkanis, V., and Uehara, M. (2023a).
\newblock Minimax instrumental variable regression and $ l\_2 $ convergence guarantees without identification or closedness.
\newblock In {\em The Thirty Sixth Annual Conference on Learning Theory}, pages 2291--2318. PMLR.

\bibitem[Bennett et~al., 2023b]{bennett2023source}
Bennett, A., Kallus, N., Mao, X., Newey, W., Syrgkanis, V., and Uehara, M. (2023b).
\newblock Source condition double robust inference on functionals of inverse problems.
\newblock {\em arXiv preprint arXiv:2307.13793}.

\bibitem[Bennett et~al., 2019]{bennett2019deep}
Bennett, A., Kallus, N., and Schnabel, T. (2019).
\newblock Deep generalized method of moments for instrumental variable analysis.
\newblock {\em Advances in neural information processing systems}, 32.

\bibitem[Borwein and Lewis, 2006]{borwein2006convex}
Borwein, J. and Lewis, A. (2006).
\newblock {\em Convex Analysis and Nonlinear Optimization: Theoryand Examples}.
\newblock Springer.

\bibitem[Boyd, 2004]{boyd2004convex}
Boyd, S. (2004).
\newblock Convex optimization.
\newblock {\em Cambridge UP}.

\bibitem[Carrasco et~al., 2007]{carrasco2007linear}
Carrasco, M., Florens, J.-P., and Renault, E. (2007).
\newblock Linear inverse problems in structural econometrics estimation based on spectral decomposition and regularization.
\newblock {\em Handbook of econometrics}, 6:5633--5751.

\bibitem[Cavalier, 2011]{cavalier2011inverse}
Cavalier, L. (2011).
\newblock Inverse problems in statistics.
\newblock In {\em Inverse Problems and High-Dimensional Estimation: Stats in the Ch{\^a}teau Summer School, August 31-September 4, 2009}, pages 3--96. Springer.

\bibitem[Chen and Pouzo, 2012]{chen2012estimation}
Chen, X. and Pouzo, D. (2012).
\newblock Estimation of nonparametric conditional moment models with possibly nonsmooth generalized residuals.
\newblock {\em Econometrica}, 80(1):277--321.

\bibitem[Chernozhukov et~al., 2018]{chernozhukov2018double}
Chernozhukov, V., Chetverikov, D., Demirer, M., Duflo, E., Hansen, C., Newey, W., and Robins, J. (2018).
\newblock Double/debiased machine learning for treatment and structural parameters.
\newblock {\em The Econometrics Journal}, 21(1):C1–C68.

\bibitem[Chernozhukov et~al., 2022a]{chernozhukov2022riesznet}
Chernozhukov, V., Newey, W., Quintas-Mart{\i}nez, V.~M., and Syrgkanis, V. (2022a).
\newblock Riesznet and forestriesz: Automatic debiased machine learning with neural nets and random forests.
\newblock In {\em International Conference on Machine Learning}, pages 3901--3914. PMLR.

\bibitem[Chernozhukov et~al., 2022b]{chernozhukov2022automatic}
Chernozhukov, V., Newey, W.~K., and Singh, R. (2022b).
\newblock Automatic debiased machine learning of causal and structural effects.
\newblock {\em Econometrica}, 90(3):967--1027.

\bibitem[Chernozhukov et~al., 2022c]{chernozhukov2022debiased}
Chernozhukov, V., Newey, W.~K., and Singh, R. (2022c).
\newblock Debiased machine learning of global and local parameters using regularized riesz representers.
\newblock {\em The Econometrics Journal}, 25(3):576--601.

\bibitem[Chernozhukov et~al., 2023]{chernozhukov2023simple}
Chernozhukov, V., Newey, W.~K., and Singh, R. (2023).
\newblock A simple and general debiased machine learning theorem with finite-sample guarantees.
\newblock {\em Biometrika}, 110(1):257--264.

\bibitem[Conway, 2012]{conway2012course}
Conway, J.~B. (2012).
\newblock {\em A course in abstract analysis}, volume 141.
\newblock American Mathematical Soc.

\bibitem[Cui et~al., 2024]{cui2024semiparametric}
Cui, Y., Pu, H., Shi, X., Miao, W., and Tchetgen~Tchetgen, E. (2024).
\newblock Semiparametric proximal causal inference.
\newblock {\em Journal of the American Statistical Association}, 119(546):1348--1359.

\bibitem[D{\'\i}az et~al., 2023]{diaz2023nonparametric}
D{\'\i}az, I., Williams, N., Hoffman, K.~L., and Schenck, E.~J. (2023).
\newblock Nonparametric causal effects based on longitudinal modified treatment policies.
\newblock {\em Journal of the American Statistical Association}, 118(542):846--857.

\bibitem[Dikkala et~al., 2020]{dikkala2020minimax}
Dikkala, N., Lewis, G., Mackey, L., and Syrgkanis, V. (2020).
\newblock Minimax estimation of conditional moment models.
\newblock {\em Advances in Neural Information Processing Systems}, 33:12248--12262.

\bibitem[Escanciano and Li, 2013]{escanciano2013identification}
Escanciano, J.~C. and Li, W. (2013).
\newblock On the identification of structural linear functionals.
\newblock Technical report, cemmap working paper.

\bibitem[Foster and Syrgkanis, 2023]{foster2023orthogonal}
Foster, D.~J. and Syrgkanis, V. (2023).
\newblock Orthogonal statistical learning.
\newblock {\em The Annals of Statistics}, 51(3):879--908.

\bibitem[Freyberger and Horowitz, 2015]{freyberger2015identification}
Freyberger, J. and Horowitz, J.~L. (2015).
\newblock Identification and shape restrictions in nonparametric instrumental variables estimation.
\newblock {\em Journal of Econometrics}, 189(1):41--53.

\bibitem[Ghassami et~al., 2022]{ghassami2022minimax}
Ghassami, A., Ying, A., Shpitser, I., and Tchetgen, E.~T. (2022).
\newblock Minimax kernel machine learning for a class of doubly robust functionals with application to proximal causal inference.
\newblock In {\em International conference on artificial intelligence and statistics}, pages 7210--7239. PMLR.

\bibitem[Haneuse and Rotnitzky, 2013]{haneuse2013estimation}
Haneuse, S. and Rotnitzky, A. (2013).
\newblock Estimation of the effect of interventions that modify the received treatment.
\newblock {\em Statistics in medicine}, 32(30):5260--5277.

\bibitem[Hohage, 2002]{hohage2002lecture}
Hohage, T. (2002).
\newblock Lecture notes on inverse problems.
\newblock {\em University of Gottingen}, pages 14--21.

\bibitem[Kallus et~al., 2021]{kallus2021causal}
Kallus, N., Mao, X., and Uehara, M. (2021).
\newblock Causal inference under unmeasured confounding with negative controls: A minimax learning approach.
\newblock {\em arXiv preprint arXiv:2103.14029}.

\bibitem[Kennedy, 2024]{kennedy2024semiparametric}
Kennedy, E.~H. (2024).
\newblock Semiparametric doubly robust targeted double machine learning: a review.
\newblock {\em Handbook of statistical methods for precision medicine}, pages 207--236.

\bibitem[Li et~al., 2023]{li2023non}
Li, W., Miao, W., and Tchetgen~Tchetgen, E. (2023).
\newblock Non-parametric inference about mean functionals of non-ignorable non-response data without identifying the joint distribution.
\newblock {\em Journal of the Royal Statistical Society Series B: Statistical Methodology}, 85(3):913--935.

\bibitem[Liao et~al., 2020]{liao2020provably}
Liao, L., Chen, Y.-L., Yang, Z., Dai, B., Kolar, M., and Wang, Z. (2020).
\newblock Provably efficient neural estimation of structural equation models: An adversarial approach.
\newblock {\em Advances in Neural Information Processing Systems}, 33:8947--8958.

\bibitem[Mastouri et~al., 2021]{mastouri2021proximal}
Mastouri, A., Zhu, Y., Gultchin, L., Korba, A., Silva, R., Kusner, M., Gretton, A., and Muandet, K. (2021).
\newblock Proximal causal learning with kernels: Two-stage estimation and moment restriction.
\newblock In {\em International conference on machine learning}, pages 7512--7523. PMLR.

\bibitem[Miao et~al., 2018]{miao2018identifying}
Miao, W., Geng, Z., and Tchetgen~Tchetgen, E.~J. (2018).
\newblock Identifying causal effects with proxy variables of an unmeasured confounder.
\newblock {\em Biometrika}, 105(4):987--993.

\bibitem[Miao et~al., 2024]{miao2024identification}
Miao, W., Liu, L., Li, Y., Tchetgen~Tchetgen, E.~J., and Geng, Z. (2024).
\newblock Identification and semiparametric efficiency theory of nonignorable missing data with a shadow variable.
\newblock {\em ACM/JMS Journal of Data Science}, 1(2):1--23.

\bibitem[Miao and Tchetgen~Tchetgen, 2016]{miao2016varieties}
Miao, W. and Tchetgen~Tchetgen, E.~J. (2016).
\newblock On varieties of doubly robust estimators under missingness not at random with a shadow variable.
\newblock {\em Biometrika}, 103(2):475--482.

\bibitem[Muandet et~al., 2020a]{muandet2020kernel}
Muandet, K., Jitkrittum, W., and K{\"u}bler, J. (2020a).
\newblock Kernel conditional moment test via maximum moment restriction.
\newblock In {\em Conference on Uncertainty in Artificial Intelligence}, pages 41--50. PMLR.

\bibitem[Muandet et~al., 2020b]{muandet2020dual}
Muandet, K., Mehrjou, A., Lee, S.~K., and Raj, A. (2020b).
\newblock Dual instrumental variable regression.
\newblock {\em Advances in Neural Information Processing Systems}, 33:2710--2721.

\bibitem[Newey and Powell, 2003]{newey2003instrumental}
Newey, W.~K. and Powell, J.~L. (2003).
\newblock Instrumental variable estimation of nonparametric models.
\newblock {\em Econometrica}, 71(5):1565--1578.

\bibitem[Olivas-Martinez, 2025]{olivas2025chap3}
Olivas-Martinez, A. (2025).
\newblock Proximal causal inference for modified treatment policies.
\newblock In {\em Advances in Proximal Inference for Continuous Exposures, Estimation of Ill-Posed Regression, and Non-Inferiority Assessment in Active-Controlled Trials}, phd dissertation~3, pages 32--62. Seattle, WA.

\bibitem[Park et~al., 2024]{park2024proximal}
Park, C., Stensrud, M., and Tchetgen, E.~T. (2024).
\newblock Proximal causal inference for conditional separable effects.
\newblock {\em arXiv preprint arXiv:2402.11020}.

\bibitem[Paulsen and Raghupathi, 2016]{paulsen2016introduction}
Paulsen, V.~I. and Raghupathi, M. (2016).
\newblock {\em An introduction to the theory of reproducing kernel Hilbert spaces}, volume 152.
\newblock Cambridge university press.

\bibitem[Robins et~al., 2008]{robins2008higher}
Robins, J., Li, L., Tchetgen, E., van~der Vaart, A., et~al. (2008).
\newblock Higher order influence functions and minimax estimation of nonlinear functionals.
\newblock In {\em Probability and statistics: essays in honor of David A. Freedman}, volume~2, pages 335--422. Institute of Mathematical Statistics.

\bibitem[Santos, 2011]{santos2011instrumental}
Santos, A. (2011).
\newblock Instrumental variable methods for recovering continuous linear functionals.
\newblock {\em Journal of Econometrics}, 161(2):129--146.

\bibitem[Severini and Tripathi, 2012]{severini2012efficiency}
Severini, T.~A. and Tripathi, G. (2012).
\newblock Efficiency bounds for estimating linear functionals of nonparametric regression models with endogenous regressors.
\newblock {\em Journal of Econometrics}, 170(2):491--498.

\bibitem[Steinwart and Christmann, 2008]{steinwart2008support}
Steinwart, I. and Christmann, A. (2008).
\newblock {\em Support vector machines}.
\newblock Springer.

\bibitem[Tchetgen et~al., 2020]{tchetgen2020introduction}
Tchetgen, E. J.~T., Ying, A., Cui, Y., Shi, X., and Miao, W. (2020).
\newblock An introduction to proximal causal learning.
\newblock {\em arXiv preprint arXiv:2009.10982}.

\bibitem[Wainwright, 2019]{wainwright2019high}
Wainwright, M.~J. (2019).
\newblock {\em High-dimensional statistics: A non-asymptotic viewpoint}, volume~48.
\newblock Cambridge university press.

\bibitem[Zhang et~al., 2023]{zhang2023instrumental}
Zhang, R., Imaizumi, M., Sch{\"o}lkopf, B., and Muandet, K. (2023).
\newblock Instrumental variable regression via kernel maximum moment loss.
\newblock {\em Journal of Causal Inference}, 11(1):20220073.

\end{thebibliography}

\appendix

\section{Proofs of Main Text Mathematical Statements}
\label{sec:proofs_theorems}
    \label{U_n-Decom}

\begin{proof}[Proof of Proposition~\ref{prop:sum_squares}]
     If $\{g_j=I_0 g_j^\dag :  g_j^\dag \in \mc{G},\: j\in\mathbb{J}\}$ is an orthonormal basis of $\overline{I_0 \mc{G}}$, then $\Vert \mc{P}_{\overline{I_0\mc{G}}}\circ \mc{T}I_0(h_0^\dag-h)\Vert_2^2=\sum_{j\in\mathbb{J}} \langle I_0\mc{T}(h_{0}^\dag-h)\: ,\: I_0 g_j^\dag \rangle_2^2 =\sum_{j\in\mathbb{J}} \{\mathbb{E} [ \rho(Z)g_j^\dag(Z)- I_0(X)r(O)h(W)\:  g_j^\dag(Z) ]\}^2
     =\sum_{j\in\mathbb{J}}\{\mathbb{E} [ m(O,g_j^\dag)- I_0(X)r(O)h(W)\:  g_j^\dag(Z) ]\}^2=Q_\mathbb{J}(h)$
\end{proof}

\begin{proof}[Proof of Proposition \ref{prop:innermax}]
To show part 1, write
\begin{align*}
M_P(h)
&\stackrel{(i)}{=}4c^2 \max_{g\in \mc{G}} \mathbb{E}\left\{\rho(Z)g(Z)-I_0(X)r(O)h(W)g(Z)-c^2I_0(X)g(Z)^2\right\}
\\
&\stackrel{(ii)}{=}4c^2 \max_{g\in \mc{G}} \mathbb{E}\Big\{I_0(X)g(Z)\mc{T}\big(h_{0}^\dag-h\big)(Z)-c^2I_0(X)g(Z)^2\Big\}
\\
&=4c^2 \max_{g\in \mc{G}} \mathbb{E}\Big[I_0(X)g(Z)\mc{T}\big(h_{0}^\dag-h\big)(Z)-c^2I_0(X)g(Z)^2-\frac{1}{4c^2}I_0(X)\big\{\mc{T}\big(h_{0}^\dag-h\big)(Z)\big\}^2\Big]
\\
& \quad +\mathbb{E}\Big[I_0(X)\big\{\mc{T}\big(h_{0}^\dag-h\big)(Z)\big\}^2\Big]
\\
&=\max_{g\in \mc{G}} -\mathbb{E}\Big[I_0(X)\big\{2c^2g(Z)-\mc{T}\big(h_{0}^\dag-h\big)(Z)\big\}^2\Big]
+\mathbb{E}\Big[I_0(X)\big\{\mc{T}\big(h_{0}^\dag-h\big)(Z)\big\}^2\Big]
\\
&\stackrel{(iii)}{=}\mathbb{E}\Big[I_0(X)\big\{\mc{T}\big(h_{0}^\dag-h\big)(Z)\big\}^2\Big]
\end{align*}
where the equality $(i)$ holds because $\rho$ is the Riesz representer of the map $g \mapsto \mathbb{E}\left[m(O;g)\right]$, the equality $(ii)$ is because $I_0 h_{0}^\dag$ is a solution to equation~\eqref{eq:int_eq_h}, and the equality $(iii)$ is because by assumption, $(2c^2)^{-1}\cdot\mc{T}(h_0^\dag-h)\in\mc{G}$. Note that if $\mc{G}$ was unrestricted, part 1 would be, on expectation, the Frenchel conjugate representation of the squared function.
This concludes the proof of part 1.
To show part 2 write
\begin{align*}
\max_{g\in \mc{G}}& -\mathbb{E}\Big[I_0(X)\big\{2c^2g(Z)-\mc{T}\big(h_{0}^\dag-h\big)(Z)\big\}^2\Big]
+\mathbb{E}\Big[I_0(X)\big\{\mc{T}\big(h_{0}^\dag-h\big)(Z)\big\}^2\Big]\\
&\stackrel{(i)}{=}-\min_{g\in \mc{G}} \mathbb{E}\Big[\big\{I_0(X)g(Z)-I_0(X)\mc{T}\big(h_{0}^\dag-h\big)(Z)\big\}^2\Big]
+\mathbb{E}\Big[I_0(X)\big\{\mc{T}\big(h_{0}^\dag-h\big)(Z)\big\}^2\Big]
\\
&\stackrel{(ii)}{=}-\big\Vert I_0\mc{T}\big(h_{0}^\dag-h\big) - \mc{P}_{\overline{I_0\mc{G}}}\circ \mc{T}I_0(h_0^\dag-h)\big\Vert_2^2 + \big\Vert I_0\mc{T}\big(h_{0}^\dag-h\big) \big\Vert_2^2
\\
&\stackrel{(iii)}{=}\big\Vert  \mc{P}_{\overline{I_0\mc{G}}}\circ \mc{T}I_0(h_0^\dag-h)\big\Vert_2^2
\end{align*}
where the equality $(i)$ holds because $\mc{G}$ is a linear space, the equality $(ii)$ holds by the definition of the projection operator and the assumption that $\mc{P}_{\overline{I_0\mc{G}}}\circ \mc{T}I_0(h_0^\dag-h)
\in I_0\mc{G}$, and the equality $(iii)$ holds by Pythagorean Theorem. Hence, by Proposition~\ref{prop:sum_squares}, we have
$ 
M_P(h)=\Vert  \mc{P}_{\overline{I_0\mc{G}}}\circ \mc{T}I_0(h_0^\dag-h)\Vert_2^2=Q_\mathbb{J}(h)$. This concludes the proof of part 2.
\end{proof}

\begin{proof}[Theorem \ref{theo:biascontrol} under Assumption~\ref{ass:source}]
Recall that $h'_{0,KRAS}$ is the solution to equation~\eqref{eq:rkhs_int_eq_h} with minimal $\mc{L}^2(P_W)$--norm and that $h_{\lambda_\mc{H},KRAS}'=\arg\min_{h'\in\mathcal{N}(T_\mc{H})^\perp}\{\Vert\rho -\widetilde{\mc{T}}h'\Vert_2^2+4\lambda_{\mc{H}}\Vert h'\Vert_2^2\}$ (see Section~\ref{sec:conv_analysis}). 
Now, Lemma~\ref{lemma:min_rkhs_norm_sol_h} establishes that $h'_{0,KRAS}\in \mc{N}(T_\mc{H})^\perp$ and $h^\dag_{0,KRAS}=T_\mc{H}^{1/2}h'_{0,KRAS}$. On the other hand, we know that 
$h'_{\lambda_\mc{H},KRAS}\in \mc{N}(T_\mc{H})^\perp$ and $h^\dag_{\lambda_\mc{H},KRAS}=T_\mc{H}^{1/2}h'_{\lambda_\mc{H},KRAS}$. It then follows from Lemma~\ref{lemma:null_space}~(d) that $\|h_{\lambda_\mc{H},KRAS}^\dag - h_{0,KRAS}^\dag\|_\mc{H} = \|h_{\lambda_\mc{H},KRAS}' - h_{0,KRAS}'\|_2$ and $\|\widetilde{\mc{T}}(h_{\lambda_\mc{H},KRAS}'-h_{0,KRAS}')\|_2 = \|I_0^{op}\circ\mc{T}(h_{\lambda_\mc{H},KRAS}^\dag - h_{0,KRAS}^\dag)\|_2$.  
Hence, under Assumption~\ref{ass:source}, Theorem~\ref{thm:source-general} with $\mc{A}=\widetilde{\mc{T}}$ ensures that there exist constants $c_1^*$ and $c_2^*$ such that
\begin{align*}
\|h_{\lambda_\mc{H},KRAS}^\dag - h_{0,KRAS}^\dag\|_\mc{H}^2&=\|h_{\lambda_\mc{H},KRAS}' - h_{0,KRAS}'\|_2 \le c_1^*\lambda_\mc{H}^{\min\{\beta,2\}}\quad\text{and}\\
\|I_0^{op}\circ\mc{T}(h_{\lambda_\mc{H},KRAS}^\dag - h_{0,KRAS}^\dag)\|_2^2&=\|\widetilde{\mc{T}}(h_{\lambda_\mc{H},KRAS}'-h_{0,KRAS}')\|_2 \le c_2^*\lambda_\mc{H}^{\min\{\beta+1,2\}}.
\end{align*}
\end{proof}

\begin{proof}[Theorem \ref{theo:conv}]
    We adapted the proof of Theorem 4 of \cite{bennett2023source} to our estimator. To simplify notation, we omit the subscript $KRAS$ from the functions $h_{0,KRAS}^\dag$, $h_{\lambda_\mc{H},KRAS}^\dag$, and $\tilde h_{KRAS}$ defined in~\eqref{def:h_0_dag}, \eqref{argmin_h}, and~\eqref{def:h_tilde}. Define $ L(\tau):=\frac{1}{4c^2}\big\Vert I_0\mc{T}\big\{h_{0}^\dag-h_{\lambda_\mc{H}}^\dag-\tau\big(\tilde h-h_{\lambda_\mc{H}}^\dag\big)\big\}\big\Vert_2^2+\lambda_{\mc{H}}\big\Vert  h_{\lambda_\mc{H}}^\dag+\tau\big(\tilde h- h_{\lambda_\mc{H}}^\dag\big)\big\Vert_{\mc{H}}^2$. $L(\tau)$ attains its unique global minimum at $\tau=0$ because $h_{\lambda_\mc{H}}^\dag$ is the unique minimizer in $\mc{H}$ of the criterion in~\eqref{argmin_h}. Then, $L'(0)=0$ and since $L(\tau)$ is quadratic, a Taylor's expansion yields
        \begin{align*}
            L(1)-L(0)&=L'(0)\cdot 1+\frac{1}{2}L''(0)\cdot 1^2=\frac{1}{2}L''(0)=\frac{1}{4c^2}\big\Vert I_0\mc{T}\big(\tilde h-h_{\lambda_\mc{H}}^\dag\big)\big\Vert_2^2+\lambda_{\mc{H}}\big\Vert \tilde h- h_{\lambda_\mc{H}}^\dag\big\Vert_{\mc{H}}^2.
        \end{align*}
        On the other hand, $L(0) =\frac{1}{4c^2}\Vert   I_0\mc{T}(h_0^\dag-h_{\lambda_\mc{H}}^\dag)\Vert_2^2+\lambda_{\mc{H}}\Vert h_{\lambda_\mc{H}}^\dag\Vert_{\mc{H}}^2$ and $ L(1)=\frac{1}{4c^2}\Vert  I_0\mc{T}(h_0^\dag-\tilde h)\Vert_2^2+\lambda_{\mc{H}}\Vert \tilde h\Vert_{\mc{H}}^2$.
         Thus, we arrive at 
         \begin{align}
           \frac{1}{4c^2}\big\Vert I_0\mc{T}\big(\tilde h-&h_{\lambda_\mc{H}}^\dag\big)\big\Vert_2^2+\lambda_{\mc{H}}\big\Vert \tilde h- h_{\lambda_\mc{H}}^\dag\big\Vert_{\mc{H}}^2\label{eq:taylor_exp_L}\\
           &= \frac{1}{4c^2}\Big\{\big\Vert   I_0\mc{T}\big(h_0^\dag-\tilde h\big)\big\Vert_2^2-\big\Vert I_0\mc{T}\big(h_0^\dag-h_{\lambda_\mc{H}}^\dag\big)\big\Vert_2^2\Big\}\nonumber+\lambda_{\mc{H}}\Big(\big\Vert \tilde h\big\Vert_{\mc{H}}^2-\big\Vert h_{\lambda_\mc{H}}^\dag\big\Vert_{\mc{H}}^2\Big).
        \end{align}

        To control the right-hand side of~\eqref{eq:taylor_exp_L}, we invoke Lemma~\ref{lemma:erm} which establishes that with probability at least $\kappa(\delta_n)$ the following event happens
        \begin{align*}
    \frac{1}{4c^2}\Big\{\big\Vert I_0 \mc{T} \big(h_0^\dag-&\tilde h\big)\big\Vert^2_2-\big\Vert I_0\mc{T} \big(h_0^\dag-h\big)\big\Vert^2_2\Big\}+\lambda_{\mc{H}}\Big(\big\Vert \tilde h\big\Vert_{\mc{H}}^2-\big\Vert h\big\Vert_{\mc{H}}^2\Big)\\
     &\quad\quad\quad\le \frac{1}{2c^2}\big\Vert I_0\mc{T} \big(h_0^\dag-h\big)\big\Vert^2_2+c_0\delta_n\big\Vert  I_0\mc{T}\big(\tilde h-h\big)\big\Vert_2+\lambda_{\mc{G}}B_1^2+c_0'\delta^2_n
\end{align*}
        where $c_0=9\cdot\max\{\Vert s\Vert_\infty,\:\Vert r\Vert_\infty\cdot k_1B,\: 2c^2b\}/(c^2 b)$ and $c'_0=c^2\cdot c_0(2_3c_0+\max\{2,2c_0/9\})$.

        Hence, identity~\eqref{eq:taylor_exp_L} implies that, with probability at least $\kappa(\delta_n)$, the following holds
        \begin{align*}
           \frac{1}{4c^2}\big\Vert I_0\mc{T}\big(\tilde h-h_{\lambda_\mc{H}}^\dag\big)\big\Vert_2^2&+\lambda_{\mc{H}}\big\Vert \tilde h- h_{\lambda_\mc{H}}^\dag\big\Vert_{\mc{H}}^2\\
           &\le   \frac{1}{2c^2}\big\Vert  I_0\mc{T}\big(h_0^\dag-h_{\lambda_\mc{H}}^\dag\big)\big\Vert_2^2+c_0\delta_n\big\Vert I_0\mc{T}\big(\tilde h-h_{\lambda_\mc{H}}^\dag\big)\big\Vert_2+\lambda_{\mc{G}}B_1^2+c'_0\delta_n^2.
        \end{align*}

        By the AM-GM inequality, $c_0\delta_n\Vert I_0\mc{T}(\tilde h-h_{\lambda_\mc{H}}^\dag)\Vert_2 \le 2c^2\cdot c_0^2 \delta_n^2 +\frac{1}{8c^2}\Vert I_0\mc{T}(\tilde h-h_{\lambda_\mc{H}}^\dag)\Vert_{2}^2$,
       which combined with the right-hand side of the last display yields that
        \begin{align}
             \lambda_{\mc{H}}\big\Vert \tilde h-h_{\lambda_\mc{H}}^\dag\big\Vert_{\mc{H}}^2+\frac{1}{8c^2}\big\Vert I_0\mc{T}\big(\tilde h-h_{\lambda_\mc{H}}^\dag\big)\big\Vert_2^2&\le  \frac{1}{2c^2}\big\Vert I_0\mc{T}\big(h_0^\dag-h_{\lambda_\mc{H}}^\dag\big)\big\Vert_2^2+\lambda_{\mc{G}}B_1^2+c_0''\delta_n^2,\label{useful_exp}
        \end{align}
        with $c_0''=c'_0+2c^2\cdot c_0^2$, happens with probability at least $\kappa(\delta_n)$, which, in turn, implies that
      $ 
        \lambda_{\mc{H}}\Vert \tilde h-h_{\lambda_\mc{H}}^\dag\Vert_{\mc{H}}^2\le  \frac{1}{2c^2}\Vert I_0\mc{T} (h_0^\dag-h_{\lambda_\mc{H}}^\dag)\Vert_2^2+\lambda_{\mc{G}}B_1^2+c_0''\delta_n^2
       $ also holds with probability at least $\kappa(\delta_n)$.
        
       Applying the inequality $\Vert h_1+h_2\Vert_{\mc{H}}^2\le 2\Vert h_1\Vert_{\mc{H}}^2+2\Vert h_2\Vert_{\mc{H}}^2$ we then obtain that the event
        \begin{align*}
            \big\Vert \tilde h-h_0^\dag\big\Vert_{\mc{H}}^2
            \le& \:2\big\Vert \tilde h-h_{\lambda_\mc{H}}^\dag\big\Vert_{\mc{H}}^2+2\big\Vert h_{\lambda_\mc{H}}^\dag-h_0^\dag\big\Vert_{\mc{H}}^2\\
            \le& \frac{1}{c^2\cdot\lambda_{\mc{H}}}\big\Vert I_0\mc{T}\big(h_0^\dag-h_{\lambda_\mc{H}}^\dag\big)\big\Vert_2^2+\frac{2\lambda_{\mc{G}}B_1^2}{\lambda_{\mc{H}}}+\frac{2c_0''\delta_n^2}{\lambda_{\mc{H}}}+2\big\Vert h_{\lambda_\mc{H}}^\dag-h_0^\dag\big\Vert_{\mc{H}}^2,
        \end{align*}
        happens with probability at least $\kappa(\delta_n)$.
       Because $\tilde h-h_0^\dag\in\mc{H}$, it follows that $ \Vert \tilde h-h_0^\dag\Vert_2^2\le k_1^2 \: \Vert \tilde h-h_0^\dag\Vert_{\mc{H}}^2$. Hence, invoking Theorem \ref{theo:biascontrol} to bound the terms $\Vert I_0\mc{T}(h_0^\dag-h_{\lambda_{\mc{H}}}^\dag)\Vert_2^2$ and $\Vert h_{\lambda_{\mc{H}}}^\dag-h_0^\dag\Vert_{\mc{H}}^2$, we conclude that, with probability at least $\kappa(\delta_n)$,
  $ \Vert \tilde h-h_0^\dag\Vert_2^2\le c_2\big(\frac{\delta_n^2}{\lambda_{\mc{H}}}+\frac{\lambda_{\mc{G}}}{\lambda_{\mc{H}}}+\lambda_{\mc{H}}^{\min\left\{\beta,1\right\}}\big),
    $ where $ c_2=k_1^2\cdot\max\{2c_1^*+c_2^*c^{-2}\:,\: 2B_1^2\:,\:2c_0''\}$, and $c_1^*$ and $c_2^*$ are the constants of Theorem~\ref{theo:biascontrol}.
    
   To derive the bound for $\Vert I_0\mc{T}(\tilde h-h_0^\dag)\Vert_2^2$, first note that from inequality~\eqref{useful_exp} we have that the event
\begin{align*}
             \frac{1}{8c^2}\big\Vert I_0\mc{T}\big(\tilde h-h_{\lambda_{\mc{H}}}^\dag\big)\big\Vert_2^2&\le  \frac{1}{2c^2}\big\Vert I_0\mc{T}\big(h_0^\dag-h_{\lambda_{\mc{H}}}^\dag\big)\big\Vert_2^2+\lambda_{\mc{G}}B_1^2+c_0''\delta_n^2
        \end{align*}
        happens with probability at least $\kappa(\delta_n)$. Combining this inequality with the inequality $ \Vert I_0\mc{T}(\tilde h-h_0^\dag)\Vert_2^2\le 2\Vert I_0\mc{T}(\tilde h-h_{\lambda_{\mc{H}}}^\dag)\Vert_2^2+2\Vert I_0\mc{T}(h_{\lambda_{\mc{H}}}^\dag-h_0^\dag)\Vert_2^2$  we conclude that
\begin{align*}
            \big\Vert I_0\mc{T}\big(\tilde h-h_0^\dag\big)\big\Vert_2^2\le 10\big\Vert I_0\mc{T}\big(h_0^\dag-h_{\lambda_{\mc{H}}}^\dag\big)\big\Vert_2^2+16c^2B_1^2\lambda_{\mc{G}}+16c^2c_0''\delta_n^2
         \end{align*}
         holds with probability at least $\kappa(\delta_n)$. Invoking again Theorem \ref{theo:biascontrol} to bound $\Vert I_0\mc{T}(h_0^\dag-h_{\lambda_{\mc{H}}}^\dag)\Vert_2^2$, we conclude that the event 
$ \Vert I_0\mc{T}(\tilde h-h_0^\dag)\Vert_2^2\le c_3\left(\delta_n^2+\lambda_{\mc{G}}+\lambda_{\mc{H}}^{\min\left\{\beta+1,2\right\}}\right) $ happens with probability at least $\kappa(\delta_n)$ 
    where $c_3=\max\{10 c_2^*\:,\:16c^2B_1^2\:,\:16c^2c_0''\}$. This completes the proof.
\end{proof}

\begin{proof}[Lemma~\ref{lemma:crit_radius_I0G}] Let $Q$ denote the conditional law of $O$ given $X\in\mc{X}_0$, where $O\sim P$. Then, for any $f\in\mc{L}^2(P)$, we have $\Vert I_0\cdot f\Vert_{\mc{L}^2(P)}^2=\mu_0\Vert f\Vert_{\mc{L}^2(Q)}^2$. Let $N := \sum_{i=1}^N I_0(X_i)$ and let $\mc{F}_0(\delta):=\{f\in\mc{F}:\Vert f\Vert_{\mc{L}^2(Q)}\le\delta/\sqrt{\mu_0}\}$. Then, letting $s=\delta/\sqrt{\mu_0}$, we have
\begin{align*}
   \mc{R}_n(I_0\cdot\mc{F},\:P,\:\delta) 
   &=\sum_{m=1}^n \E_{P,P_\varepsilon}\Big\{\sup_{f\in\mc{F}_0(\delta)}\frac{1}{n}\sum_{i=1}^n\varepsilon_iI_0(X_i)f(O_i)\Big\vert N=m\Big\}P(N=m)\\
     &=\sum_{m=1}^n \E_{P,P_\varepsilon}\Big\{\sup_{f\in\mc{F}_0(\delta)}\frac{1}{m}\sum_{i=1}^m\varepsilon_if(O_i)\Big\vert X_i\in\mc{X}_0,\forall \:i=1,\dots,m\Big\}\frac{m}{n}P(N=m)\\
     &=\sum_{m=1}^n \E_{Q,P_\varepsilon}\Big\{\sup_{f\in\mc{F}_0(\delta)}\frac{1}{m}\sum_{i=1}^m\varepsilon_if(O_i)\Big\}\frac{m}{n}P(N=m)\\
      &=\sum_{m=1}^n \mc{R}_m(\mc{F},Q,\:\delta/\sqrt{\mu_0})\:\frac{m}{n}\:P(N=m)=\E\left\{\mc{R}_N(\mc{F},Q,s)\frac{N}{n}\right\},
\end{align*}
where the second equality follows because the $O_1,\dots,O_n$ are i.i.d.

Let $\mc{D}_n$ be the event $\{N-n\mu_0<-n\mu_0/2\}$. By Hoeffding inequality we have $P(D_n)<\exp\{-2n^2\mu_0^2/(4n)\}=\exp(-dn)$ where $d:=\mu_0^2/2$. 
Then for $s:=\sup\{\delta_m: n\mu_0/2\le m\le n\}$
\begin{align*}
    \mc{R}_n(I_0\cdot\mc{F},P,\delta)&=\E\left\{\mc{R}_N(\mc{F},Q,s)\frac{N}{n}I_{\mc{D}_N}(N)\right\}+E\left\{\mc{R}_N(\mc{F},Q,s)\frac{N}{n}I_{\mc{D}_N^c}(N)\right\}\\
    &\le bP(\mc{D}_n)+\sup_{n\mu_0/2\le m\le n}\mc{R}_m(\mc{F},Q,s)\\
    &\le b\exp(-dn)+s^2.
\end{align*}
Let $\varepsilon$ be such that $1/\mu_0>1+\varepsilon^2$. Because by assumption $\delta_n^2=o(\exp\{-dn\})$, then there exists $n_0$ such that if $n\ge n_0$, $b\exp(-dn)\le\varepsilon^2s^2$. On the other hand, our choice of $\varepsilon$ implies that  $\mc{R}_n(I_0\cdot\mc{F},P,\delta)>\mc{R}_n(I_0\cdot\mc{F},P,s\sqrt{1+\varepsilon^2})$. Then, the last display implies that $\mc{R}_n(I_0\cdot\mc{F},P,s\sqrt{1+\varepsilon^2})<(1+\varepsilon^2)s^2$. Therefore, $\delta_n'\le s\sqrt{1+\varepsilon^2}=O(\alpha_{\lfloor n\mu_0/2\rfloor})$ by the assumed monotonicity of $\alpha_n$.

\end{proof}

\begin{proof}[Verifying Assumption~\ref{ass:bounded} for the Map $\widetilde{m}(o;h)$ from Example~\ref{example_mtp}]
    
We denote the support of random variables $(V_1,\dots,V_k)$ and their conditional versions $V_1|(V_2=v_2,\dots,V_k=v_k)$, for $k\ge2$, as $\text{supp}(V_1,\dots,V_k)$ and $\text{supp}(V_1|V_2=v_2,\dots,V_k=v_k)$.

 Observe that $q$, defined in Example~\ref{example_mtp}, is strictly monotone and differentiable a.e. with respect to $P_{0,A}$, and admits an inverse that is differentiable a.e. with positive derivative.

Suppose that $\tilde\rho_{P_0}(a,l,w) = I_0(a,l)\cdot\frac{dq^{-1}(a)}{da}\cdot\frac{p_{0,A|L,W}\{q^{-1}(a)\mid l,w\}}{p_{0,A|L,W}(a|l,w)}$, where $I_0(a,l):=I\{a\in[c+\delta,d]\}$, is bounded by some $B_0>0$. Then, for any $h\in\mathcal{L}^2(A,L,W)$,
\begin{align*}
    \Vert\widetilde{m}(O;h)\Vert_2^2&=\iint \left[\int h\{q(a),l,w\}^2p_{0,A|L,W}(a|l,w)da\right]dP_0(l,w)\\
    &\stackrel{(i)}{=}\iiint h(\tilde a,l,w)^2\cdot I\{\tilde a\in[c+\delta,d]\}\cdot \frac{dq^{-1}(\tilde a)}{d\tilde a}p_{0,A|L,W}(q^{-1}(\tilde a)|l,w)\:d\tilde a\cdot dP_{0}(l,w)\\
    &\stackrel{(ii)}{=}\iiint h(\tilde a,l,w)^2\cdot I_0(\tilde a,l)\cdot\tilde \rho_{P_0}(\tilde a,l,w)p_{0,A|L,W}(\tilde a|l,w)\:d\tilde a\cdot dP_{0}(l,w)\\
    &\stackrel{(iii)}{\le} B_0\iiint I_0(\tilde a,l)\cdot h(\tilde a,l,w)^2p_{0,A|L,W}(\tilde a|l,w)\:d\tilde a\cdot dP_0(l,w)=B_0\cdot \Vert I_0h\Vert_2^2,
\end{align*}
    where the equality in (i) follows from the change of variable $\tilde a = q(a)$, which is valid since $q$ has an a.e. differentiable inverse with strictly positive derivative and maps $[c,d]$ onto $[c+\delta,d]$. Equality (ii) follows directly from the definition of $\tilde\rho_{P_0}$, and the inequality (iii) uses the assumption that $|\tilde\rho_{P_0}(a,l,w)|\le B_0$ for all $(a,l,w)\in\text{supp}(A,L,W)$. This proves the first inequality in part 1 of Assumption~\ref{ass:bounded} with constant $B_2=\sqrt{B_0}$.

    To verify part 2 of Assumption~\ref{ass:bounded}, we impose the following \textit{common support} assumption: $ \text{supp}(A|L=l,U=u)=\text{supp}(A|L=l)$ for all $(l,u)\in\text{supp}(L,U)$, which is standard in proximal inference (see \cite{olivas2025chap3}). This assumption, together with the fact that $W$ is a negative control outcome-- i.e., $(A,Z)\perp W|L,U$ -- imply that $\text{supp}(A|L=l,W=w)=\text{supp}(A|L=l)$ for all $(l,w)\in\text{supp}(L,W)$. Moreover, $\text{supp}(A,L,W)=[c,d]\times\text{supp}(L,W)$ given that $\text{supp}(A|L=l)=[c,d]$ for all $l\in\mc{L}$.
    
    Next, we define an equivalence relation on $\mc{H}$ as follows: $ h\leftrightarrow h'$ if and only if $h(a,l,w)=h'(a,l,w)$ for all $ (a,l,w)\in[c+\delta, d]\times\text{supp}(L,W)$. Let $\widetilde{\mc{H}}:=\{[h]:h\in\mc{H}\}$ denote the set of equivalences classes in $\mc{H}$ under this relation. Note that $h\leftrightarrow h'$ iff $h-h'\in[0]:=\{h\in\mc{H}:h(a,l,w)=0\:\forall (a,l,w)\in[c+\delta,d]\times\text{supp}(L,W)\}$. In the special case of a Gaussian RKHS, the zero function is the only function that vanishes on a non-empty open subset (see Corollary 4.44 of \cite{steinwart2008support}); hence each equivalence class contains a single function, and $\widetilde{\mc{H}}$ coincides with $\mc{H}$. Define the map $\tilde d:\widetilde{\mc{H}}\times\widetilde{\mc{H}}\mapsto\mathbb{R}$ by
    \[
    \tilde d([h_1],[h_2]):=\inf\{\Vert h_1-h_2+v\Vert_\mc{H}\mid v\in[0]\}.
    \]
    Note that $\tilde d$ is well-defined because the set in the right hand side is invariant to the choice of representer $h_1\in[h_1]$ and $h_2\in[h_2]$.
    We now verify that $\tilde d$ is a metric on $\widetilde{\mc{H}}$. By construction, $\tilde d$ is symmetric. To see that the triangle inequality holds, given $\epsilon >0$ and $[h_1],[h_2],[h_3]\in\widetilde{\mc{H}}$, pick $v,v'\in[0]$ such that $\Vert h_1-h_2+v\Vert_\mc{H} < \tilde d([h_1],[h_2]) + \epsilon$ and $\Vert h_2-h_3+v'\Vert_\mc{H} < \tilde d([h_2],[h_3]) + \epsilon$. Then, $\Vert h_1-h_3+v+v'\Vert_\mc{H} \le \tilde d([h_1],[h_2])+d([h_2],[h_3])+ 2 \epsilon$. Because $v+v'\in[0]$, it follows that $\tilde d([h_1],[h_3]) \le \tilde d([h_1],[h_2])+\tilde d([h_2],[h_3]) + 2 \epsilon$. Since this holds for all $\epsilon >0$, the triangular inequality follows.

    Finally, we show that $\tilde d([h_1],[h_2])=0$ if and only if $[h_1]=[h_2]$. The if direction is trivial since $\tilde d([h],[h])=\inf\{\Vert v\Vert\mid v\in[0]\}=0$ for all $[h]\in\widetilde{\mc{H}}$. To show the only if direction suppose that $\tilde d([h_1],[h_2])=0$ but $[h_1]\neq[h_2]$. Then there exist representatives $h_1\in[h_1]$, $h_2\in[h_2]$, and a point $(a_0,l_0,w_0)\in[c+\delta,d]\times\text{supp}(L,W)$ such that $ h_1(a_0,l_0,w_0)\neq h_2(a_0,l_0,w_0)$. Let $\epsilon:=|h_1(a_0,l_0,w_0)- h_2(a_0,l_0,w_0)|>0$, and $ k:=[\sup_{(a,l,w)\in\text{supp}(A,L,W)}K_{\mc{H}}\{(a,l,w),(a,l,w)\}]^{1/2}<\infty$. By the definition of $\tilde d$ and the fact that $\tilde d([h_1],[h_2])=0$, there exists $v'\in[0]$ such that $\Vert h_1-h_2+v'\Vert_\mc{H}<\epsilon/k$. Let $h_1':=h_1+v'$. Then we arrive at the following contradiction:
 \begin{align*}
        k\Vert h_1'-h_2\Vert_\mc{H}<|h_1(a_0,l_0,w_0)-h_2(a_0,l_0,w_0)|=|h'_1(a_0,l_0,w_0)-h_2(a_0,l_0,w_0)|\le k\Vert h_1'-h_2\Vert_\mc{H},
    \end{align*}
 where the equality holds because $h_1\leftrightarrow h'_1$ and the final inequality because for any $h\in\mc{H}$, $\Vert h\Vert_\infty\le k\:\Vert h\Vert_\mc{H}$. Hence, $\tilde d([h_1],[h_2])=0$ implies $[h_1]=[h_2]$.

 We now show that $(\widetilde{\mc{H}},d)$ is separable. Since $K_\mc{H}$ is continuous and $\text{supp}(A,L,W)$ is compact, $\mc{H}$ is separable under the metric induced by its norm (see Proposition 11.7 of \cite{paulsen2016introduction}). Let $\{h_n\}\subset\mc{H}$ be a countable dense subset, and consider the set of equivalence classes $\{[h_n]\}$. We claim this set is dense in $\widetilde{\mc{H}}$. Indeed, for any $[h]\in\widetilde{\mc{H}}$ and $\epsilon>0$, there exists a representative $h'\in[h]$ and some $h_n\in\mc{H}$ such that $\Vert h'-h_n\Vert_\mc{H}<\epsilon$. Then $d([h],[h_n])\le\Vert h-h_n\Vert_\mc{H}<\epsilon$, so $\{[h_n]\}$ is dense in $\widetilde{\mc{H}}$, and thus $\widetilde{\mc{H}}$ is separable.

    Next, consider the class $\widetilde{m}\circ \mc{H}:=\{o\mapsto h(q(a),l,w)\}$, which is in a one-to-one correspondence with $\widetilde{\mc{H}}$. More precisely, the map $h\mapsto \widetilde{m}(\cdot,h)$ is well-defined on equivalence classes, since $ \widetilde{m}(o;h_1)=\widetilde{m}(o;h_2)$ for all $o\in\mc{O}$ iff $h_1(q(a),l,w)=h_1(q(a),l,w)$ for all $(a,l,w)\in[c,d]\times\text{supp}(L,W)$, $h_1(a,l,w)=h_2(a,l,w)$ for all $(a,l,w)\in[c+\delta,d]\times\text{supp}(L,W)$ (since $q
    \left([c,d]\right)=[c+\delta,d]$) iff  $h_1\leftrightarrow h_2$. Thus, we can equip $\widetilde{m}\circ\mc{G}$ with the metric $d$ defined as:
    \[
    d(f_1,f_2):=\tilde d([h_{f_1}],[h_{f_2}])\quad \text{ for any }f_1,f_2\in\widetilde{\mc{H}},
    \]
    where $[h_{f_1}]$ and $[h_{f_2}]$ are their corresponding equivalence classes in $\widetilde{\mc{H}}$. This
    yields $(\widetilde{m}\circ\mc{H},d)$ a separable metric space. By construction, the metric $d$ is controlled by the RKHS norm.

    Finally, we show that the metric $d$ dominates the sup-norm on $\widetilde{m}\circ\mc{H}$. For any $f_1,f_2\in \widetilde{m}\circ\mc{H}$, let $[h_{f_1}]$ and $[h_{f_2}]$ be their corresponding equivalence classes, and take any representatives $h_{f_1}\in [h_{f_1}]$ and $h_{f_2}\in [h_{f_2}]$. Then, for any $o\in\mc{O}$, we have
    \begin{align*}
        |f_1(o)-f_2(o)|:=|\widetilde{m}(o;h_{f_1})-\widetilde{m}(o;h_{f_2})|&=|h_{f_1}(q(a),l,w)-h_{f_2}(a,l,w)|\le k\Vert h_{f_1}-h_{f_2}\Vert_\mc{H}.
    \end{align*}
    Since the representatives were arbitrary, we conclude that $ \sup_{o\in\mc{O}} |f_1(o)-f_2(o)|\le k\cdot \tilde d([h_1],[h_2])=k\cdot d(f_1,f_2)$, implying that $d$ dominates the metric induced by the sup-norm with constant $B_3=k$. This completes the verification of part 2 of Assumption~\ref{ass:bounded}.
    \end{proof}

\begin{proof}[Theorem~\ref{theo:main}]
    The proof follows analogously to that of Theorem~\ref{theo:conv}, but applies Lemma~\ref{lemma:erm_mg} in place of Lemma~\ref{lemma:erm}, and uses the updated constants defined in Lemma~\ref{lemma:erm_mg}.
\end{proof}

\begin{proof}[Lemma~\ref{lemma:crit_radius_mG}] Let $P_0^q:=P_0\circ q^{-1}$ be the pushforward measure of $P_0$ under $q$. For any $g\in\mc{G}_{B_1}$, the composition $m(\:\cdot\:,g)=g\circ q(\cdot)$ satisfies
\[
\Vert m(\:\cdot\:,g)\Vert_{\mc{L}^2(P_0)}^2=\Vert g\circ q \Vert_{\mc{L}^2(P_0)}^2=\int_\mc{Z}g\{q(z)\}^2dP_0(z)=\int_{q(\mc{Z})}g(\tilde z)^2d(P_0\circ q^{-1})(\tilde z)=\Vert g\Vert_{\mc{L}^2(P_0^q)}^2,
\]
so the map $g\mapsto g\circ q=m(\:\cdot\:,g)$ is an isometry between $(\mc{G}_{B_1},\mc{L}^2(P_0^q))$ and $(m\circ\mc{G}_{B_1},\mc{L}^2(P_0))$.

Observe that the Rademacher complexity is preserved under reparameterization: if  $Z_1,\dots,Z_n\sim P_{0}$, $\tilde Z_1,\dots,\tilde Z_n\sim P_0^q$, and $\varepsilon_1,\dots,\varepsilon_n$ are i.i.d. Rademacher variables, then $\frac{1}{n}\sum_{i=1}^n\varepsilon_ig\{q(Z_i)\}^2=\frac{1}{n}\sum_{i=1}^n\varepsilon_ig(\tilde Z_i)^2$,
so for all $\delta>0$, $\mathcal{R}_n(m\circ \mathcal{G}_{B_1},\:P_0,\:\delta) = \mathcal{R}_n(\mathcal{G}_{B_1},\:P_0^q,\:\delta)$. Then the critical radius of $m\circ\mc{G}_{B_1}$ under $P_0$ equals the critical radius of $\mc{G}_{B_1}$ under $P_0^q$. The result follows because the critical radius of an RKHS class is preserved up to multiplicative constants under changes of measure. That $o \mapsto m(o, g) = g(q(z))$ is star shaped follows because $o \mapsto m(o, g) = g(q(z))$ is linear in $g$ and $\mc{G}_{B_1}$ is star-shaped.
\end{proof}

\begin{proof}[Ensuring Bounds in Theorems~\ref{theo:conv} and \ref{theo:main} are Achieved with High Probability]
    When $\delta_n<5b/6$, we show that $\log(\log(b/\delta_n))\le c_1n\delta_n^2/2$, where $c_1$ is the constant defined in Theorem~\ref{theo:conv}, for a suitable choice of $\delta_n$. Suppose $\delta_n^2\ge c_0\log(\log(n))/n$ with $c_0=\max\{b^2/\log(\log(3)),2/c_1\}$. Since $\delta_n/b<1$, it holds that $\delta_n/b\ge(\delta_n/b)^2\ge c_0\log(\log(n))/(nb^2)$. If $n\ge 3$, the assumption $c_0\ge b^2/\log(\log(3))$ implies $\delta_n/b\ge \log(\log(n))/\{n\log(\log(3))\}\ge1/n$ which then implies $ n\ge b/\delta_n$. Furthermore, since $c_0\ge 2/c_1$, we have $\log(\log(b/\delta_n))\le \log(\log(n))\le n\delta_n^2/c_0\le c_1n\delta_n^2/2$,
as desired.
\end{proof}

\begin{proof}[Deriving the adjoint of $T_\mc{G} \circ \mc{T} : \mc{H} \rightarrow \mc{G}$]
    For $h\in\mc{H}$ and $g\in\mc{G}$, applying Lemma~\ref{lemma:null_space}(g) with the integral operators induced by $K_\mc{H}$ and $K_\mc{G}$ yields $\left\langle T_\mc{G} \circ \mc{T} h,g \right\rangle_\mc{G}=\langle \mc{T}h,g\rangle_{\mc{L}^2(P_Z)}=\langle h,\mc{T}^*g\rangle_{\mc{L}^2(P_W)}=\left\langle h,T_\mc{H}\circ \mc{T}^*g\right\rangle_\mc{H}$.
Hence, the adjoint of $T_\mc{G} \circ \mc{T}$ is $T_\mc{H} \circ \mc{T}^*$.
\end{proof}

\section{Supporting Lemmas}\label{sec:proof_lemmas}

  The following Lemma establishes several important properties of the operator $T_{\mc{H}}^{1/2}$ that are invoked in the proofs of Lemmas~\ref{lemma:i0t_bounded} and~\ref{lemma:min_rkhs_norm_sol_h}, Theorem~\ref{theo:biascontrol}, and in the derivation of expression~\eqref{argmin_h}. Let $\mc{N}(T_\mc{H}):=\left\{h\in\mc{L}^2(P_{0,W}):T_{\mc{H}}h=0\right\}$ denote the null space of the operator $T_{\mc{{H}}}$, and let $\mc{N}(T_\mc{H})^\perp$ denote its orthogonal complement in $\mc{L}^2(P_{0,W})$. 

  \begin{lemma} Suppose Assumption \ref{ass:K_H} holds. The following holds:
        \begin{enumerate}
        \item[a)] $\mc{N}(T_{\mc{H}})=\mc{N}(T_{\mc{H}}^{1/2})$.
        \item[b)] For any $h\in \mc{L}^2(P_{0,W})$, it holds that $T_{\mc{H}}^{1/2}h\in\mc{H}$.
        \item[c)] The linear operator $T_{\mc{H}}^{1/2}:\mc{L}^2(P_{0,W})\rightarrow \mc{H}$ is bounded and satisfies $\Vert\mc{T}_{\mc{H}}^{1/2}h\Vert_{\mc{H}} \leq \Vert h\Vert_2$.
        \item[d)] The restriction of $T_{\mc{H}}^{1/2}$ to $\mc{N}(T_\mc{H})^\perp$ is an isometry, i.e. $\left\Vert  T_{\mc{H}}^{1/2} h\right\Vert_{\mc{H}}=\Vert h\Vert_2$ for all  $h\in\mc{N}(T_\mc{H})^\perp$.
        \item[e)] The operator $T_{\mc{H}}^{1/2}:\mc{N}(T_\mc{H})^\perp \rightarrow\mc{H}$ is bijective.
        \item[f)] The operator $T_{\mc{H}}^{1/2}:\mc{L}^2(P_{0,W}) \rightarrow  \mc{H}$, with its range endowed with the $\mc{L}^2(P_{0,W})$--norm, is compact.
         \item[g)] For any $h'\in\mc{L}^2(P_W)$ and $h\in\mc{H}$, it holds that $\langle T_\mc{H}h',h\rangle_\mc{H}=\langle h',h\rangle_{\mc{L}^2(P_W)}$.
        \end{enumerate}
        \label{lemma:null_space}
  \end{lemma}
  
\begin{proof}
As discussed in Section \ref{sec:conv_analysis}, Assumption \ref{ass:K_H} implies that the integral operator $T_\mc{H}:\mc{L}^2(P_{0,W})\mapsto\mc{L}^2(P_{0,W})$ is compact and positive. Therefore, by the spectral decomposition theorem (see, e.g., Theorem 4.3 of \cite{hohage2002lecture}) there exists eigenfunctions $(\varphi_j)_{j=1}^N$ corresponding to all the non-zero eigenvalues $ (\eta_j)_{j=1}^N$ of $T_{\mc{H}}$, where $N\in\mathbb{N}\cup\left\{+\infty\right\}$. Such eigenfunctions are orthonormal in $\mathcal{L}^2(P_{0,W})$ and satisfy that for any $h \in \mathcal{L}^2(P_{0,W})$, it holds that $T_{\mc{H}}h=\sum_{j=1}^N\eta_j\langle h,\varphi_j\rangle_2\varphi_j$. Because the integral operator $T_\mc{H}$ is positive, its non-zero eigenvalues $ (\eta_j)_{j=1}^N$ are positive. When $N=\infty$, $\eta_j \rightarrow 0$  as $j \rightarrow \infty$ and the sequence $ (\eta_j)_{j=1}^N$ is arranged in decreasing order. The square root of $T_{\mc{H}}$, denoted by $T_{\mc{H}}^{1/2}$, is defined by $T_{\mc{H}}^{1/2}h:=\sum_{j=1}^N \sqrt{\eta_j}\langle h,\varphi_j\rangle_2\varphi_j$.

Assertion (a) follows from the fact that for any $h\in\mc{L}^2(P_{0,W})$: $T_{\mc{H}}h=\sum_{j=1}^N\eta_j\langle h,\varphi_j\rangle_2\varphi_j=0$ iff $\langle h,\varphi_j\rangle_2=0\:\forall j=1,\dots,N$ iff $T_{\mc{H}}^{1/2}h=\sum_{j=1}^N\sqrt{\eta_j}\langle h,\varphi_j\rangle_2\varphi_j=0$.

Another result (see Theorem 11.18 of \cite{paulsen2016introduction}), establishes that under Assumption~\ref{ass:K_H}, the sequence $(\sqrt{\eta_j}\varphi_j)_{j=1}^N$ forms an orthogonal basis of the RKHS $\mc{H}$. Furthermore,  $\mc{H}$ admits the representation 
\begin{align}
 \mc{H}=\left\{h=\sum_{j=1}^N\gamma_j\varphi_j \: : \: \text{ for some } (\gamma_j)_{j=1}^{\infty}\text{ with }\sum_{j=1}^N\frac{\gamma_j^2}{\eta_j}<\infty\right\},\label{H_rep}
\end{align}
where if $N=\infty$, the limit of the finite sums is with respect to the $\mathcal{L}^2(P_{0,W})$-metric. 
Additionally, its inner product decomposes as (see Theorem 11.3 of \cite{paulsen2016introduction}) $\langle h,h'\rangle_{\mc{H}}=\sum_{j=1}^N\langle h,\varphi_j\rangle_2\langle h',\varphi_j\rangle_2/\eta_j$.

Assertion (b) follows from the representation of $\mc{H}$ in~\eqref{H_rep} and the definition of $T_{\mc{H}}^{1/2} h$. Specifically, for any $h \in \mc{L}^2(P_{0,W})$, we have $T_{\mc{H}}^{1/2}h=\sum_{j=1}^N\sqrt{\eta_j}\langle h,\varphi_j\rangle_2\varphi_j$, so to show that $T_{\mc{H}}^{1/2}h$ is in $\mc{H}$ we need to show that $\sum_{j=1}^N(\sqrt{\eta_j}\langle h,\varphi_j\rangle_2)^2/\eta_j < \infty$. This holds because
\begin{equation}
\sum_{j=1}^N\frac{\left(\sqrt{\eta_j}\langle h,\varphi_j\rangle_2\right)^2}{\eta_j}=\sum_{j=1}^N\langle h,\varphi_j\rangle_2^2\le \Vert h\Vert_2^2<\infty.\label{eq:th_norm}
\end{equation}
 Assertion (c) follows from~\eqref{eq:th_norm} since $\Vert T^{1/2}_\mc{H}h\Vert_\mc{H}^2 = \sum_{j=1}^N\left(\sqrt{\eta_j}\langle h,\varphi_j\rangle\right)^2/\eta_j$, thus proving that $T^{1/2}_\mc{H}$ is bounded with operator norm at most one. The linearity of $T^{1/2}_\mc{H}$ is immediate from the linearity of the inner products \(\langle h, \varphi_j \rangle_2\), completing the proof of (c).

To show Assertions (d) and (e), we use that $\mc{N}(T_\mc{H})^\perp=\text{span}(\varphi_1,\dots,\varphi_N)$. Assertion (d) follows because, for $h\in\text{span}(\varphi_1,\dots,\varphi_N)$, the first inequality in~\eqref{eq:th_norm} becomes an equality.

To prove Assertion (e), we first show that $T_{\mc{H}}^{1/2}:\mc{N}(T_\mc{H})^\perp \rightarrow \mc{H}$ is injective. Let $h_1,h_2\in\mc{N}(T_\mc{H})^\perp$ be such that $T_\mc{H}^{1/2}h_1=T_\mc{H}^{1/2}h_2$. Then, $h_1-h_2\in\mc{N}(T_\mc{H}^{1/2})$. By part (a), $\mc{N}(T_\mc{H}^{1/2}) = \mc{N}(T_\mc{H})$. Then, $h_1-h_2\in\mc{N}(T_\mc{H}) \cap \mc{N}(T_\mc{H})^\perp =\{0\}$. This shows injectivity. Surjectivity follows from the representation of $\mc{H}$ in (\ref{H_rep}). Specifically, any $h\in\mc{H}$ can be written as $h=\sum_{j=1}^N\gamma_j\varphi_j$ with $\sum_{j=1}^N\frac{\gamma_j^2}{\eta_j}<\infty$.
Then $T_\mc{H}^{1/2}h'=h$ where $h':=\sum_{j=1}^N\gamma_j\varphi_j/\sqrt{\eta_j}\in\mc{N}(T_\mc{H})^\perp$.

Assertion (f) follows from the definition of $\mc{T}_{\mc{H}}^{1/2}:\mc{L}^2(P_{0,W})\mapsto\mc{L}^2(P_{0,W})$, as it is the norm-limit of finite-rank operators (see Proposition 5.6.5 of \cite{conway2012course}).

To prove assertion (g), observe that for any $h'\in\mc{L}^2(P_W)$ and $h\in\mc{H}$, it holds that $ \langle T_\mc{H}h',h\rangle_\mc{H}=\sum_{j=1}^N\langle T_\mc{H}h',\varphi_j\rangle_2\langle h,\varphi_j\rangle_2/\eta_j=\sum_{j=1}^N\langle h',\varphi_j\rangle_2\langle h,\varphi_j\rangle_2=\langle h',\sum_{j=1}^N\langle h,\varphi_j\rangle_2\varphi_j\rangle_2$.
For every $h\in\mc{H}$, $h=\sum_{j=1}^N\langle h,\varphi_j\rangle_2\varphi_j$. Hence, the last expression equals $\langle h',h\rangle_2$.
\end{proof}
\begin{lemma}
    Suppose Assumption~\ref{ass:K_H} holds. Then the operators $I_0^{op}\circ\mc{T}\mid_\mc{H}:\mc{H} \rightarrow \mc{L}^2(P_{Z})$ and $I_0^{op}\circ\mc{T}\circ T_\mc{H}^{1/2}:\mc{L}^2(P_{W})\rightarrow\mc{L}^2(P_{Z})$ are bounded and linear, and their adjoints $(I_0^{op}\circ\mc{T}\mid_\mc{H})^*:\mc{L}^2(P_Z)\rightarrow\mc{H}$ and $(I_0^{op}\circ\mc{T}\circ T_\mc{H}^{1/2})^*:\mc{L}^2(P_Z)\rightarrow\mc{L}^2(P_W)$ are given by $T_\mc{H}\circ\mc{T}^*\circ I_0^{op}$ and $T_\mc{H}^{1/2}\circ\mc{T}^*\circ I_0^{op}$. \label{lemma:i0t_bounded}
\end{lemma}

\begin{proof}
    Linearity is immediate from that of $I_0^{op}$, $\mc{T}$, and $T_\mc{H}^{1/2}$. Boundedness of $I_0^{op}\circ\mc{T}\mid_\mc{H}$ follows because for $h \in \mc{H}$, it holds that $\Vert I_0^{op}\mc{T}h\Vert_2^2\le\Vert r\Vert_\infty^2\Vert h\Vert_2^2\le\Vert r\Vert_\infty^2\Vert h\Vert_\infty^2\le k_1^2\Vert r\Vert_\infty^2\Vert h\Vert_\mc{H}^2$,
   where $k_1^2=\sup_{w\in\mc{W}}K_\mc{H}(w,w)<\infty$. Boundedness of $I_0^{op}\circ\mc{T}\circ T_\mc{H}^{1/2}$ follows because $T_\mc H^{1/2}$(by Lemma~\ref{lemma:null_space}(f)), $\mc T$, and $I_0^{op}$ are all bounded, and so is their composition.

   To determine the adjoint of $I_0^{op}\circ\mc{T}\mid_\mc{H}$, let $h\in\mc{H}$ and $g\in\mc{L}^2(P_Z)$. Then, applying Lemma~\ref{lemma:null_space}(g) yields $\langle I_0^{op}\circ\mc{T}h,g\rangle_{\mc{L}^2(P_Z)}=\langle \mc{T}h,I_0^{op}g\rangle_{\mc{L}^2(P_Z)}=\left\langle h,\mc{T}^*\circ I_0^{op}g\right\rangle_{\mc{L}^2(P_W)}=\left\langle h,T_\mc{H}\circ\mc{T}^*\circ I_0^{op}g\right\rangle_\mc{H}$. Hence, $(I_0^{op}\circ\mc{T}\mid_\mc{H})^*=T_\mc{H}\circ\mc{T}^*\circ I_0^{op}$. Now let $h\in\mc{L}^2(P_W)$ and $g\in\mc{L}^2(P_Z)$. Since $T_\mc{H}^{1/2}$ is self-adjoint, we have $ \langle I_0^{op}\circ\mc{T}\circ T_\mc{H}^{1/2}h,g\rangle_{\mc{L}^2(P_Z)}=\langle T_\mc{H}^{1/2}h,\mc{T}^*\circ I_0^{op}g\rangle_{\mc{L}^2(P_W)}=\langle h,T_\mc{H}^{1/2}\circ\mc{T}^*\circ I_0^{op}g\rangle_{\mc{L}^2(P_W)}$. Thus $(I_0^{op}\circ\mc{T}\circ T_\mc{H}^{1/2})^*=T_\mc{H}^{1/2}\circ\mc{T}^*\circ I_0^{op}$.
\end{proof}

\begin{lemma}
The second equality in display~\eqref{argmin_h} holds.
    \label{lemma:h_lambda}
\end{lemma}

\begin{proof}
    To show the second equality in~\eqref{argmin_h} it suffices to show that $\Vert h^{\dag\dag}_{\lambda_\mc{H},KRAS}\Vert_\mc{H}\le B$ where $h^{\dag\dag}_{\lambda_\mc{H},KRAS}:=\arg\min_{h\in\mc{H}}\left\{\frac{1}{4}\left\Vert\rho- I_0^{op} \circ \mc{T}\vert_\mc{H} \:h\right\Vert_2^2+\lambda_{\mc{H}}\left\Vert h\right\Vert_{\mc{H}}^2\right\}$. This holds because $\Vert h^{\dag}_{0,KRAS}\Vert_\mc{H}<B$ by definition of $B$ and $\Vert h^{\dag\dag}_{\lambda_\mc{H},KRAS}\Vert_\mc{H}\le \Vert h^{\dag}_{0,KRAS}\Vert_\mc{H}$ because otherwise $ \frac{1}{4c^2}\Vert I_0\mc{T}(h_{0,KRAS}^\dag-h_{0,KRAS}^\dag)\Vert_2^2+\lambda_{\mc{H}}\big\Vert h_{0,KRAS}^\dag\big\Vert_{\mc{H}}^2<\frac{1}{4c^2}\Vert I_0\mc{T}(h_{0,KRAS}^\dag-h_{\lambda_\mc{H},KRAS}^{\dag\dag})\Vert_2^2+\lambda_{\mc{H}}\big\Vert h_{\lambda_\mc{H},KRAS}^{\dag\dag}\big\Vert_{\mc{H}}^2$ violating the definition of $h_{\lambda_\mc{H},KRAS}^{\dag\dag}$.
\end{proof}

\begin{lemma}
        Under Assumptions \ref{ass:relizability} and \ref{ass:K_H}, $h_{0,KRAS}^\dag=T_{\mc{H}}^{1/2}h'_{0,KRAS}$ and $h'_{0,KRAS}\in\mc{N}(T_\mc{H})^\perp$.\label{lemma:min_rkhs_norm_sol_h}
    \end{lemma}

    \begin{proof} To simplify notation, we drop the subscripts $KRAS$ from $h_{0,KRAS}^\dag$ and $h'_{0,KRAS}$. Under Assumptions~\ref{ass:relizability}--\ref{ass:K_H}, $h_0^\dag$ in~\eqref{def:h_0_dag} is well-defined, and $\widetilde{\mc{T}}$ is bounded (see Section~\ref{sec:conv_analysis}). 

By Lemma~\ref{lemma:null_space} (e), there exists a unique $h_0'\in\mc{N}(T_\mc{H})^\perp$ such that $h_0^\dag=T_\mc{H}^{1/2}h_0'$. Since $\rho=\mc{T}\left(I_0h_0^\dag\right)=I_0^{op}\circ\mc{T}h_0^\dag=I_0^{op}\circ\mc{T}\circ T_\mc{H}^{1/2}h_0'=\widetilde{\mc{T}}h_0'$
it follows that $h_0'$ solves equation~\eqref{eq:rkhs_int_eq_h}. It remains to show that it is the solution of minimal $\mc{L}^2(P_{W})-$norm. Let $h_1\in\mc{L}^2(P_{W})$ be any solution to~\eqref{eq:rkhs_int_eq_h}. Decompose $h_1$ as $h_1'+h_1^\perp$, where $h_1'\in\mc{N}(T_\mc{H})$ and $h_1^\perp\in\mc{N}(T_\mc{H})^\perp$. Since $\widetilde{\mc{T}}h_1^\perp =\widetilde{\mc{T}}h_1$, it follows that $h_1^\perp$ also solves equation~\eqref{eq:rkhs_int_eq_h}. Moreover, $\rho=\widetilde{\mc{T}}h_1^\perp=I_0^{op}\circ\mc{T}\circ T_\mc{H}^{1/2}h_1^\perp =\mc{T}(I_0T_\mc{H}^{1/2}h_1^\perp)$, so $T_\mc{H}^{1/2}h_1^\perp\in\mc{H}_0$. By Lemma~\ref{lemma:null_space} (d), $T_\mc{H}^{1/2}$ restricted to $\mc{N}(T_\mc{H})^\perp$ is an isometry, so $\Vert h_1^\perp\Vert_2=\Vert T^{1/2}_\mc{H}h_1^\perp\Vert_\mc{H}\ge \Vert h_0^\dag\Vert_\mc{H}^2=\Vert h_0'\Vert_2^2$ with equality iff $T^{1/2}_{\mc{H}}h_1^\perp=h_0^\dag=T_\mc{H}^{1/2}h_0'$, which by Lemma~\ref{lemma:null_space} (e) implies $h_1^\perp=h_0'$. Therefore, $\Vert h_1\Vert_2^2=\Vert h_1'\Vert_2^2+\Vert h_1^\perp\Vert_2^2\ge \Vert h_0'\Vert_2^2$, with equality, i.e., $\Vert h_1\Vert_2=\Vert h_0'\Vert_2$, if and only if $h_1'=0$ and $h_1^\perp=h_0'$. This shows that $h_0'$ is the unique solution to equation~\eqref{eq:rkhs_int_eq_h} of minimal $\mc{L}^2(P_{W})$-norm.
\end{proof}

Next we will present several results that are used in the proofs of Theorems~\ref{theo:conv} and \ref{theo:main}.

In the following Lemmas \ref{lemma:talagrand}--\ref{lemma:lemma14foster}, $O_1,\dots,O_n$ are $n$ independent identically distributed copies of a random variable $O$ on the sample space $\mc{O}$. 
Furthermore, $\mc{F}$ is an index set such that for each $f \in \mc{F}$,  $t_f:\mc{O} \mapsto \mathbb{R}$ is a measurable, real-valued function satisfying $\mathbb{E}[t_f(O)^2]<\infty$. 

\begin{lemma}[Talagrand inequality, Theorem 3.27 of \cite{wainwright2019high}]
    Suppose $\mc{F}$ is countable and $\left\{t_f:f\in\mc{F}\right\}$ is uniformly bounded by $b$. Define $Z:=\sup_{f\in\mc{F}}\left\{\frac{1}{n}\sum_{i=1}^nt_f(O_i)\right\}$. Then, for all $u>0$, $\text{P}\left\{Z\ge\mathbb{E}(Z)+u\right\}\le2\exp\Big(-\frac{nu^2}{16eb\mathbb{E}(Z)+8e\sup_{f\in\mc{F}}\mathbb{E}\left\{t_f(O)^2\right\}+4bu}\Big)$
    where $e:=exp(1)$.\label{lemma:talagrand}
\end{lemma}

Lemma ~\ref{lemma:talagrand} assumes that $\mc{F}$ is countable. However, the proofs of Theorems \ref{theo:conv} and \ref{theo:main} require applying the lemma to  uncountable sets $\mc{F}$. This extension is valid when $\left\{t_f:f\in\mc{F}\right\}$ is a Caratheodory set, i.e. when $\mc{F}$ can be equipped with a metric under which it is separable and the evaluation map $f \mapsto t_f(o)$ is continuous for each $o \in \mc{O}$, a condition satisfied by the function classes we consider. For completeness, we now state and prove some results that imply that the Lemma \ref{lemma:talagrand} remains valid under this more general setting.

To facilitate the extension of Lemma \ref{lemma:talagrand} to certain uncountable uniformly bounded function classes $\left\{t_f:f\in\mc{F}\right\}$ used in the proof of Theorem \ref{theo:conv}, in the next Lemma we adapt a result from \cite{steinwart2008support}.

\begin{lemma}
Suppose that:
\begin{enumerate}
    \item $\mc{F}$ is a separable metric space with metric $d$, and let $\mc{S}\subset\mc{F}$ be a countable dense set,
    \item For each $f\in\mc{F}$, $t_f:\mc{O}\mapsto\mathbb{R}$ is a measurable function,  where $\mc{O}$ is the sample space of a random vector $O$,
    \item For each $o\in\mc{O}$, the map $f\mapsto t_f(o)$ is continuous with respect to $d$ at all $f \in \mc{F}$.
\end{enumerate}
Then for all $o_1,\dots,o_n\in\mc{O}$, $\sup_{f\in\mc{F}}\left\{\frac{1}{n}\sum_{i=1}^nt_f(o_i) \right \} = \sup_{f\in\mc{S}}\left\{\frac{1}{n}\sum_{i=1}^nt_f(o_i)\right\}$ and, consequently, for $O_1,\cdots,O_n$ i.i.d. copies of $O$, $\sup_{f\in\mc{F}}\left\{\frac{1}{n}\sum_{i=1}^nt_f(O_i)\right\}$ is a well-defined random variable.
    \label{lemma:countable}
\end{lemma}

\begin{proof} Fix $\underset{\sim}{o}:=(o_1,\dots, o_n)\in\mc{O}^n$ and define $g(\underset{\sim}{o};f):=\frac{1}{n}\sum_{i=1}^nt_f(o_i)$, $A_1(\underset{\sim}{o}):=\sup_{f\in \mc{F}} g(\underset{\sim}{o};f)$, and $A_2(\underset{\sim}{o}):=\sup_{f\in \mc{S}}g(\underset{\sim}{o};f)$. We will show that $A_1(\underset{\sim}{o})=A_2(\underset{\sim}{o})$.  

Clearly $A_2(\underset{\sim}{o})\le A_1(\underset{\sim}{o})$ because $\mc{S}\subset \mc{F}$. Suppose that $A_2(\underset{\sim}{o})<A_1(\underset{\sim}{o})$ and set $\varepsilon:=A_1(\underset{\sim}{o})-A_2(\underset{\sim}{o})>0$. By the definition of the supremum, there exists $f_\varepsilon\in\mc{F}$ such that $A_1(\underset{\sim}{o})-\frac{\varepsilon}{2}< g(\underset{\sim}{o};f_{\varepsilon})$.
Now, the map $f\mapsto g(\underset{\sim}{o};f)$ is continuous because the map $f\mapsto t_f(o_i)$ is continuous for each $i$. Then, there exists $\delta>0$ such that $d(f,f_\varepsilon)<\delta$ implies $|g(\underset{\sim}{o};f)-g(\underset{\sim}{o};f_{\varepsilon})|<\frac{\varepsilon}{2}$. Since $\mc{S}$ is dense in $\mc{F}$, there exits $f_s\in \mc{S}$ such that $d(f_s,f_\varepsilon)<\delta$. Then, $g(\underset{\sim}{o};f_s)>g(\underset{\sim}{o};f_{\varepsilon})-\frac{\varepsilon}{2}>A_1(\underset{\sim}{o})-\varepsilon=A_2(\underset{\sim}{o})$, which contradicts the definition of $A_2(\underset{\sim}{o})$. We thus conclude that $A_1(\underset{\sim}{o})=A_2(\underset{\sim}{o})$.

Next, since $g(\cdot;f):\mc{O}^n\mapsto\mathbb{R}$ is measurable for all $f\in \mc{S}$, and the map $f\mapsto g(\underset{\sim}{o};f)$ is continuous for all $\underset{\sim}{o}\in\mc{O}^n$, Lemma A.3.17 of \cite{steinwart2008support} ensures that the mapping $(\underset{\sim}{o},f)\mapsto g(\underset{\sim}{o};f)$ is measurable on $\mc{O}^n\times\mc{S}$. Then, the map $\underset{\sim}{o}\mapsto\sup_{f\in\mc{S}}g(\underset{\sim}{o};f)$ is also measurable because $\mc{S}$ is countable. Hence, $\sup_{f\in\mc{S}}g(O_1,\dots,O_n;f)= \sup_{f\in\mc{F}}g(O_1,\dots,O_n;f)$
is a well-defined random variable. This completes the proof.
\end{proof}

\begin{corollary}
Let $\mc{O}_0\subset\mc{O}$ be a measurable set. Suppose that:
\begin{enumerate}
    \item $\mc{F}$ is a separable metric space with metric denoted by $d$ such that $\mc{F}\subset\mc{L}^2(P_O)$, and let $\mc{S}$ denote a countable dense subset of it,
    \item For each $f\in\mc{L}^2(P_O)$, $t_f:\mc{O}\mapsto\mathbb{R}$ is a measurable function,  where $\mc{O}$ is the sample space of a random vector $O$,
    \item For each $o\in\mc{O}$, the map $f\mapsto t_{I_{\mc{O}_0}\cdot f}(o)$ is continuous with respect to $d$ at all $f \in \mc{F}$.
    \item The map $f\mapsto t_{I_{\mc{O}_0}\cdot f}$ is continuous from $(\mc{F},d)$ into $\mc{L}^1(P_O)$.
\end{enumerate}
For given fixed $\beta\in\mathbb{R}$ and any $o\in\mc{O}$, define the map $f\in\mc{F}\mapsto h_f(o):= t_{I_{\mc{O}_0}\cdot f}(o)-\beta\cdot\E\{t_{I_{\mc{O}_0}\cdot f}(O)\}$. Given $\mc{F}'\subset\mc{F}$, define $S':=\mc{F}'\cap\mc{S}$. Then,
\begin{enumerate}
    \item[i)] $\sup_{f\in \mc{F}'}\left\vert \frac{1}{n}\sum_{i=1}^nh_f(O_i) \right\vert
= \sup_{f \in \mc{S}'}\left\vert \frac{1}{n}\sum_{i=1}^nh_f(O_i) \right\vert$. In particular, the left-hand side is a well-defined random variable.
\item[ii)] When $\beta=0$, the conclusion in i) holds even if assumption 4 fails.
\end{enumerate}
 \label{corollary:countable_Lf}
\end{corollary}

\begin{proof} Fix $o\in\mc{O}$. First, we show that the map $f\mapsto h_f(o)$ is continuous with respect to the metric $d$. Take $f_1,f_2\in\mc{F}$. By triangle inequality, we have $ \vert h_{f_1}(o)-h_{f_2}(o)\vert\le \vert t_{I_{\mc{O}_0}\cdot f_1}(o)-t_{I_{\mc{O}_0}\cdot f_2}(o)\vert +|\beta|\cdot\E_P\{\vert  t_{I_{\mc{O}_0}\cdot f_1}(O)-t_{I_{\mc{O}_0}\cdot f_2}(O)\vert\}$.

Let $\varepsilon>0$. By assumption 3, there exists $\delta>0$ such that if $d(f_1,f_2)<\delta$, then $\left\vert t_{I_{\mc{O}_0}\cdot f_1}(o)-t_{I_{\mc{O}_0}\cdot f_2}(o)\right\vert <\varepsilon/2$. By assumption 4, there exists $\delta'>0$ such that  if $d(f_1,f_2)<\delta'$, then $|\beta|\cdot\E_P\{\vert  t_{I_{\mc{O}_0}\cdot f_1}(O)-t_{I_{\mc{O}_0}\cdot f_2}(O)\vert\}<\varepsilon/2$. Then, if $d(f_1,f_2)<\min\{\delta,\delta'\}$, we have $ \vert h_{f_1}(o)-h_{f_2}(o)\vert<\varepsilon$.

By Assumption 1 and definition of $\mc{F}'$, 
$\mc{F}'$ is separable with respect to $d$ with dense subset $\mc{S}'$. Then, it follows from Lemma~\ref{lemma:countable} that 
\[
\sup_{f\in \mc{F}'} \Big\{\frac{1}{n}\sum_{i=1}^nh_f(O_i) \Big\}
= \sup_{f \in \mc{S}'} \Big\{\frac{1}{n}\sum_{i=1}^nh_f(O_i) \Big\},\: \sup_{f\in \mc{F}'} \Big\{-\frac{1}{n}\sum_{i=1}^nh_f(O_i) \Big\}
= \sup_{f \in \mc{S}'} \Big\{-\frac{1}{n}\sum_{i=1}^nh_f(O_i) \Big\},
\]
and the left hand sides of both equalities are well-defined random variables. Therefore
\begin{align*}
    \sup_{f\in \mc{F}'} \Big\vert\frac{1}{n}\sum_{i=1}^nh_f&(O_i) \Big\vert = \max\Bigg\{\sup_{f\in \mc{F}'} \Big[\frac{1}{n}\sum_{i=1}^nh_f(O_i) \Big],\sup_{f\in \mc{F}'} \Big[-\frac{1}{n}\sum_{i=1}^nh_f(O_i) \Big]\Bigg\}
    \\
    &=\max\Bigg\{\sup_{f\in \mc{S}'} \Big[\frac{1}{n}\sum_{i=1}^nh_f(O_i) \Big],\sup_{f\in \mc{S}'} \Big[-\frac{1}{n}\sum_{i=1}^nh_f(O_i) \Big]\Bigg\}=\sup_{f\in \mc{S}'} \Big\vert\frac{1}{n}\sum_{i=1}^nh_f(O_i) \Big\vert
\end{align*}
and the leftmost hand side is a well-defined random variable. This concludes the proof.
\end{proof}

Lemmas \ref{lemma:lemma11foster}--\ref{lemma:lemma14foster} below are minor variations of results in \cite{wainwright2019high} and  \cite{foster2023orthogonal}. The following Lemma provides a weaker bound than Lemma 14.21 of \cite{wainwright2019high}. 

\begin{lemma}
Suppose that:
\begin{enumerate}
    \item The assumptions of Corollary \ref{corollary:countable_Lf} hold with $\mc{F}$ satisfying: (1.1) $\mc{F}$ is uniformly bounded by $b$, (1.2) for each $o\in\mc{O}$, the evaluation functional $f\mapsto f(o)$ is continuous relative to $d$, and (1.3) $f=0$ belongs to $\mc{F}$.
    \item $t_f$ satisfies $t_f=0$ when $f=0$.
    \item There exists a constant $L$ such that for any $f_1,f_2\in\mc{L}^2(P_O)$, $|t_{f_1}(o)-t_{f_2}(o)|\le L\cdot\left\vert f_1(o)-f_2(o)\right\vert$ for all $o\in\mc{O}$ ($o\in\mc{O}_0$).
\end{enumerate}
Define $h_f$ as in Corollary \ref{corollary:countable_Lf} with $\beta = 1$. For any $r>0$ define $\mc{F}_r:=\left\{f \in \mc{F}\::\:\Vert I_{\mc{O}_0}\cdot f\Vert_2 \le r \right\}$ and $Z_n(r):=\sup_{f\in \mc{F}_r} \big\vert\frac{1}{n}\sum_{i=1}^nh_f(O_i) \big\vert$. Then, $Z_n(r)$ is a well-defined random variable and for all $u>0$, $\text{P}\left\{Z_n(r)\ge 4L\mathcal{R}_n\left(I_{\mc{O}_0}\cdot\mc{F}, r\right)+u\right\}\le 2\exp\Big(-\frac{nu^2}{8eL^2r^2+128ebL^2\mc{R}_n\left(I_{\mc{O}_0}\cdot \mc{F},r\right)+8bLu}\Big)$,
    where $e=\exp(1)$ and $\mc{R}_n(I_{\mc{O}_0}\cdot\mc{F},r)$ is the localized Rademacher complexity of radius $r$ of the function class $I_{\mc{O}_0}\cdot\mc{F}:=\{I_{\mc{O}_0}\cdot f\:|\:f\in\mc{F}\}$.
    Moreover, if $\delta_n$ is any solution to the inequality $\mc{R}_n\left(star\left(I_{\mc{O}_0}\cdot\mc{F}\right),\delta\right)\le\frac{\delta^2}{b}$,
    then for each $r\ge\delta_n$, $ \text{P}\{Z_n(r)\ge\frac{4Lr\delta_n}{b}+u\}\le 2\exp\Big(-\frac{nu^2}{8eL^2r^2+128eL^2r\delta_n+8bLu}\Big)$.
    \label{lemma:lemma11foster}
\end{lemma}

\begin{proof} By assumption 1, the conclusion of Corollary \ref{corollary:countable_Lf} holds; i.e., $Z_n(r)$ is well-defined. Furthermore, $ Z_n(r)=\sup_{f\in \mc{S}_r} \big\vert\frac{1}{n}\sum_{i=1}^nh_f(O_i) \big\vert$.
where $\mc{S}_r := \mc{S} \cap \mc{F}_r$.
Now define the function class $\mc{C}:=\left\{\alpha h_f\::\:\alpha\in\{-1,1\}, f\in\mc{S}_r\right\}$ which is countable by construction. With $\mc{C}$ so defined we have that $ Z_n(r)=\sup_{h\in\mc{C}}\frac{1}{n}\sum_{i=1}^nh(O_i)$.

For any $h\in\mc{C}$, write $h=\alpha h_f$ for some $\alpha\in\{-1,1\}$ and $f\in\mc{S}_r$. Letting $f'=0$, we have
\begin{align*}
    |h(o)|&\le |t_{I_{\mc{O}_0}\cdot f}(o)|+\E\big\{\big\vert t_{I_{\mc{O}_0}\cdot f}(O)\big\vert\big\}\\
    &=|t_{I_{\mc{O}_0}\cdot f}(o)-t_{I_{\mc{O}_0}\cdot f'}(o)|+\E\big\{\big\vert t_{I_{\mc{O}_0}\cdot f}(O)-t_{I_{\mc{O}_0}\cdot f'}(O)\big\vert\big\}\\
    &\le L\cdot I_{\mc{O}_0}(o)\big\vert f(o)-f'(o)\big\vert + L\cdot \E\big\{I_{\mc{O}_0}(O)\big\vert f(O)-f'(O)\big\vert \big\}\le 2Lb,
\end{align*}
where the first inequality is by the triangle inequality, the first and second equalities are by assumptions 1.3 and 2, the second inequality is by assumption 3, and the third inequality is by assumption 1.1. We thus conclude that $\mc{C}$ is uniformly bounded by $2Lb$. Since $\mc{C}$ is countable, applying Talagrand inequality (Lemma \ref{lemma:talagrand}) we obtain that, for any $u>0$,
    \[
    \text{P}\left[Z_n(r)\ge \mathbb{E}\left\{Z_n(r)\right\}+u\right]\le 2\exp\left(-\frac{nu^2}{8e\sup_{h\in\mc{C}}\mathbb{E}[h(O)^2]+32ebL\mathbb{E}\left\{Z_n(r)\right\}+8bLu}\right).
    \]
    Again, take $h\in\mc{C}$ and write $h=\alpha h_f$ for some $\alpha\in\{-1,1\}$ and $f\in\mc{S}_r$. Thus $\mathbb{E}[h(O)^2]=\text{var}[t_{I_{\mc{O}_0}\cdot f}(O)]\le \mathbb{E}[\{ t_{I_{\mc{O}_0}\cdot f}(O)\}^2]\le L^2\E[\{I_{\mc{O}_0}(O)\cdot f(O)\}^2]\le L^2r^2$, where the second inequality follows by invoking assumption 3 applied to $f_1=I_{\mc{O}_0}\cdot f$ and $f_2=0$ and noting that by assumption 2, $t_{f_2}=0$, and the third inequality is because $f\in\mc{F}_r$.

   We now bound $\E\{Z_n(r)\}=\E\{\sup_{f\in\mc{S}_r}\vert\frac{1}{n}\sum_{i=1}^n[t_{I_{\mc{O}_0}\cdot f}(O_i)-\E\{t_{I_{\mc{O}_0}\cdot f}(O)\}]\vert\}$ using a symmetrization argument. Letting $\{O'_i\}_{i=1}^n$ be independent copies of $O$ and independent of $\{O_i\}_{i=1}^n$, we can write $\E\{Z_n(r)\}=\E\{\sup_{f\in\mc{S}_r}\vert\E[\frac{1}{n}\sum_{i=1}^n\{t_{I_{\mc{O}_0}\cdot f}(O_i)-t_{I_{\mc{O}_0}\cdot f}(O_i')\}]\vert\}$. By the inequality $\sup_{d\in\mc{D}}\E\left\{|d(O)|\right\}\le\E\left\{\sup_{d\in\mc{D}}|d(O)|\right\}$, which is valid for any countable class $\mc{D}$ of real-valued measurable functions, we conclude $\E\{Z_n(r)\}\le\E[\sup_{f\in\mc{S}_r}\vert\frac{1}{n}\sum_{i=1}^n\{t_{I_{\mc{O}_0}\cdot f}(O_i)-t_{I_{\mc{O}_0}\cdot f}(O_i')\}\vert]$.

    Now, let $\varepsilon_1,\dots,\varepsilon_n$ be $n$ i.i.d. Rademacher random variables, independent of $O_i$ and $O'_i$ for all $i=1,\dots,n$. Given this independence assumption, for any function $f\in\mc{F}$, the $n \times 1$ random vector with $i^{th}$ component $\varepsilon_i\{t_{I_{\mc{O}_0}\cdot f}(O_i)-t_{I_{\mc{O}_0}\cdot f}(O_i')\}$ has the same joint distribution as the $n \times 1$ random vector with $i^{th}$ component $ t_{I_{\mc{O}_0}\cdot f}(O_i)-t_{I_{\mc{O}_0}\cdot f}(O_i')$.
    Hence, $\E[\sup_{f\in\mc{S}_r}\vert\frac{1}{n}\sum_{i=1}^n\varepsilon_i\{t_{I_{\mc{O}_0}\cdot f}(O_i)-t_{I_{\mc{O}_0}\cdot f}(O_i')\}\vert]\le2\E\{\sup_{f\in\mc{S}_r}\vert\frac{1}{n}\sum_{i=1}^n\epsilon_i t_{I_{\mc{O}_0}\cdot f}(O_i)\vert\}\\=2\E\{\sup_{f\in\mc{S}'_r}\vert\frac{1}{n}\sum_{i=1}^n\epsilon_i  t_f(O_i)\vert\}$.
where $\mc{S}_r':=\left\{I_{\mc{O}_0}\cdot f:f\in\mc{S}_r\right\}$. Next, invoking the Ledoux-Talagrand contraction inequality (see inequality (5.61) of \cite{wainwright2019high}), we obtain $ 2\E\{\sup_{f\in\mc{S}'_r}\vert\frac{1}{n}\sum_{i=1}^n\epsilon_i  t_f(O_i)\vert\}\le
4L\E\{\sup_{f\in\mc{S}'_r}\vert\frac{1}{n}\sum_{i=1}^n\epsilon_i f(O_i)\vert\}\\=4L\E\{\sup_{f\in\mc{S}_r}\vert\frac{1}{n}\sum_{i=1}^n\epsilon_i I_{\mc{O}_0}(O_i)\cdot f(O_i)\vert\}=4L\E\{\sup_{f\in\mc{F}_r}\vert\frac{1}{n}\sum_{i=1}^n\epsilon_i I_{\mc{O}_0}(O_i)\cdot f(O_i)\vert\}=4L \mc{R}_n\left(I_{\mc{O}_0}\cdot\mc{F},r\right)$
    where the second equality follows by applying Corollary \ref{corollary:countable_Lf} to the functions $t_f(o,\varepsilon):=\varepsilon f(o)$ and $\beta=0$, noting that assumption 3 of the corollary holds because the map $f\mapsto f(o)$ is continuous (assumption 1.2). Thus, $\mathbb{E}\left\{Z_n(r)\right\}\le4L\cdot \mc{R}_n\left(I_{\mc{O}_0}\cdot\mc{F},r\right)$. Therefore, $ \text{P}[Z_n(r)\ge 4L\cdot \mc{R}_n(I_{\mc{O}_0}\cdot\mc{F},r)+u]\le \text{P}[Z_n(r)\ge \mathbb{E}\{Z_n(r)\}+u]\\ \le2\exp\Big(-\frac{nu^2}{8eL^2r^2+32ebL\mathbb{E}\left\{Z_n(r)\right\}+8bLu}\Big)\le 2\exp\Big(-\frac{nu^2}{8eL^2r^2+128ebL^2\mc{R}_n(I_{\mc{O}_0}\cdot\mc{F},r)+8bLu}\Big)$.
    
    Because $star\left(I_{\mc{O}_0}\cdot\mc{F}\right)$ is star-shaped, by the arguments in the proof of Lemma 13.6 of \cite{wainwright2019high}, the function $\delta\mapsto\frac{\mc{R}_n\left(star\left(I_{\mc{O}_0}\cdot\mc{F}\right),\delta\right)}{\delta}$ is non-decreasing on the interval $(0,\infty)$. Also, since $I_{\mc{O}_0}\cdot\mc{F}\subset star\left(I_{\mc{O}_0}\cdot\mc{F}\right)$, then, for any $r\ge\delta_n$, it holds $\mc{R}_n\left(I_{\mc{O}_0}\cdot\mc{F},r\right)/r\le\mc{R}_n\left(star\left(I_{\mc{O}_0}\cdot\mc{F}\right),r\right)/r\le\mc{R}_n\left(star\left(I_{\mc{O}_0}\cdot\mc{F}\right),\delta_n\right)/\delta_n\le\frac{\delta_n}{b}$. Then, $\text{P}\{Z_n(r)\ge \frac{4Lr\delta_n}{b}+u\}\le\text{P}\{Z_n(r)\ge 4L \mc{R}_n(I_{\mc{O}_0}\cdot\mc{F},r)+u\}\le 2\exp\Big(-\frac{nu^2}{8eL^2r^2+128eL^2r\delta_n+8bLu}\Big)$.
\end{proof}

\begin{lemma}[Lemma 12 of \cite{foster2023orthogonal}]
     Suppose the assumptions of Lemma \ref{lemma:lemma11foster} hold. Let $h_f$ and $Z_n(r) $ be defined as in Lemma \ref{lemma:lemma11foster} and let $\delta_n$ be any solution to $\mc{R}_n\left(star\left(I_{\mc{O}_0}\cdot\mc{F}\right),\delta\right)\le\frac{\delta^2}{b}$. For fixed $\theta>1$ define the event $\mc{E}:=\{\exists f\in\mc{F}:\Vert I_{\mc{O}_0}\cdot f\Vert_2\ge\delta_n\text{ and } \big\vert\frac{1}{n}\sum_{i=1}^nh_f(O_i) \big\vert\ge \frac{5\theta L\delta_n}{b}\left\Vert I_{\mc{O}_0}\cdot f\right\Vert_2\}$. Also, set $c_1=\left\{8b^2(1+17e)\right\}^{-1}$. Then, $\text{P}(\mathcal{E})\le2\{ 1+I\{\delta_n<b/\theta\}\lfloor\{\log(\theta)\}^{-1}\log(b/\delta_n)\rfloor\} \:\exp(-c_1n\delta_n^2)$ where $\lfloor x \rfloor$ denotes the greatest integer smaller than or equal to $x$.
     \label{lemma:lemma12foster}
\end{lemma}

\begin{proof}
   This proof uses a peeling argument as in the proof of Lemma 12 of \citet*{foster2023orthogonal}. Since $\Vert I_{\mc{O}_0}\cdot f\Vert_2\le \Vert f\Vert_{2}\le\Vert f\Vert_{\infty}$, Assumption 1.1 of Lemma \ref{lemma:lemma11foster} implies that $\Vert I_{\mc{O}_0}\cdot f\Vert_{2}\le b$ for any $f\in\mc{F}$. Then, any $f\in\mc{F}$ satisfying $\Vert I_{\mc{O}_0}\cdot f\Vert_2\ge \delta_n$ belongs to the set $\mc{S}_m:=\{f\in\mc{F}\text{ }\::\:\text{ }\theta^{m-1}\delta_n\le\Vert I_{\mc{O}_0}\cdot f\Vert_2\le \theta^m\delta_n\}$ for some $m\in\left\{1,\cdots,M\right\}$ where $M$ is the smallest integer satisfying $b\le \theta^M \delta_n$. We find such $M$ as follows. If $\delta_n\ge b/\theta$ then $M=1$ is such constant because $\theta>1$. If $\delta_n<b/\theta$ then $b\le \theta^M \delta_n$ iff $\log(b/\delta_n)\le M\log(\theta)$ iff $\left\{\log(\theta)\right\}^{-1}\log(b/\delta_n)\le M$. Therefore, we conclude that $M=1+I\{\delta_n<b/\theta\}\lfloor\{\log(\theta)\}^{-1}\log(b/\delta_n)\rfloor$. Note that $M$ is a positive integer because $\theta>1$ and $b/\delta_n>\theta$ when $\delta_n<b/\theta$.
    
    Now, if the event $\mc{E}\cap\mc{S}_m$ occurs for some $m\in\{1,\dots,M\}$, then there exists $f$ with $\theta^{m-1}\delta_n\le\Vert I_{\mc{O}_0}\cdot f\Vert_2\le r_m:=\theta^m\delta_n$ such that $\vert\frac{1}{n}\sum_{i=1}^nh_f(O_i) \vert\ge \frac{5\theta L\delta_n}{b}\left\Vert I_{\mc{O}_0}\cdot f\right\Vert_2
    \ge \frac{5\theta L\delta_n}{b} \frac{r_m}{\theta} =5 \frac{L\delta_n r_m}{b}$.
    Consequently, we have $\text{P}\left(\mc{E}\cap\mc{S}_m\right)\le\text{P}\left[Z_n(r_m)\ge5 \frac{L\delta_n r_m}{b}\right]$. Applying Lemma \ref{lemma:lemma11foster} with $r=r_m$ and $u=\frac{L\delta_nr_m}{b}$, and setting $c_1=\left\{8b^2(1+17e)\right\}^{-1}$, we obtain that $\text{Pr}\left(\mc{E}\cap\mc{S}_m\right)$ happens with probability at most $2\exp\Big\{-\frac{nL^2r_m^2\delta_n^2}{8b^2(eL^2r_m^2+16eL^2\delta_nr_m+L^2\delta_nr_m)}\Big\}
    \le 2\exp\left(-c_1n\delta_n^2\right)$ where the inequality holds since $\delta_n\le r_m$. Then $\text{P}(\mc{E})\le\sum_{m=1}^M\text{P}(\mc{E}\cap\mc{S}_m)\le 2M\exp(-c_1 n\delta_n^2)=2\{ 1+I\{\delta_n<b/\theta\}\lfloor\{\log(\theta)\}^{-1}\log(b/\delta_n)\rfloor\}\exp\left(-c_1n\delta_n^2\right)$.
\end{proof}

\begin{lemma}[Lemma 14 of \cite{foster2023orthogonal}]
Suppose the assumptions of Lemma \ref{lemma:lemma11foster} hold. Let $h_f$ and $Z_n(r) $ be defined as in Lemma \ref{lemma:lemma11foster} and let $\delta_n$ be any solution to the inequality $\mc{R}_n\left(star\left(I_{\mc{O}_0}\cdot\mc{F}\right),\delta\right)\le\frac{\delta^2}{b}$. Set $c_1=\left\{8b^2(1+17e)\right\}^{-1}$. Then, 
\begin{equation}
\big\vert\frac{1}{n}
\sum_{i=1}^n h_f(O_i) \big\vert \le 
\frac{6L\delta_n}{b} \left\{ \left\Vert I_{\mc{O}_0}\cdot f \right\Vert_2 + \delta_n \right\} \quad \forall f\in\mc{F}
\label{emp_process}
\end{equation}
happens with probability at least 
$1-\{4+2I\{\delta_n<5b/6\}\lfloor\{\log(6/5)\}^{-1}\log(b/\delta_n)\rfloor\}\exp(-c_1n\delta_n^2)$.
\label{lemma:lemma14foster}
\end{lemma}

\begin{proof}
    Consider the events $ \mc{E}_0=\{Z_n(\delta_n)\ge5L\delta_n^2/b\}$ and
    \[
     \mc{E}_1=\Big\{\exists f\in\mc{F}:\Vert I_{\mc{O}_0}\cdot f\Vert_2\ge\delta_n\text{ and } \Big\vert\frac{1}{n}
\sum_{i=1}^n h_f(O_i) \Big\vert\ge \frac{6L\delta_n}{b}\Vert I_{\mc{O}_0}\cdot f\Vert_2\Big\}.
    \]

If (\ref{emp_process}) is violated, then either $\mc{E}_0$ or $\mc{E}_1$ must occur. Applying Lemma \ref{lemma:lemma11foster} with $r=\delta_n$ and $u=\frac{L\delta_n^2}{b}$ implies that $\mc{E}_0$ happens with probability at most $2\exp(-c_1n\delta_n^2)$ where $c_1=\{8b^2(1+17e)\}^{-1}$. Applying Lemma \ref{lemma:lemma12foster} with $\theta = 6/5$ we conclude that $\mc{E}_1$ happens with probability at most $\{2+2I\{\delta_n<5b/6\}\lfloor\{\log(6/5)\}^{-1}\log(b/\delta_n)\rfloor\}\:\exp(-c_1n\delta_n^2)$. Thus, (\ref{emp_process}) is violated with probability at most $\{4+2I\{\delta_n<5b/6\}\lfloor\{\log(6/5)\}^{-1}\log(b/\delta_n)\rfloor\}\:\exp(-c_1n\delta_n^2)$. This concludes the proof.
\end{proof}

To state the next lemma, let $B$ and $B'$ be the constants defined in Assumption \ref{ass:closeness}, and $B_1$, $b$, and $ c_1$ as in Theorem~\ref{theo:conv}. Let $k_1:=\{\sup_{w\in\mc{W}}K_{\mc{H}}(w,w)\}^{1/2}$, $k_2:=\{\sup_{z\in\mc{Z}}K_{\mc{G}}(z,z)\}^{1/2}$, $c_0:=9\max\{\Vert s\Vert_\infty,\Vert r\Vert_\infty\cdot k_1B, 2c^2b\}/(c^2 b)$, and $c_0':=c^2\cdot c_0(2+3c_0+2\max\{1,c_0/9\})$. The following lemma extends Lemma 10 of \cite{bennett2023source} to allow for $\lambda_\mc{G}\neq0$.

\begin{lemma} Under the assumptions of Theorem~\ref{theo:conv} and with the constants defined as in that Theorem, the following event happens for any $h\in \mc{H}_{B}$:
\begin{align*}
    \frac{1}{4c^2}\Big\{\big\Vert I_0 \mc{T} \big(h_0^\dag-&\tilde h\big)\big\Vert^2_2-\big\Vert I_0\mc{T} \big(h_0^\dag-h\big)\big\Vert^2_2\Big\}+\lambda_{\mc{H}}\Big(\big\Vert \tilde h\big\Vert_{\mc{H}}^2-\big\Vert h\big\Vert_{\mc{H}}^2\Big)\\
     &\quad\quad\quad\le \frac{1}{2c^2}\big\Vert I_0\mc{T} \big(h_0^\dag-h\big)\big\Vert^2_2+c_0\delta_n\big\Vert  I_0\mc{T}\big(\tilde h-h\big)\big\Vert_2+\lambda_{\mc{G}}B_1^2+c_0'\delta^2_n.
\end{align*}
\label{lemma:erm}
\end{lemma}
\begin{proof}
For any $h\in\mc{H}_{B}$, define $g_h:=\frac{1}{2c^2}\mc{T}(h_0^\dag-h)$. Since $\Vert h_0^\dag-h\Vert_{\mc{H}}\le\Vert h_0^\dag\Vert_{\mc{H}}+\Vert h\Vert_{\mc{H}}\le 2B$, Assumption~\ref{ass:closeness} of the main text implies that $g_h\in\mc{G}_{B_1}$ where $B_1:=B'/c^2$. In addition, for any $h\in\mc{H}_{B}$, we have $\sup_{w\in\mc{W}}|h(w)|\le k_1 B $, so the function class $\mc{H}_B$ is uniformly bounded by $k_1 B$. Similarly, the function class $\mc{G}_{B_1}$ is uniformly bounded by $k_2B_1$.
 
 By part 1 of Proposition~\ref{prop:innermax} and Assumption~\ref{ass:closeness}, we have
 \[
  \frac{1}{4c^2}\big\Vert I_0 \mc{T} \big(h_0^\dag-\tilde h\big)\big\Vert^2_2= \E_{P_0}\big\{I_0(X)s(O)g_{\tilde h}(Z)-I_0(X)r(O)\tilde h(W)g_{\tilde h}(Z)-c^2\cdot I_0(X)g_{\tilde h}(Z)^2\big\},
 \]
 where the expectation in the right hand-side is with respect to $O$, and $\tilde h$ is regarded as a fixed, non-random function.

Next, we apply Lemma~\ref{lemma:lemma14foster} to bound the expectation of each of the three terms in the right hand side of the last display. Specifically, with $c_1:=\left\{8b^2(1+17e)\right\}^{-1}$, $c_2:=6L'/b$, $b:=k_2B_1$, $\xi(\delta_n):=\{4+2I\{\delta_n<5b/6\}\lfloor\{\log(6/5)\}^{-1}\log(b/\delta_n)\rfloor\}\:\exp(-c_1n\delta_n^2)$, and
\begin{equation}
    L'=\max\{\Vert s\Vert_\infty,\:\Vert r\Vert_\infty\cdot k_1B,\: 2c^2\cdot b\},\label{def:Lmax}
\end{equation}
applying Lemma~\ref{lemma:lemma14foster}, we can conclude that
\begin{enumerate}
    \item[I.] With probability at least $1-\xi(\delta_n)$, the following event happens
\begin{equation}
    \Big\vert \E_n\big[I_o(X)s(O)g(Z) - \E_{P_0}\left\{I_0(X)s(O)g(Z)\right\} \big]\Big \vert \le c_2\delta_n\left\{\left\Vert I_0g\right\Vert_2+\delta_n\right\} \quad \forall {g \in \mc{G}_{B_1}}. \label{eq:term1sup}
\end{equation}
In particular, $\E_{P_0}\{I_0(X)s(O)g_{\tilde h}(Z)\}\le \E_n\{I_0(X)s(O)g_{\tilde h}(Z)\}+ c_2\delta_n\{\Vert I_0g_{\tilde h}\Vert_2+\delta_n\}$ holds with probability at least $1-\xi(\delta_n)$.

\item[II.] With probability at least $1-\xi(\delta_n)$, for all $g \in \mc{G}_{B_1}$, the following event happens
\begin{equation}
    \Big\vert \E_n\big[I_0(X)r(O)\tilde h(W)g(Z) - \E_{P_0}\{I_0(X)r(O)\tilde h(W)g(Z)\} \big]\Big \vert \le c_2\delta_n\{\left\Vert I_0g\right\Vert_2+\delta_n\}. \label{eq:term2sup}
\end{equation}
Notably, $-\E_{P_0}\{I_0(X)r(O)\tilde h(W)g_{\tilde h}(Z)\}\le -\E_n\{I_0(X)r(O)\tilde h(W)g_{\tilde h}(Z)\}+ c_2\delta_n\{\Vert I_0g_{\tilde h}\Vert_2+\delta_n\}$ holds with probability at least $1-\xi(\delta_n)$.
\item[III.]  With probability at least $1-\xi(\delta_n)$, the following event happens
\begin{equation}
    \Big\vert \E_n\big[c^2 I_0(X)g(Z)^2 - \E_{P_0}\{c^2 I_0(X)g(Z)^2\} \big]\Big \vert \le c_2\delta_n\{\Vert I_0g\Vert_2+\delta_n\} \quad \forall {g \in \mc{G}_{B_1}}. \label{eq:term3sup}
\end{equation}
In particular, $ -\E_{P_0}\{c^2 I_0(X)g_{\tilde h}(Z)^2\}\le -\E_n\{c^2 I_0(X)g_{\tilde h}(Z)^2\}+ c_2\delta_n\{\Vert I_0g_{\tilde h}\Vert_2+\delta_n\}$ holds with probability at least $1-\xi(\delta_n)$.
\end{enumerate}

To apply Lemma~\ref{lemma:lemma14foster} to prove (I), (II), and (III), we define convenient function classes $\mc{F}$, maps $f \in \mc{L}^2(P_{0,O})\mapsto t_f$, and sets $\mc{O}_0$. For these objects, we argue that the following conditions hold:

 \begin{enumerate}
    \item[a.] $\mc{F}$ is a separable metric space with metric denoted by $d$ such that $\mc{F}\subset\mc{L}^2(P_{0,O})$ and: (a.1) $\mc{F}$ is uniformly bounded by $b$, (a.2) for each $o\in\mc{O}$, the evaluation functional $f\mapsto f(o)$ is continuous relative to the metric $d$, and (a.3) $f=0$ belongs to $\mc{F}$.
    \item[b.] For each $f\in\mc{L}^2(P_{0,O})$, $t_f:\mc{O}\mapsto\mathbb{R}$ is a measurable function.
    \item[c.] For each $o\in\mc{O}$, the map $f\mapsto t_{I_{\mc{O}_0}\cdot f}(o)$ is continuous with respect to $d$ at all $f \in \mc{F}$.
    \item[d.] The map $f\mapsto t_{I_{\mc{O}_0}\cdot f}$ is continuous from $(\mc{F},d)$ to $\mc{L}^1(P_{0,O})$. 
    \item[e.] $t_f$ satisfies $t_f=0$ when $f=0$.
    \item[f.] $\exists L$ such that for any $f_1,f_2\in\mc{L}^2(P_{0,O})$, $|t_{f_1}(o)-t_{f_2}(o)|\le L\left\vert f_1(o)-f_2(o)\right\vert$ $\forall\:o\in\mc{O}$.
\end{enumerate}

\textit{(i) Application of Lemma \ref{lemma:lemma14foster} to show (I)}. Define $\mc{F}:=\mc{G}_{B_1}$, $t_f(o):=s(o)\cdot f(z)$, and $\mc{O}_0:=\{o\in\mc{O}\: :\:\text{the subvector }x\text{ of }o\text{ is in }\mc{X}_0\}$. We endow $\mc{F}$ with the metric $d$ associated with the norm $\Vert\cdot\Vert_{\mc{G}}$ of the RKHS $\mc{G}$. To show that condition (a) holds first note that $\mc{F}= \mc{G}_{B_1}\subset\mc{L}^2(P_{0,Z})\subset\mc{L}^2(P_{0,O})$, and $\mc{G}_{B_1}$ is separable because $\mc{G}$ is a separable RKHS by Assumption~\ref{ass:K_G} (see Proposition 11.7 of \cite{paulsen2016introduction}). We already argued that (a.1) holds with $b=k_2B_1$, (a.2) holds because the evaluation functionals are continuous in RKHSs, and (a.3) trivially holds.

Condition (b) is trivially satisfied. Condition (c) holds because $ |t_{I_{\mc{O}_0}\cdot f_1}(o)-t_{I_{\mc{O}_0}\cdot f_2}(o)|=|I_{\mc{O}_0}(o)s(o)\left\{f_1(z)-f_2(z)\right\}|\le \Vert s\Vert_\infty |f_1(z)-f_2(z)|\le  \Vert s\Vert_\infty k_2\Vert f_1-f_2\Vert_{\mc{G}}$, where the last inequality holds because $\Vert f\Vert_\infty\le k_2\Vert f\Vert_\mc{F}$ $\forall f\in\mc{F}$. Condition (d) holds because $\E_{P_0}\{|t_{I_{\mc{O}_0}\cdot f_1}(O)-t_{I_{\mc{O}_0}\cdot f_2}(O)|\}\le\sup_{o\in\mc{O}}|t_{I_{\mc{O}_0}\cdot f_1}(o)-t_{I_{\mc{O}_0}\cdot f_2}(o)|\le  \Vert s\Vert_\infty\cdot k_2\cdot \Vert f_1-f_2\Vert_{\mc{G}}$. Condition (e) trivially holds. Arguing as in the proof of (c), condition (f) holds with $L=\Vert s\Vert_{\infty}$.

Before applying Lemma~\ref{lemma:lemma14foster}, note that $I_0\cdot\mc{G}_{B_1}$ is star-shaped because $\mc{G}_{B_1}$ is star-shaped and multiplication by $I_0$ preserves this. Hence, its critical radius equals that of its star-shaped hull, so the lemma applies with an upper bound on the critical radius of $I_0\cdot\mc{G}_{B_1}$.

 Applying Lemma~\ref{lemma:lemma14foster}, we obtain that the display~\eqref{eq:term1sup} happens with probability at least $1-\xi(\delta_n)$ with $c_2=6\cdot \Vert s\Vert_{\infty}/b$ and $b=k_2B_1$. Consequently, the display~\eqref{eq:term1sup} also happens with probability at least $1-\xi(\delta_n)$ with $c_2$ now equal to $6L'/b$ where $L'$ is defined in~\eqref{def:Lmax}.
\smallskip

\textit{(ii) Application of Lemma~\ref{lemma:lemma14foster} to show (II)}. Define $\mc{F}$, the metric $d$, and $\mc{O}_0$ as in (i). Define $t_f(o):=r(o)\cdot \tilde h(w)\cdot f(z)$. Condition (a) holds as shown in (i). Condition (b) holds trivially. Conditions (c) and (d) hold as shown in (i), using in addition that $\sup_{w\in\mc{W}}|\tilde h(w)|\le k_1B$. Condition (e) is immediate. Arguing as in the proof of condition (c) in (i), (f) holds with $L=\Vert r\Vert_{\infty}\cdot k_1B$. Applying Lemma \ref{lemma:lemma14foster}, we obtain that the display~\eqref{eq:term2sup} happens with probability at least $1-\xi(\delta_n)$ with $c_2=6 \Vert r\Vert_{\infty} k_1B/b$ and $b=k_2B_2$. Then, \eqref{eq:term2sup} also happens with probability at least $1-\xi(\delta_n)$ with $c_2$ now equal to $6L'/b$ where $L'$ is defined in~\eqref{def:Lmax}.
\smallskip

\textit{(iii) Application of Lemma \ref{lemma:lemma14foster} to show (III)}. Define $\mc{F}$, $d$, and $\mc{O}_0$ as in (i). Define $t_f(o):=c^2\cdot f(z)^2$. Condition (a) holds as shown in (i), (b) holds trivially, (c) holds because $  |t_{I_{\mc{O}_0}\cdot f_1}(o)-t_{I_{\mc{O}_0}\cdot f_2}(o)|=|c^2I_{\mc{O}_0}(o)\left\{f_1(z)^2-f_2(z)^2\right\}|
    \le c^2 \left\vert f_1(z)+f_2(z)\right\vert\cdot\left\vert f_1(z)-f_2(z)\right\vert
    \le 2c^2 k_2 B_1\left\vert f_1(z)-f_2(z)\right\vert\le  2c^2 k_2^2 B_1\Vert f_1-f_2\Vert_{\mc{G}}$, and (d) holds because $\E_{P_0}\{|t_{I_{\mc{O}_0}\cdot f_1}(O)-t_{I_{\mc{O}_0}\cdot f_2}(O)|\}\le\sup_{o\in\mc{O}}|t_{I_{\mc{O}_0}\cdot f_1}(o)-t_{I_{\mc{O}_0}\cdot f_2}(o)|\le  2c^2\cdot k_2^2\cdot B_1\cdot \Vert f_1-f_2\Vert_{\mc{G}}$. Condition (e) is trivial and condition (f) holds with $L=2c^2 b$ arguing as in the proof of (c).

 Applying Lemma \ref{lemma:lemma14foster}, we obtain that the display~\eqref{eq:term3sup} happens with probability at least $1-\xi(\delta_n)$ with $c_2=6\cdot 2c^2\cdot b/b$ and $b=k_2B_1$. Consequently, the display~\eqref{eq:term3sup} also happens with probability at least $1-\xi(\delta_n)$ with $c_2$ now equal to $6L'/b$ where $L'$ is defined in~\eqref{def:Lmax}.

Setting $c_3=3c_2=18L'/b$ with $L'$ defined in~\eqref{def:Lmax}, with probability at least $1-3\xi(\delta_n)$, we have $\frac{1}{4c^2}\Vert I_0\mc{T}(h_0^\dag-\tilde h)\Vert^2_2\le\E_n\{I_0(X)s(O)g_{\tilde h}(Z)-I_0(X)r(O)\tilde h(W)g_{\tilde h}(Z)-c^2 I_0(X)g_{\tilde h}(Z)^2\}+c_3\delta_n\Vert I_0 g_{\tilde h}\Vert_2 +c_3\delta_n^2$.

Then, the following event happens with probability at least $1-3\xi(\delta_n)$, for any $h\in\mc{H}_B$:
\begin{align}
      \frac{1}{4c^2}\big\Vert I_0\mc{T}\big(&h_0^\dag-\tilde h\big)\big\Vert^2_2\le\E_n\big\{I_0(X)s(O)g_{\tilde h}(Z)-I_0(X)r(O)\tilde h(W)g_{\tilde h}(Z)-c^2 I_0(X)g_{\tilde h}(Z)^2\big\}\nonumber\\
    &\quad-\lambda_{\mc{G}}\Vert g_{\tilde h}\Vert^2_{\mc{G}}+\lambda_{\mc{G}}\Vert g_{\tilde h}\Vert^2_{\mc{G}}+c_3\big(\delta_n\Vert I_0\cdot g_{\tilde h}\Vert_2 +\delta_n^2\big)\nonumber\\
    &\le \sup_{g\in\mc{G}}\big[\E_n\big\{I_0(X)s(O)g(Z)-I_0(X)r(O)\tilde h(W)g(Z)-c^2 I_0(X)g(Z)^2\big\}-\lambda_{\mc{G}}\Vert g\Vert_{\mc{G}}^2\big]\nonumber\\
    &\quad+\lambda_{\mc{G}}B_1^2+c_3\big(\delta_n\Vert I_0\cdot g_{\tilde h}\Vert_2 +\delta_n^2\big)\nonumber\\
    &\le\sup_{g\in\mc{G}}\big[\E_n\big\{I_0(X)s(O)g(Z)-I_0(X)r(O)h(W)g(Z)-c^2 I_0(X)g(Z)^2\big\}-\lambda_{\mathcal{G}}\Vert g\Vert_{\mathcal{G}}^2\big]\nonumber\\
&\quad+\lambda_{\mc{G}}B_1^2+\lambda_{\mc{H}}\Big(\big\Vert h\big\Vert_{\mc{H}}^2-\big\Vert \tilde h\big\Vert_{\mc{H}}^2\Big)+c_3\big(\delta_n\Vert I_0\cdot g_{\tilde h}\Vert_2 +\delta_n^2\big).\label{last_line_emr_step}
\end{align}

Note that the second inequality is always true because $g_{\tilde h}\in \mc{G}_{B_1}$, and the third inequality is also true because $\tilde h$ is the empirical risk minimizer defined in~\eqref{def:h_tilde}.

Set $c_4:=\max\{c_3[1+c_3/(2c^2)], 5c_2^2/c^2\}$. For any $h\in\mc{H}_B$, we define
\begin{align*}
    T_n(h)&:=\sup_{g\in\mc{G}}\E_{P_0}\big\{I_0(X)s(O)g(Z)-I_0(X)r(O)h(W) g(Z)-\frac{c^2}{2}\cdot I_0(X) g(Z)^2\big\}+c_4\delta_n^2,\\
    A_n(h)&:=\sup_{g\in\mc{G}_{B_1}}\left[\E_n\big\{I_0(X)s(O)g(Z)-I_0(X)r(O)h(W)g(Z)-c^2 I_0(X)g(Z)^2\big\}-\lambda_{\mathcal{G}}\Vert g\Vert_{\mathcal{G}}^2\right]\\
    B_n(h)&:=\sup_{g\in\mc{G}\setminus\mc{G}_{B_1}}\left[\E_n\big\{I_0(X)s(O)g(Z)-I_0(X)r(O)h(W)g(Z)-c^2 I_0(X)g(Z)^2\big\}-\lambda_{\mathcal{G}}\Vert g\Vert_{\mathcal{G}}^2\right],\\
    \text{and}\quad&\\
    C_n(h)&:= \sup_{g\in\mc{G}}\left[\E_n\big\{I_0(X)s(O)g(Z)-I_0(X)r(O)h(W)g(Z)-c^2\cdot I_0(X)g(Z)^2\big\}-\lambda_{\mathcal{G}}\Vert g\Vert_{\mathcal{G}}^2\right].
\end{align*}

Below we will show that, with probability at least $1-3\xi(\delta_n)$, each of the following event happens: $A_n(h)\le T_n(h)$ and $B_n(h)\le T_n(h)$. This will then imply that with probability at least $1-6\xi(\delta_n)$ for any $h\in\mc{H}_B$: 
\begin{align}
   C_n(h)\le T_n(h).\label{eq:emr_step_both}
\end{align}

Assume for now that the last assertion holds. By part 1 of Proposition~\ref{prop:innermax} and Assumption~\ref{ass:closeness}, we have that $\sup_{g\in\mc{G}}\E_{P_0}\{I_0(X)s(O)g(Z)-I_0(X)r(O)h(W)g(Z)-\frac{c^2}{2} I_0(X)g(Z)^2\}=\frac{1}{2c^2}\Vert I_0\mc{T}(h_0^\dag-h)\Vert^2_2$. Combining this equality with the inequalities in~\eqref{last_line_emr_step} and~\eqref{eq:emr_step_both}, we then conclude that with probability at least $1-9\xi(\delta_n)$, for any $h\in\mc{H}_B$:
\begin{align}
    \frac{1}{4c^2}\big\Vert I_0\mc{T}\big(h_0^\dag-\tilde h \big)\big\Vert^2_2\le&  \frac{1}{2c^2}\big\Vert I_0\mc{T}\big(h_0^\dag-h\big)\big\Vert^2_2+\lambda_{\mc{H}}\Big(\big\Vert h\big\Vert_{\mc{H}}^2-\big\Vert \tilde h\big\Vert_{\mc{H}}^2\Big)+\lambda_{\mc{G}}B_1^2\nonumber \\
    &\quad+c_3\cdot\delta_n\Vert  I_0\cdot g_{\tilde h}\Vert_2 +\left(c_3+c_4\right)\delta_n^2\label{ineq_aux}.
\end{align}

Furthermore, by the triangle inequality $\Vert I_0\mc{T}(h_0^\dag-\tilde h)\Vert_2\le \Vert I_0\mc{T}(h_0^\dag-h)\Vert_2+\Vert I_0\mc{T}(\tilde h-h)\Vert_2$ and since $\Vert I_0\cdot g_{\tilde h}\Vert_2=\frac{1}{2c^2}\Vert I_0\mc{T}(h_0^\dag-\tilde h)\Vert_2$, the inequality~\eqref{ineq_aux} implies that the following event happens with probability at least $1-9\xi(\delta_n)$ for any $h\in\mc{H}_B$
\begin{align*}
     \frac{1}{4c^2}\big\Vert I_0\mc{T}\big(h_0^\dag-\tilde h\big)\big\Vert^2_2\le&\frac{1}{2c^2}\big\Vert I_0\mc{T}\big(h_0^\dag-h\big)\big\Vert^2_2+\lambda_{\mc{H}}\Big(\big\Vert h\big\Vert_{\mc{H}}^2-\big\Vert \tilde h\big\Vert_{\mc{H}}^2\Big)+\lambda_{\mc{G}}B_1^2\\
    &+\frac{c_3}{2c^2}\Big\{\delta_n\big\Vert I_0\mc{T}\big(h_0^\dag-h\big)\big\Vert_2+\delta_n\big\Vert I_0\mc{T}\big(\tilde h-h\big)\big\Vert_2\Big\}+(c_3+c_4)\delta_n^2.
\end{align*}
By the AM-GM inequality $\frac{c_3}{2c^2}\delta_n \Vert I_0\mc{T}(h_0^\dag-h)\Vert_2\le\frac{1}{4c^2}\Vert I_0\mc{T}(h_0^\dag-h)\Vert_2^2+\frac{c_3^2}{4c^2}\delta_n^2$. Hence, the following event happens with probability at least $1-9\xi(\delta_n)$ for all $h\in\mc{H}_B$:
\begin{align*}
    \frac{1}{4c^2}&\Big\{\big\Vert I_0\mc{T}\big(h_0^\dag-\tilde h\big)\big\Vert^2_2-\big\Vert I_0\mc{T}\big(h_0^\dag-h\big)\big\Vert^2_2\Big\}+\lambda_{\mc{H}}\Big(\big\Vert \tilde h\big\Vert_{\mc{H}}^2-\big\Vert h\big\Vert_{\mc{H}}^2\Big)\\
    &\quad\quad\quad\le \frac{1}{2c^2}\big\Vert I_0\mc{T}\big(h_0^\dag-h\big)\big\Vert^2_2+\frac{c_3}{2c^2}\delta_n\big\Vert I_0\mc{T}\big(\tilde h-h\big)\big\Vert_2+\lambda_{\mc{G}}B_1^2+\Big(c_3+c_4+\frac{c_3^2}{4c^2}\Big)\delta^2_n,
\end{align*}
thus, arriving at the desired inequality.

We now turn to prove that $A_n(h)\le T_n(h)$ holds with probability at least $1-3\xi(\delta_n)$. Fix $h\in\mc{H}_B$ and take $g\in\mc{G}_{B_1}$. Applying Lemma~\ref{lemma:lemma14foster} as we did to show (I), (II), and (III), and noting that $-\lambda_{\mathcal{G}}\Vert g\Vert_{\mathcal{G}}^2\le0$, we have that the following event happens with probability at least $1-3\xi(\delta_n)$, 
\begin{align}
&\E_n\big\{I_0(X)s(O)g(Z)-I_0(X)r(O)h(W)g(Z)-c^2I_0(X)g(Z)^2\big\}-\lambda_{\mathcal{G}}\Vert g\Vert_{\mathcal{G}}^2\label{eq:bound_case1}\\
    &\le E_{P_0}\big\{I_0(X)s(O)g(Z)-I_0(X)r(O)h(W)g(Z)-c^2I_0(X) g(Z)^2\big\}+c_3\delta_n\Vert I_0\cdot  g\Vert_2+c_3\delta_n^2.\nonumber
\end{align}

By the AM-GM inequality, $c_3\delta_n\Vert I_0\cdot g\Vert_2\le \frac{c^2}{2}\Vert I_0\cdot g\Vert_2^2+\frac{c_3^2}{2c^2}\delta_n^2$. Hence, the right-hand side of (\ref{eq:bound_case1}) is upper bounded by
\begin{align*}
    \sup_{g\in\mc{G}}\E_{P_0}\Big\{I_0(X)s(O)g(Z)-I_0(X)r(O)h(W)g(Z)-\frac{c^2}{2} I_0(X)g(Z)^2\Big\}+\left(c_3+\frac{c^2_3}{2c^2}\right)\delta_n^2\le T_n(h).
\end{align*}
Since $g\in\mc{G}_{B_1}$ was arbitrary, $A_n(h)\le T_n(h)$ happens with probability at least $1-3\xi(\delta_n)$.

Turn now to prove that $B_n(h)\le T_n(h)$ holds with probability at least $1-3\xi(\delta_n)$. Fix $h\in\mc{H}_B$ and take $g\in\mc{G}\setminus\mc{G}_{B_1}$. Then, $\Vert g\Vert_{\mc{G}}>B_1$. Define $\tilde g:=B_1\cdot  g/\Vert g\Vert_{\mc{G}}$. Since $\mc{G}$ is star-shaped and $B_1/\Vert g\Vert_{\mc{G}}\in(0,1)$, $\tilde g\in\mc{G}$. Furthermore, $\Vert\tilde g\Vert_{\mc{G}}=B_1$, and so $\tilde g\in\mc{G}_{B_1}$. Applying Lemma~\ref{lemma:lemma14foster} as we did to prove (I) and (II), we obtain that, with probability at least $1-2\xi(\delta_n)$,  the following happens:
\begin{align*}
 \E_n\big\{I_0(X)&s(O)g(Z)-I_0(X)r(O)h(W)g(Z)\big\}\\
    &=\frac{\Vert g\Vert_\mc{G}}{B_1}\cdot \E_n\big\{I_0(X)s(O)\tilde g(Z)-I_0(X)r(O)h(W)\tilde g(Z)\big\}\\
    &\le \frac{\Vert g\Vert_\mc{G}}{B_1}\cdot \big[\E_{P_0}\big\{I_0(X)s(O)\tilde g(Z)-I_0(X)r(O)h(W)\tilde g(Z)\big\}+2c_2\delta_n\Vert I_0\cdot \tilde g\Vert_2+2c_2\delta_n^2\big]\\
    &=\E_{P_0}\big\{I_0(X)s(O)g(Z)-I_0(X)r(O)h(W)g(Z)\big\}+2c_2\delta_n\Vert I_0\cdot g\Vert_2+\frac{2c_2}{B_1}\Vert g\Vert_\mc{G}\cdot\delta_n^2,
\end{align*}
where $c_2=6L'/b$, with $L'$ defined in~\eqref{def:Lmax}. Applying Lemma~\ref{lemma:lemma14foster} as we did to prove (III), with probability at least $1-\xi(\delta_n)$, the following event happens:
\begin{align*}
    - \E_n\big\{c^2 I_0(X) g(Z)^2\big\}&=-\frac{\Vert g\Vert_\mc{G}^2}{B_1^2} \E_n\big\{c^2 I_0(X)\tilde g(Z)^2\big\}\\
    &\le \frac{\Vert g\Vert_\mc{G}^2}{B_1^2} \left[-\E_{P_0}\big\{c^2I_0(X)\tilde g(Z)^2\big\}+12c^2\delta_n\left(\Vert I_0\cdot \tilde g\Vert_2+\delta_n\right)\right]\\
    &=-\E_{P_0}\big\{c^2 I_0(X) g(Z)^2\big\}+\frac{12c^2}{B_1}\delta_n\Vert g\Vert_\mc{G}\cdot \Vert I_0\cdot  g\Vert_2+\frac{12c^2}{ B_1^2}\Vert g\Vert_\mc{G}^2\cdot\delta_n^2.
\end{align*}

Hence, with probability at least $1-3\xi(\delta_n)$, the following event happens:
\begin{align}
\E_n&\big\{I_0(X)s(O)g(Z)-I_0(X)r(O)h(W)g(Z)-c^2 I_0(X)g(Z)^2\big\}-\lambda_{\mathcal{G}}\Vert g\Vert_{\mathcal{G}}^2\nonumber\\
    &\le \E_P\big\{I_0(X)s(O)g(Z)-I_0(X)r(O)h(W)g(Z)-c^2 I_0(X) g(Z)^2\big\}-\lambda_{\mc{G}}\Vert g\Vert_{\mc{G}}^2\label{eq:bound_case2}\\
    &\quad+2c_2\delta_n\Vert I_0\cdot g\Vert_2+\frac{2c_2}{B_1}\Vert g\Vert_\mc{G}\cdot\delta_n^2+\frac{12c^2}{B_1}\delta_n\Vert g\Vert_\mc{G}\cdot \Vert I_0\cdot  g\Vert_2+\frac{12c^2}{ B_1^2}\Vert g\Vert_\mc{G}^2\cdot\delta_n^2.\nonumber
\end{align}

Applying the AM-GM inequality, we obtain (1) $ 2c_2\delta_n\Vert I_0\cdot g\Vert_2\le \frac{c^2}{4}\Vert I_0\cdot g\Vert_2^2+\frac{4c_2^2}{c^2}\delta_n^2$, (2) $\frac{2c_2}{B_1}\Vert g\Vert_\mc{G}\cdot\delta_n^2\le \frac{c_2^2}{c^2}\delta_n^2+\frac{c^2}{B_1^2}\Vert  g\Vert_\mc{G}^2\cdot\delta_n^2$, and (3) $\frac{12c^2}{B_1}\delta_n\Vert g\Vert_\mc{G}\cdot \Vert I_0\cdot  g\Vert_2\le \frac{c^2}{4}\Vert I_0\cdot g\Vert_2^2+\frac{144c^2}{B_1^2}\Vert g\Vert_\mc{G}^2\cdot\delta_n^2.$
Using these inequalities and the fact that $\lambda_\mc{G}\ge 157c ^2\cdot\delta_n^2/B_1^2$, the right-hand side of~\eqref{eq:bound_case2} is upper bounded by $ E_{P_0}\{I_0(X)s(O)g(Z)-I_0(X)r(O)h(W)g(Z)-\frac{c^2}{2} I_0(X) g(Z)^2\}+\frac{5c_2^2}{c^2}\delta_n^2\le\sup_{g\in\mc{G}} E_{P_0}\big\{m(O; g)-I_0(X)r(O)h(W)g(Z)-\frac{c^2}{2} I_0(X) g(Z)^2\big\}+\frac{5c_2^2}{c^2}\delta_n^2\le T_n(h)$. Thus, $B_n(h)\le T_n(h)$ holds with probability at least $1-3\xi(\delta_n)$. This completes the proof.
\end{proof}

We now state a more general version of Lemma~\ref{lemma:erm} that applies for an arbitrary measurable map $m(o;g)$, linear in $g$, that satisfies Assumption~\ref{ass:bounded}.

\begin{lemma} Under the assumptions of Theorem~\ref{theo:main} and with $\lambda_\mc{G}\ge 157c^2/B_1^2$, with probability at least $1-\{36+18\cdot I\{\delta_n<5b/6\}\lfloor\{\log(6/5)\}^{-1} \log(b/\delta_n)\rfloor\}\exp(-c_1n\delta_n^2)$, the following event happens for any $h\in \mc{H}_{B}$:
\begin{align*}
    \frac{1}{4c^2}\Big\{\big\Vert I_0 \mc{T} \big(h_0^\dag-&\tilde h\big)\big\Vert^2_2-\big\Vert I_0\mc{T} \big(h_0^\dag-h\big)\big\Vert^2_2\Big\}+\lambda_{\mc{H}}\Big(\big\Vert \tilde h\big\Vert_{\mc{H}}^2-\big\Vert h\big\Vert_{\mc{H}}^2\Big)\\
     &\quad\quad\quad\le \frac{1}{2c^2}\big\Vert I_0\mc{T} \big(h_0^\dag-h\big)\big\Vert^2_2+c_0\delta_n\big\Vert  I_0\mc{T}\big(\tilde h-h\big)\big\Vert_2+\lambda_{\mc{G}}B_1^2+c_0'\delta^2_n,
\end{align*}
where $\tilde h=\arg\min_{h\in\mc{H}_{B}}\max_{g\in\mc{G}}\{\mathbb{E}_n[\{m(O;g) -I_0(X)r(O)h(W)g(Z)-c^2I_0(X)g(Z)^2\}]-\lambda_{\mc{G}}\Vert g\Vert_{\mc{G}}^2+\lambda_{\mc{H}}\Vert h\Vert_{\mc{H}}^2\}$, $c_0:=3\cdot\max\{3,B_2+2\}\cdot \max\{1\:\Vert r\Vert_\infty\cdot k_1B,\: 2c^2b\}/(c^2 b)$, $b=B_1\max\{k_2,B_3\}$, with $B_2$ and $B_3$ the constants in Assumption~\ref{ass:bounded}, and $c_1$ and $c'_0$ are defined as in Lemma~\ref{lemma:erm}.\label{lemma:erm_mg}
\end{lemma}

\begin{proof}
The proof proceeds similarly to that of Lemma~\ref{lemma:erm}, with minor changes to handle the terms  $\E_{P_0}\{m(O;g)\}$ and $\E_n\{m(O;g)\}$. Here, $\xi(\delta_n)$ is as defined in~\ref{lemma:erm}, and the constants $c_1$ and $c_2$ are as in Lemma~\ref{lemma:erm}, except that $L'$ and $b$ are now $L' := \max\left\{1,\; \|r\|_\infty \cdot k_1 B,\; 2c^2 \cdot b \right\}$ and $ b:=B_1\max\{k_2,B_3\}$,
with $B_3$ being the constant specified in Assumption~\ref{ass:bounded}.

To use Lemma~\ref{lemma:lemma14foster} to bound $\E_{P_0}\{m(O;g_{\tilde h})\}$, define $\mc{F}:=\{f(\cdot) :=m(\cdot;g):\mc{O}\mapsto\mathbb{R}\:|\: g\in\mc{G}_{B_1}\}$, $t_f:=f$, and $\mc{O}_0:=\mc{O}\text{ (equivalently $I_{\mc{O}_0}=1$)}$. We verify conditions (a)--(f) listed in the proof of Lemma~\ref{lemma:erm}, which are required for Lemma~\ref{lemma:lemma14foster}. Condition (a) follows immediately from part 2 of Assumption~\ref{ass:bounded}.  Condition (a.3) holds because $g=0$ belongs to $\mc{G}_{B_1}$ and therefore $f_0:=0=m(\cdot,0)$ belongs to $\mc{F}$ by the linearity of $m$. Next, for any $f \in \mc{F}$, $f=m(\:\cdot\:,g_f)$ for some $g\in\mc{G}_{B_1}$. Then, 
    \begin{align*}
    |f(o)| = |f(o)-f_0(o)|\le B_3\cdot d(f,f_0)= B_3\cdot d(m(o;g_f),m(o,0))
    \le B_3\Vert g_f-0\Vert_\mc{G}\le B_3B_1,
    \end{align*} 
    where the first and second inequalities follow by part 2 of Assumption \ref{ass:bounded}, and the third holds because $g_f \in \mc{G}_{B_1}$. Thus, $\mc{F}$  is uniformly bounded by $B_3B_1\le b$, and so it satisfies condition (a.1). Condition (a.2) follows from part 2 of Assumption~\ref{ass:bounded}, which implies for all $o\in\mc{O}$ and $f_1,f_2\in\mc{F}$ that $ |f_1(o)-f_2(o)|\le\sup_{o\in\mc{O}}|f_1(o)-f_2(o)|\le B_3\cdot d(f_1,f_2)$. Condition (b) holds by the assumption on $m$ specified in Section \ref{sec:targer_param}. Condition (c) holds because $t_{I_{\mc{O}_0}\cdot f}=f$ and therefore (c) is the same as (a.2) that we have already argued to hold. Condition (d) holds because the last display implies $\E_{P_0}\{|f_1(O)-f_2(O)|\}\le\sup_{o\in\mc{O}}|f_1(o)-f_2(o)|\le  B_3\cdot d(f_1,f_2)$. Condition (e) trivially holds because $t_f=f$. Likewise, (f) trivially holds with $L=1\le L'$.

 By Lemma~\ref{lemma:lemma14foster}, $\E_{P_0}\{m(O;g_{\tilde h})\}\le \E_n\{m(O;g_{\tilde h})\}+ c_2\delta_n\{\Vert m(O;g_{\tilde h})\Vert_2+\delta_n\}$
happens with probability at least $1-\xi(\delta_n)$. Also, $\Vert m(O;g_{\tilde h})\Vert_2\le B_2\Vert I_0\cdot\tilde  g\Vert_2$ by part 1 of Assumption~\ref{ass:bounded}. Then, $\E_{P_0}\left\{m(O;g_{\tilde h})\right\}\le \E_n\left\{m(O;g_{\tilde h})\right\}+ B_2c_2\cdot\delta_n\left\Vert I_0\cdot g_{\tilde h}\right\Vert_2+c_2\delta^2_n$ holds with probability at least $1-\xi(\delta_n)$. Combining the latter with the events in (\ref{eq:term2sup}) and (\ref{eq:term3sup}) from the proof of Lemma~\ref{lemma:erm}, each of which holds with probability at least $1-\xi(\delta_n)$, we conclude that, with probability at least $1-3\xi(\delta_n)$,
\begin{align*}
      \frac{1}{4c^2}\big\Vert I_0\mc{T}\left(h_0^\dag-\tilde h\right)\big\Vert^2_2&\le\E_n\big\{I_0(X)s(O)g_{\tilde h}(Z)-I_0(X)r(O)\tilde h(W)g_{\tilde h}(Z)-c^2\cdot I_0(X)g_{\tilde h}(Z)^2\big\}\\
    &\quad +c_2\delta_n\Vert I_0\cdot g_{\tilde h}\Vert_2 +c_2\delta_n^2,
\end{align*}
where $c_3:=\max\{B_2+2,3\}\cdot c_2$. The remainder of the proof then follows exactly as in Lemma~\ref{lemma:erm}, with this updated constant $c_3$ replacing the value $c_3=3c_2=18L'/b$ used there.
\end{proof}

\end{document}